%% file: intreevalLattices.tex
\documentclass{amsart}

\usepackage{enumerate, amsmath, amsfonts, amssymb, amsthm, thmtools, wasysym, graphics, graphicx, url, hyperref, hypcap, a4wide, pdflscape, multido, xargs, colortbl, multicol, multirow, calc, shuffle, hvfloat}
\hypersetup{colorlinks=true, citecolor=darkblue, linkcolor=darkblue, hypertexnames=false}
\usepackage{tikz}
\usetikzlibrary{trees, decorations, decorations.markings, shapes, arrows, matrix, calc, fit, intersections, patterns, angles, cd}
\usepackage{tikz-qtree}
\usepackage{comment}
\usepackage{etex}
\usepackage{ulem}
\usepackage[noabbrev,capitalise]{cleveref}
\usepackage[export]{adjustbox}
\usepackage{xcolor}

\graphicspath{{figures/}}
\makeatletter
\def\input@path{{figures/}}
\makeatother

\title[Ornamentation lattices and intreeval hypergraphic lattices]{Ornamentation lattices and \\ intreeval hypergraphic lattices}

\author{Antoine Abram}
\address[A.~Abram]{LACIM, Université du Québec à Montréal, Montréal}
\email{abram.antoine@courrier.uqam.ca}
\urladdr{https://antoineabram.codeberg.page}

\author{Jose Bastidas}
\address[J.~Bastidas]{LACIM, Université du Québec à Montréal, Montréal}
\email{bastidas.math@proton.me}
\urladdr{https://bastidas-jose.codeberg.page}

\author{F\'elix G\'elinas}
\address[F.~G\'elinas]{Department of Mathematics and Statistics, York University, Toronto}
\email{felixgel@yorku.ca}
\urladdr{https://felixgelinas.github.io/}

\author{Vincent Pilaud}
\address[V.~Pilaud]{Universitat de Barcelona \& Centre de Recerca Matemàtica, Barcelona}
\email{vincent.pilaud@ub.edu}
\urladdr{https://www.ub.edu/comb/vincentpilaud/}

\author{Andrew Sack}
\address[A.~Sack]{Department of Mathematics, University of Michigan, Ann Arbor}
\email{asack@umich.edu}
\urladdr{https://andrewsack.com/}

\thanks{
VP was partially supported by the Spanish project PID2022-137283NB-C21 of MCIN/AEI/10.13039/501100011033 / FEDER, UE, by the Spanish--German project COMPOTE (AEI PCI2024-155081-2 \& DFG 541393733), by the Severo Ochoa and María de Maeztu Program for Centers and Units of Excellence in R\&D (CEX2020-001084-M), by the Departament de Recerca i Universitats de la Generalitat de Catalunya (2021 SGR 00697), and by the French--Austrian project PAGCAP (ANR-21-CE48-0020 \& FWF I 5788).
}

\newtheorem{theorem}{Theorem}[section]

\crefname{theoremA}{Theorem}{Theorems}
\newtheorem{corollary}[theorem]{Corollary}
\newtheorem{proposition}[theorem]{Proposition}
\newtheorem{lemma}[theorem]{Lemma}
\newtheorem{conjecture}[theorem]{Conjecture}
\crefname{conjecture}{Conjecture}{Conjectures}

\crefname{conjectureA}{Conjecture}{Conjectures}
\crefname{conjectureA}{Conjecture}{Conjectures}

\theoremstyle{definition}
\newtheorem{definition}[theorem]{Definition}
\newtheorem{example}[theorem]{Example}
\newtheorem{remark}[theorem]{Remark}

\crefname{notation}{Notation}{Notations}

\newcommand{\R}{\mathbb{R}} % reals
\newcommand{\N}{\mathbb{N}} % naturals
\renewcommand{\b}[1]{\boldsymbol{#1}} % bold
\renewcommand{\c}[1]{\mathcal{#1}} % bold
\newcommand{\set}[2]{\left\{ #1 \;\middle|\; #2 \right\}} % set notation
\newcommand{\ssm}{\smallsetminus} % small set minus
\newcommand{\eqdef}{\mbox{\,\raisebox{0.2ex}{\scriptsize\ensuremath{\mathrm:}}\ensuremath{=}\,}} % :=
\newcommand{\defeq}{\mbox{~\ensuremath{=}\raisebox{0.2ex}{\scriptsize\ensuremath{\mathrm:}} }} % =:
\newcommand{\simplex}{\triangle} % simplex
\newcommand{\surjection}{\longrightarrow\mathrel{\mkern-22mu}\rightarrow\,}     
\newcommand{\bijection}{\longleftrightarrow}     

\DeclareMathOperator{\conv}{conv} % convex hull
\DeclareMathOperator{\inv}{inv} % inversion set
\DeclareMathOperator{\tc}{tc} % transitive closure

\newcommand{\ie}{\textit{i.e.}~} % id est
\newcommand{\eg}{\textit{e.g.}~} % exempli gratia
\newcommand{\para}[1]{\bigskip\noindent\textbf{#1}} % paragraph
\definecolor{blue}{RGB}{0,0,255} % blue color
\definecolor{red}{RGB}{255,0,0} % red color
\definecolor{green}{RGB}{57,181,74} % green color
\definecolor{orange}{RGB}{247,147,30} % orange color
\definecolor{purple}{RGB}{147,39,143} % purple color
\definecolor{darkblue}{rgb}{0,0,0.7} % darkblue color
\newcommand{\blue}[1]{{\color{blue} #1}} % blue command
\newcommand{\red}[1]{{\color{red} #1}} % red command
\newcommand{\green}[1]{{\color{green} #1}} % green command
\newcommand{\orange}[1]{{\color{orange} #1}} % orange command
\newcommand{\darkblue}{\color{darkblue}} % darkblue command
\newcommand{\defn}[1]{\textsl{\darkblue #1}} % emphasis of a definition
\newcommand{\OEIS}[1]{{\rm \href{http://oeis.org/#1}{\texttt{#1}}}}
\usepackage{todonotes}

\newcommand{\fS}{\mathfrak{S}} % symmetric group
\newcommand{\meet}{\wedge} % meet
\newcommand{\join}{\vee} % join
\newcommand{\bigMeet}{\bigwedge} % meet
\newcommand{\bigJoin}{\bigvee} % join
\newcommand{\projDown}[1]{\smash{\pi^\downarrow_{#1}}} % Down projection
\newcommand{\projUp}[1]{\smash{\pi^\uparrow_{#1}}} % Down projection
\newcommandx{\JI}[1][1=L]{\mathcal{JI}(#1)} % join irreducibles
\newcommandx{\MI}[1][1=L]{\mathcal{MI}(#1)} % meet irreducibles
\newcommand{\CJR}{\mathbf{cjr}} % canonical join representation
\newcommand{\CMR}{\mathbf{cmr}} % canonical meet representation
\newcommand{\lessin}[2]{#1_{\le#2}} % less in #1

\newcommand{\mymap}[2]{\mathsf{#1}_{\hspace{-.7pt}#2}}
\DeclareMathOperator{\Orn}{\c{O}}  % ornamentations
\newcommand{\orn}[1]{\mymap{O}{#1}}  % ornamentation
\newcommand{\minorn}[2]{\mymap{O}{#1}^{\downarrow#2}}  % ornamentation
\newcommand{\maxorn}[2]{\mymap{O}{#1}^{\uparrow#2}}  % ornamentation
\DeclareMathOperator{\AOrn}{\c{AO}}  % acyclic ornamentations
\newcommand{\aorn}[1]{\mymap{AO}{#1}}  % acyclic ornamentation
\DeclareMathOperator{\Reori}{\c{R}}  % reorientations
\newcommand{\reori}[1]{\mymap{R}{#1}}  % reorientation
\newcommand{\maxreori}[1]{\smash{\mymap{R}{#1}^{\uparrow}}}  % maximal reorientation
\DeclareMathOperator{\AReori}{\c{AR}}  % acyclic reorientations
\newcommand{\areori}[1]{\mymap{AR}{#1}}  % acyclic reorientation
\DeclareMathOperator{\Rcl}{\c{R}^{cl}}  % transitively closed reorientations
\DeclareMathOperator{\Rco}{\c{R}^{co}}  % transitively coclosed reorientations
\DeclareMathOperator{\Rbi}{\c{R}^{bi}}  % transitively biclosed reorientations
\DeclareMathOperator{\rev}{rev} % reversed arrows
\DeclareMathOperator{\Sour}{\mathcal{S}}  % sourcings
\newcommand{\sour}[1]{\mymap{S}{#1}}  % sourcing
\DeclareMathOperator{\ASour}{\mathcal{AS}}  % acyclic sourcings
\newcommand{\asour}[1]{\mymap{AS}{#1}}  % acyclic sourcing
\DeclareMathOperator{\arr}{arr} % arrows

\newcommand{\HH}{\mathbb H}  % general hypergraph
\newcommand{\II}{\mathbb I} % interval hypergraph
\newcommand{\PP}{\mathbb P} % path hypergraph

\newcommand{\Igraph}{\sf I} % I graph
\newcommand{\Ngraph}{\sf N} % N graph
\newcommand{\Xgraph}{\sf X} % X graph
\newcommand{\Tgraph}{\sf I\!X\!I} % tie graph
\newcommand{\Ygraph}{\sf Y} % Y graph
\newcommand{\Dgraph}{\boldsymbol{\Diamond}} % D graph
\newcommand{\Kgraph}{\rotatebox[origin=c]{90}{$\boldsymbol{\triangle}$}} % K graph
\newcommand{\Agraph}{\rotatebox[origin=c]{180}{\sf Y}} % A graph

\begin{document}

\begin{abstract}
Given a directed graph~$D$ with transitive closure~$\tc(D)$ and path hypergraph~$\PP(D)$, we study the connections between the (acyclic) reorientation poset of~$\tc(D)$, the (acyclic) sourcing poset of~$\PP(D)$, and the (acyclic) ornamentation poset of~$D$.
Geometrically, the acyclic reorientation poset of~$\tc(D)$ (resp.~the acyclic sourcing poset of~$\PP(D)$) is the transitive closure of the skeleton of the graphical zonotope of~$\tc(D)$ (resp.~of the hypergraphic polytope of~$\PP(D)$) oriented in a linear direction.
When~$D$ is a rooted (or even unstarred) increasing tree, we show that the acyclic sourcing poset of~$\PP(D)$ is isomorphic to the ornamentation lattice of~$D$, and that they form a lattice quotient of the acyclic reorientation lattice of~$\tc(D)$.
As a consequence, we obtain polytopal realizations of the ornamentation lattices of rooted (or even unstarred) increasing trees, answering an open question of C.~Defant and A.~Sack.
When~$D$ is an increasing tree, we show that the ornamentation lattice of~$D$ is the MacNeille completion of the acyclic sourcing poset of~$\PP(D)$.
Finally, still when~$D$ is an increasing tree, we use the ornamentation lattice of~$D$ to characterize the subhypergraphs of the path hypergraph~$\PP(D)$ whose acyclic sourcing poset is a lattice.
\end{abstract}

\vspace*{-1cm}
\maketitle

\tableofcontents
\vspace*{-1cm}

%%%%%%%%%%%%%%%%%%%%%%%%%%%%%%%%%%%%%%
%%%%%%%%%%%%%%%%%%%%%%%%%%%%%%%%%%%%%%
%%%%%%%%%%%%%%%%%%%%%%%%%%%%%%%%%%%%%%

\section{Introduction}
\label{sec:introduction}

%%%%%%%%%%%%%%%%%%%%%%%%%%%%%%%%%%%%%%

Tamari lattices and associahedra are classical objects of algebraic and geometric combinatorics with various applications to topology, algebra, statistical physics, etc.
We refer to~\cite{TamariFestschrift, CeballosSantosZiegler, PilaudSantosZiegler} and the references therein for an overview of the myriad of  perspectives on these objects, and we focus here on three specific aspects.
First, the Hasse diagram of the Tamari lattice is isomorphic to a linear orientation of the graph of the associahedron of~\cite{ShniderSternberg, Loday}, which is a Minkowski sum of standard simplices~\cite{Postnikov}. 
Second, the Tamari lattice (on binary trees) is a lattice quotient of the weak order (on permutations)~\cite{Tonks, Reading-latticeCongruences}, and the associahedron is obtained by deleting inequalities in the facet description of the permutahedron~\cite{ShniderSternberg, Loday}.
Third, the Tamari lattice is realized by componentwise comparisons of the bracket vectors of binary trees~\cite{HuangTamari}.
The objective of this paper is to study the interconnections between natural extensions of these three perspectives, already considered separately in the literature.

%%%%%%%%%%%%%%%%%%%%%%%%%%%%%%%%%%%%%%

\para{Hypergraphic posets and hypergraphic polytopes.}
Consider a hypergraph~$\HH$ on~$[n]$.
A \defn{sourcing} of~$\HH$ is a map~$S$ associating to each hyperedge~$H \in \HH$ a source~$S(H) \in H$.
Note that sourcings of~$\HH$ are often called orientations of~$\HH$, we changed the name to avoid confusion with reorientations of directed graphs also manipulated in this paper.
A sourcing~$S$ of~$\HH$ is \defn{acyclic} if there is no directed cycle in the directed graph with edges~$(h,S(H))$ for all~$h \in H \in \HH$.
The \defn{(acyclic) sourcing poset} of~$\HH$ is the poset of (acyclic) sourcings of~$\HH$ ordered by componentwise comparison, meaning~$S \le S'$ if and only if~$S(H) \le S'(H)$ for all~$H \in \HH$.
While the sourcing poset~$\Sour(\HH)$ is clearly a lattice (as a product of intervals), the acyclic sourcing poset~$\ASour(\HH)$ is not always a lattice.
The question to characterize the hypergraphs~$\HH$ whose acyclic sourcing posets~$\ASour(\HH)$ are lattices has been studied for specific families of hypergraphs, notably for graphs~\cite{Pilaud-acyclicReorientationLattices}, for interval hypergraphs~\cite{BergeronPilaud}, and (partially for) graph associahedra~\cite{BarnardMcConville}.
In this paper, we will consider the family of hypergraphs arising as the collection of paths in a directed graph, and study their acyclic sourcing posets from a lattice theoretic perspective.

Let~$(\b{e}_v)_{v \in [n]}$ denote the standard basis of~$\R^n$.
The \defn{hypergraphic polytope} of~$\HH$ is the Minkowski sum~$\simplex_\HH \eqdef \sum_{H \in \HH} \simplex_H$, where~$\simplex_H \eqdef \conv\set{\b{e}_h}{h \in H}$ denotes the face of the standard simplex indexed by the hyperedge~$H$.
The face lattice of~$\simplex_\HH$ was described combinatorially in terms of acyclic orientations of~$\HH$ in~\cite{BenedettiBergeronMachacek}.
In particular, it is proved in~\cite{Gelinas} that the acyclic sourcing poset~$\ASour(\HH)$ is the transitive closure of the graph of the hypergraphic polytope~$\simplex_\HH$, oriented in the standard direction~$\b{\omega} \eqdef (n-1, n-3, \dots, 3-n, 1-n)$.

For instance, 
\begin{itemize}
\item if~$\HH = \binom{[n]}{2}$ is the complete graph, $\simplex_\HH$ is the permutahedron and $\ASour(\HH)$ is the weak order.
\item if~$\HH = \set{[i,j]}{1 \le i \le j \le n}$ is the complete interval hypergraph, $\simplex_\HH$ is the associahedron of~\cite{ShniderSternberg,Loday} and~$\ASour(\HH)$ is the Tamari lattice~\cite{Tamari}.
\end{itemize}

%%%%%%%%%%%%%%%%%%%%%%%%%%%%%%%%%%%%%%

\para{Digraph Tamari lattices and digraph associahedra.}
Consider a directed graph~$E$ on~$[n]$.
A \defn{reorientation} of~$E$ is a directed graph~$R$ obtained from~$E$ by reversing a subset~$\rev(R)$ of edges of~$E$.
A reorientation~$R$ of~$E$ is \defn{acyclic} if it contains no directed cycle.
The \defn{(acyclic) reorientation poset} of~$E$ is the set of (acyclic) reorientations of~$E$ ordered by inclusion of reversion sets, meaning~$R \le R'$ if and only if~$\rev(R) \subseteq \rev(R')$.
While the reorientation poset~$\Reori(E)$ is clearly a boolean lattice, the acyclic reorientation poset~$\AReori(E)$ is not always a lattice.
See~\cite{Pilaud-acyclicReorientationLattices} for a characterization of the directed graphs~$E$ for which~$\Reori(E)$ is a lattice (or even a semidistributive lattice).

The \defn{graphical zonotope} of~$E$ is the Minkowski sum~$\simplex_E \eqdef \sum_{(u,v) \in E} \simplex_{uv}$, where~$\simplex_{uv}$ is the segment with endpoints~$\b{e}_u$ and~$\b{e}_v$.
The face lattice of~$\simplex_E$ is described by ordered partitions of~$E$.
In particular, whenever the edges of $E$ are increasing, it turns out that the acyclic reorientation poset~$\AReori(E)$ is isomorphic to the transitive closure of the graph of~$\simplex_E$, oriented in the direction~$\b{\omega}$.

When~$\AReori(E)$ is a lattice, one can consider its lattice congruences and quotients.
They have been largely investigated in~\cite{Pilaud-acyclicReorientationLattices}.
Among all these congruences, the \defn{sylvester congruence} yields the \defn{Tamari lattice} of~$E$, which is isomorphic to the transitive closure of the graph of the \defn{associahedron} of~$E$ oriented in the direction~$\b{w}$.
The later is obtained by deleting inequalities in the facet description of the graphical zonotope of~$E$.

For instance, when~$E = K_n$ is the complete graph,
\begin{itemize}
\item $\simplex_E$ is the permutahedron and~$\AReori(E)$ is the weak order,
\item the associahedron of~$E$ is the associahedron of~\cite{ShniderSternberg,Loday} and the Tamari lattice of~$E$ is the Tamari lattice.
\end{itemize}

%%%%%%%%%%%%%%%%%%%%%%%%%%%%%%%%%%%%%%

\para{Ornamentation lattices.}
Consider a directed graph~$D$ on~$V$ and a vertex~$v \in V$.
An \defn{ornament} of~$D$ at~$v$ is a subset~$U$ of~$V$ such that any vertex of~$U$ admits a directed path to~$v$ in~$U$.
An \defn{ornamentation} of~$D$ is an assignment~$O$ of an ornament at each vertex of~$V$ with the nesting condition that $u \in O(v)$ implies~$O(u) \subseteq O(v)$.
Ornamentations of rooted trees were introduced in~\cite{DefantSack} and further studied in~\cite{ajran2025pop}.
Ornamentations of directed graphs and beyond will also appear in~\cite{Sack}.
An ornamentation~$O$ of~$D$ is \defn{acyclic} if a certain reorientation of the transitive closure~$\tc(D)$ of~$D$ constructed from~$O$ is acyclic (see precise definition later).
The \defn{(acyclic) ornamentation poset} of~$D$ is the poset of (acyclic) ornamentations of~$D$ ordered by componentwise inclusion, meaning~$O \le O'$ if and only if~${O(v) \subseteq O'(v)}$ for all~$v \in V$.
The ornamentation poset~$\Orn(D)$ is always a lattice, but the acyclic ornamentation poset~$\Orn(D)$ is not always a lattice.
Interestingly, we will see that all ornamentations turn out to be acyclic when~$D$ is a rooted (or even unstarred, see below) tree, and the ornamentation lattice is the MacNeille completion of the acyclic ornamentation lattice when~$D$ is a directed tree.

The geometry of ornamentation lattices is not that clear.
One of the open questions in~\cite{DefantSack} was to find polytopal realizations of ornamentation lattices.
We will answer this question for rooted (or even unstarred) trees and discuss it for arbitrary directed graphs.

For instance, when~$D$ is a path, the (acyclic) ornamentation lattice~$\Orn(D)$ is the Tamari lattice which is indeed realized by the associahedron, and the ornamentations can be seen as bracket vectors of binary trees.

%%%%%%%%%%%%%%%%%%%%%%%%%%%%%%%%%%%%%%

\medskip
\para{Main results and overview.}
The general purpose of the present paper is to connect these three extensions of the Tamari lattice and associahedron, in particular for directed trees.

In \cref{sec:general}, we first present in more detail the reorientation lattice~$\Reori(E)$, the sourcing lattice~$\Sour(\HH)$, and the ornamentation lattice~$\Orn(D)$ mentioned above, and we connect them as \linebreak follows.

\begin{proposition}
Given a directed graph~$D$, denote by~$\tc(D)$ its transitive closure and by~$\PP(D)$ its path hypergraph, whose hyperedges are (the vertex sets of) the directed paths in~$D$.
There are natural order-preserving surjections
\[
\begin{array}{rcl}
	\Reori(\tc(D)) & \surjection & \Orn(D) \\
	R & \longmapsto & \orn{R}
\end{array}
\qquad\text{and}\qquad
\begin{array}{rcl}
	\Sour(\PP(D)) & \surjection & \Orn(D) \\
	S & \longmapsto & \orn{S}
\end{array}
.
\]
Moreover, $R \mapsto \orn{R}$ restricts to a meet semilattice morphism from the transitively closed reorientation lattice of~$\tc(D)$ to the ornamentation lattice~of~$D$.
\end{proposition}

In \cref{sec:acyclic}, we focus on acyclicity.
We consider the acyclic reorientation poset~$\AReori(E)$, the acyclic sourcing poset~$\ASour(\HH)$, and define the acyclic ornamentation poset~$\AOrn(D)$. We then connect them as follows.

\begin{proposition}
For a directed graph~$D$, there are natural order-preserving surjections
\[
\begin{array}{rcccccl}
	\fS_n & \surjection & \AReori(\tc(D)) & \surjection & \ASour(\PP(D)) & \bijection & \AOrn(D) \\
	\pi & \longmapsto & \areori{\pi} \qquad R & \longmapsto & \asour{R} \qquad S & \longmapsto & \aorn{S}
\end{array}
.
\]
In particular, the acyclic sourcing poset~$\ASour(\PP(D))$ can be seen as a subposet of the ornamentation lattice~$\Orn(D)$.
\end{proposition}

See \cref{fig:allMaps} for an overview of these connections.

\begin{figure}
	\centerline{\input{figures/diagram.tex}}
	\caption{Connections between the posets studied in this paper. Below each poset or map appears a pointer to the corresponding definition. Below each map, we also point to the main statements concerning it (in the general case of a directed graph). The maps in green are order preserving while those in red are not. The symbol~\smash{\Large$\circlearrowleft$} means a commuting diagram, the symbol~\smash{\raisebox{-.05cm}{\rotatebox{90}{{$\curvearrowright$}}}} means that one map is a section of the other. Dashed arrows indicate inclusions.}
	\label{fig:allMaps}
\end{figure}
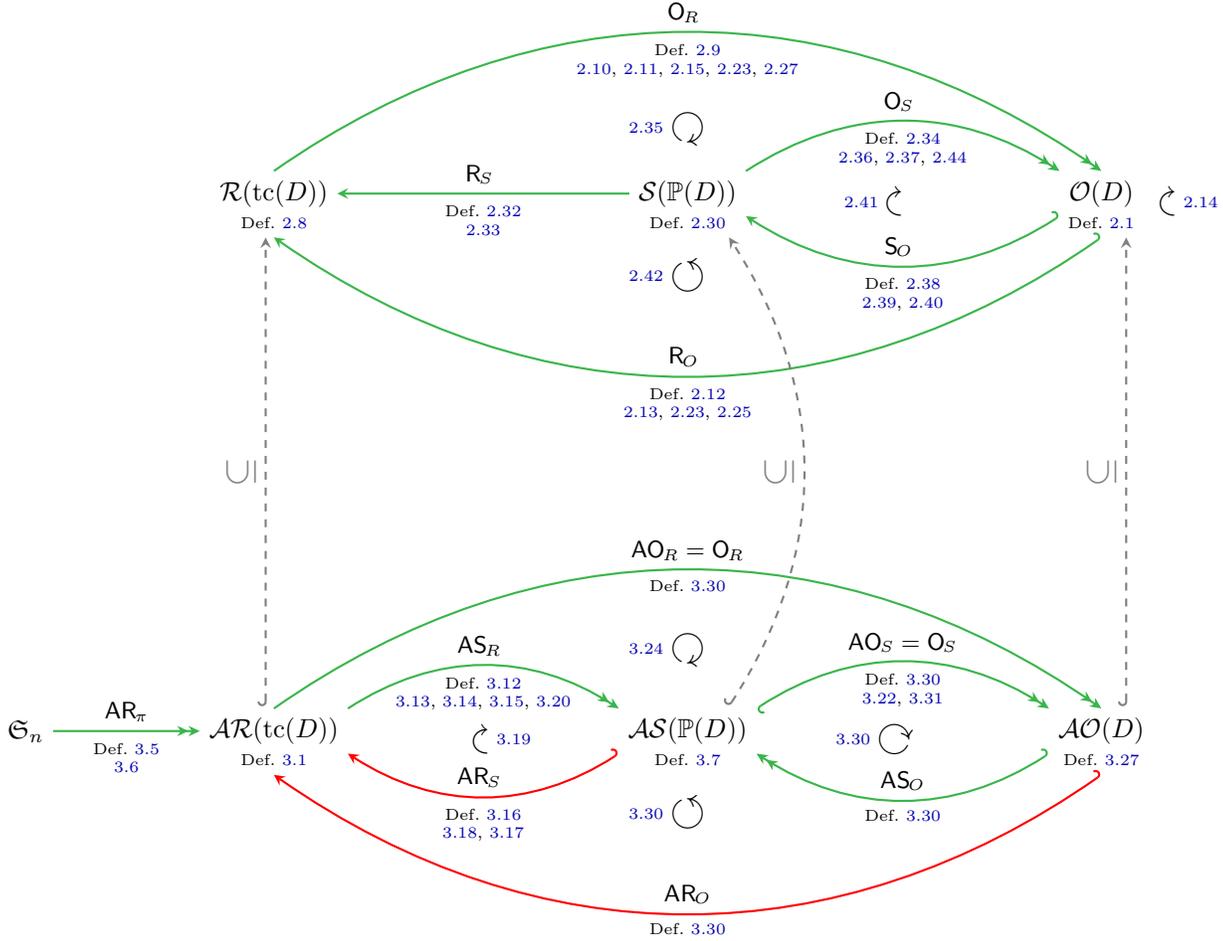

\medskip
These connections become particularly interesting when the underlying directed graph~$D$ is an increasing tree~$T$, and even more if it is rooted (or unstarred).
In \cref{sec:trees}, we first consider a directed tree.
We describe the join irreducible ornamentations of~$T$ and meet irreducible ornamentations of~$T$, and prove that the ornamentation lattice~$\Orn(T)$ is semidistributive.
Observing that all join or meet irreducible ornamentations of~$T$ (resp.~transitively biclosed reorientations of~$\tc(T)$) are acyclic, we obtain the following statement.
Recall that the MacNeille completion of a poset~$P$ is the smallest lattice which admits an embedding of~$P$.

\begin{theorem}
\label{thm:main1}
For any increasing tree~$T$ on~$[n]$, 
\begin{itemize}
\item the transitively biclosed reorientation lattice~$\Rbi(\tc(T))$ is precisely the MacNeille completion of the acyclic reorientation poset~$\AReori(\tc(T))$,
\item the ornamentation lattice~$\Orn(T)$ is precisely the MacNeille completion of the acyclic sourcing poset~$\ASour(\PP(T))$.
\end{itemize}
\end{theorem}

We then consider, in \cref{sec:unstarredTrees}, a special class of increasing trees containing, in particular, all rooted trees.
A tree~$T$ is \defn{unstarred} if there are no vertices~$u,v$ of~$T$ such that~$u$ has at least two incoming edges in~$T$, $v$ has at least two outgoing edges in~$T$, and $T$ has a directed path from~$u$ to~$v$.
For these trees, the relation between its acyclic sourcing poset and ornamentation lattice is even stronger.

\begin{theorem}
\label{thm:main2}
For an unstarred increasing tree~$T$,
\begin{enumerate}
\item the acyclic reorientation poset~$\AReori(\tc(T))$, the acyclic sourcing poset~$\ASour(\PP(T))$, and the acyclic ornamentation poset~$\AOrn(T)$ are all lattices,
\item all ornamentations of~$T$ are in fact acyclic, and so~$\ASour(\PP(T)) \simeq \AOrn(T) = \Orn(T)$,
\item the map~$R \mapsto \orn{R}$ is a surjective lattice map, and so the ornamentation lattice~$\Orn(T)$ is a lattice quotient of the acyclic reorientation lattice~$\AReori(\tc(T))$,
\item the ornamentation lattice~$\Orn(T)$ is isomorphic to the transitive closure of the graph of the path hypergraphic polytope~$\simplex_{\PP(T)}$ oriented in the direction~$\b{w}$. \label{item:polytopalRealization}
\end{enumerate}
\end{theorem}

Note that \cref{thm:main2}\,\eqref{item:polytopalRealization} answers a question posed in~\cite{DefantSack} for rooted trees.
This question remains open beyond the case of unstarred trees.
We state it as a conjecture.

\begin{conjecture}
For any directed graph~$D$, the ornamentation lattice~$\Orn(D)$ is isomorphic to the transitive closure of the graph of a polytope oriented in a linear direction.
\end{conjecture}

Finally, in \cref{sec:intreevalHypergraphicPosets}, we exploit the connection between sourcing posets and ornamentation lattices to extend our understanding of the hypergraphs~$\HH$ whose acyclic sourcing poset~$\ASour(\HH)$ is a lattice.
Our results largely extend the characterization of~\cite{BergeronPilaud} for interval hypergraphic~lattices.

An \defn{interval hypergraph} is a hypergraph whose hyperedges all are intervals of~$[n]$.
For an interval hypergraph~$\II$ containing all singletons, it was proved in~\cite[Thm.~A]{BergeronPilaud} that the acyclic sourcing poset~$\ASour(\II)$ is a lattice if and only if~$\II$ is closed under intersection, meaning that~$I,J \in \II$ implies~$I \cap J \in \II$.
(As adding or removing singletons to~$\II$ preserves~$\ASour(\II)$, we impose~$\II$ to contain all singletons just to simplify the expression of intersection closedness.)
The approach of~\cite{BergeronPilaud} is to show that the map~$\pi \mapsto \asour{\pi}$ sending a permutation of~$[n]$ to an acyclic sourcing of~$\II$ is a \defn{quasi-lattice map}, meaning that the fiber of each acyclic sourcing~$S$ of~$\II$ forms an interval~$[\projDown{S}, \projUp{S}]$ of the weak order and that~$S \le S'$ in~$\ASour(\II)$ implies that~$\projDown{S} \le \projUp{S'}$.
While much weaker than the lattice map condition (which requires instead that~$S \le S'$ implies~$\projDown{S} \le \projDown{S'}$ and~$\projUp{S} \le \projUp{S'}$), the quasi-lattice map condition still ensures that the image forms a lattice.

In \cref{sec:intreevalHypergraphicPosets}, we extend this characterization to intreeval hypergraphs.
An \defn{intreeval hypergraph} of an increasing tree~$T$ on~$[n]$ is any subhypergraph~$\II$ of the path hypergraph~$\PP(T)$, meaning any collection of (vertex sets of) directed paths in~$T$.
Unfortunately, even when $\ASour(\II)$ is a lattice, the map~$\pi \mapsto \asour{\pi}$ from permutations of~$[n]$ to acyclic sourcings of~$\II$ is not necessarily a quasi-lattice map anymore (its fibers are not intervals in general).
Moreover, it is also not possible to use the map~$R \mapsto \asour{R}$ from acyclic reorientations of~$\tc(T)$ to acyclic sourcings of~$\II$, as the acyclic reorientation poset~$\AReori(\tc(T))$ is not a lattice in general, in contrast to the situation of~\cref{thm:main2}.
We use instead a natural quasi-lattice map from the ornamentation lattice~$\Orn(T)$ to the acyclic sourcing poset~$\ASour(\II)$ to obtain the following complete characterization.

\begin{theorem}
\label{thm:main3}
Let~$\II$ be an intreeval hypergraph of an increasing tree~$T$.
The acyclic sourcing poset~$\ASour(\II)$ is a lattice if and only if $\II$ is path intersection closed (\cref{def:pathIntersectionClose}) and star sparse (\cref{def:pathStarSparse}).
\end{theorem}

Path intersection closedness and star sparsity are slightly technical conditions whose details are deferred to \cref{sec:intreevalHypergraphicPosets}.
When~$T$ is a path, path intersection closedness boils down to intersection closedness, and star sparsity always holds, hence \cref{thm:main3} generalizes~\cite[Thm.~A]{BergeronPilaud}.
Note that intersection closedness implies path intersection closedness (but is stronger in general, except when~$T$ is a path), and that star sparsity always holds for rooted trees.
In particular, the acyclic sourcing poset~$\ASour(\II)$ of any intersection closed intreeval hypergraph~$\II$ of a rooted~tree~is~a~lattice.

%%%%%%%%%%%%%%%%%%%%%%%%%%%%%%%%%%%%%%
%%%%%%%%%%%%%%%%%%%%%%%%%%%%%%%%%%%%%%
%%%%%%%%%%%%%%%%%%%%%%%%%%%%%%%%%%%%%%

\clearpage
\section{Reorientations, sourcings, and ornamentations}
\label{sec:general}

Throughout this section, we fix a directed graph~$D$ on a vertex set~$V$, and denote by~$\tc(D)$ its transitive closure and by~$\PP(D)$ its hypergraph of directed paths.
It is a priori not required that~$D$ is acyclic, although this is the case that interest us most in this paper.
In particular, in all our pictures, $V = [n]$ and the edges are increasing (so we can omit their orientations).
We define the ornamentation lattice~$\Orn(D)$ (\cref{subsec:ornamentations}), the reorientation lattice~$\Reori(\tc(D))$ (\cref{subsec:reorientations}), and the sourcing lattice~$\Sour(\PP(D))$ (\cref{subsec:sourcings}), and define natural surjective poset morphisms:
\[
\begin{array}{rcl}
	\Reori(\tc(D)) & \surjection & \Orn(D) \\
	R & \longmapsto & \orn{R}
\end{array}
\qquad\text{and}\qquad
\begin{array}{rcl}
	\Sour(\PP(D)) & \surjection & \Orn(D) \\
	S & \longmapsto & \orn{S}
\end{array}
.
\]

%%%%%%%%%%%%%%%%%%%%%%%%%%%%%%%%%%%%%%

\subsection{Ornamentations}
\label{subsec:ornamentations}

\begin{figure}[b]
	\centerline{\includegraphics[scale=.68]{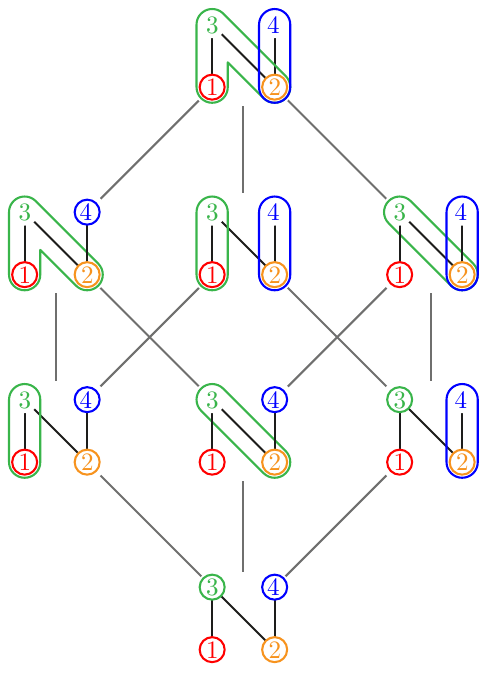} \qquad \raisebox{.5cm}{\includegraphics[scale=.68]{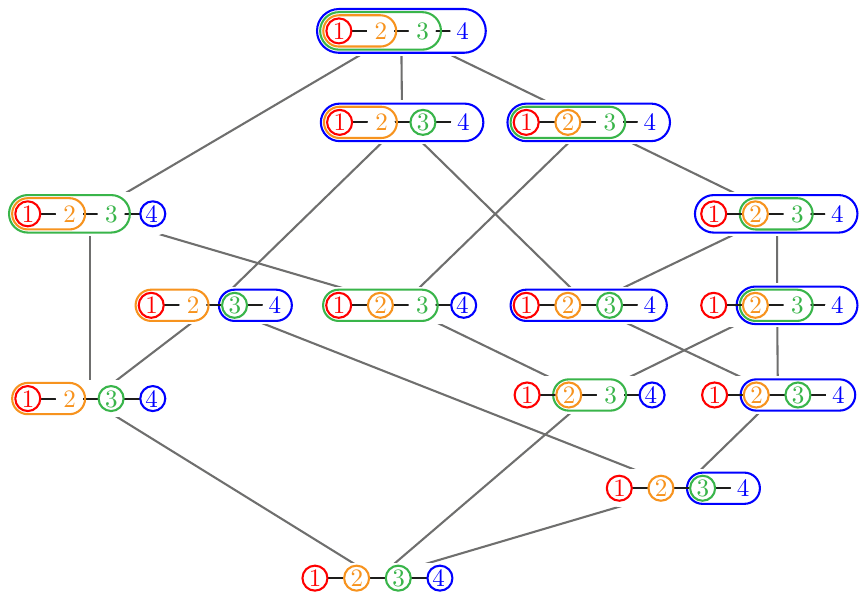}}}
	\caption{The ornamentation lattices~$\Orn(\Ngraph)$ and~$\Orn(\Igraph)$.}
	\label{fig:ornamentationsNI}
\end{figure}

\begin{figure}
	\centerline{\includegraphics[scale=.68]{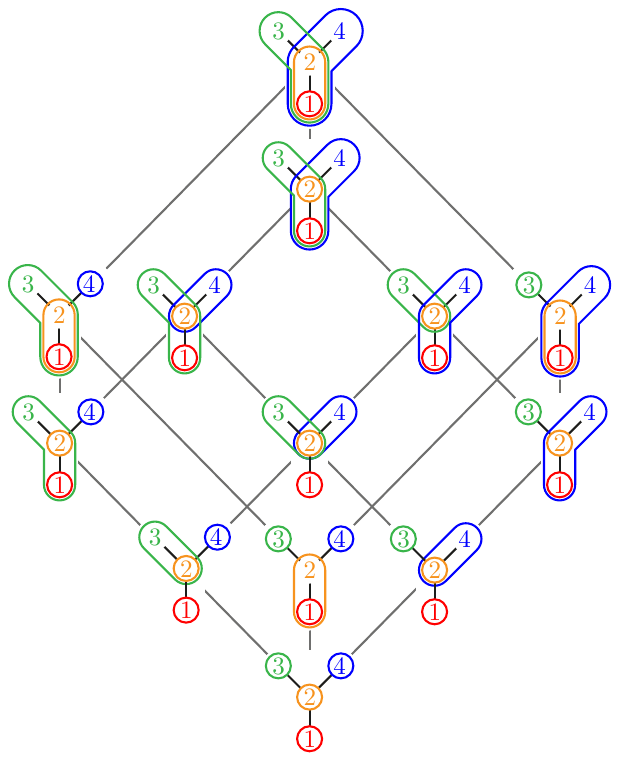} \qquad \includegraphics[scale=.68]{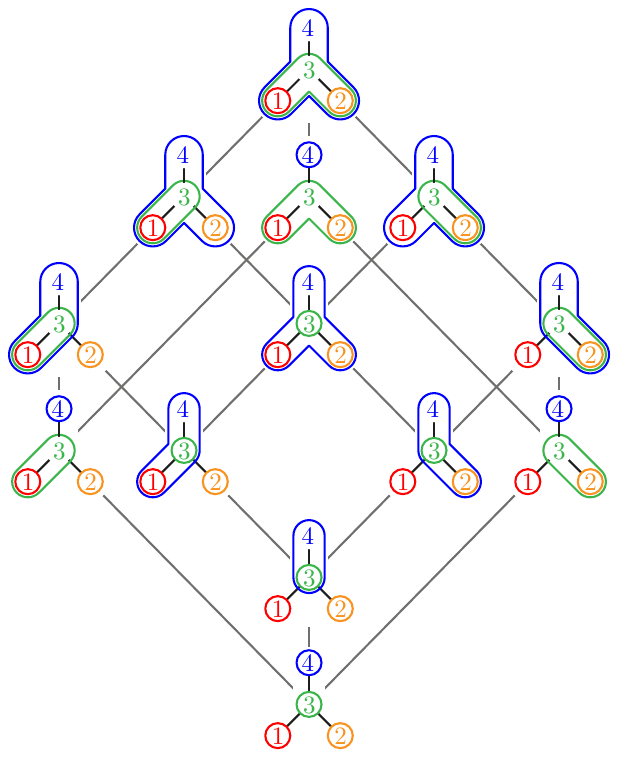}}
	\caption{The ornamentation lattices~$\Orn(\Ygraph)$ and~$\Orn(\Agraph)$.}
	\label{fig:ornamentationsAY}
\end{figure}

\begin{figure}
	\centerline{\includegraphics[scale=.68]{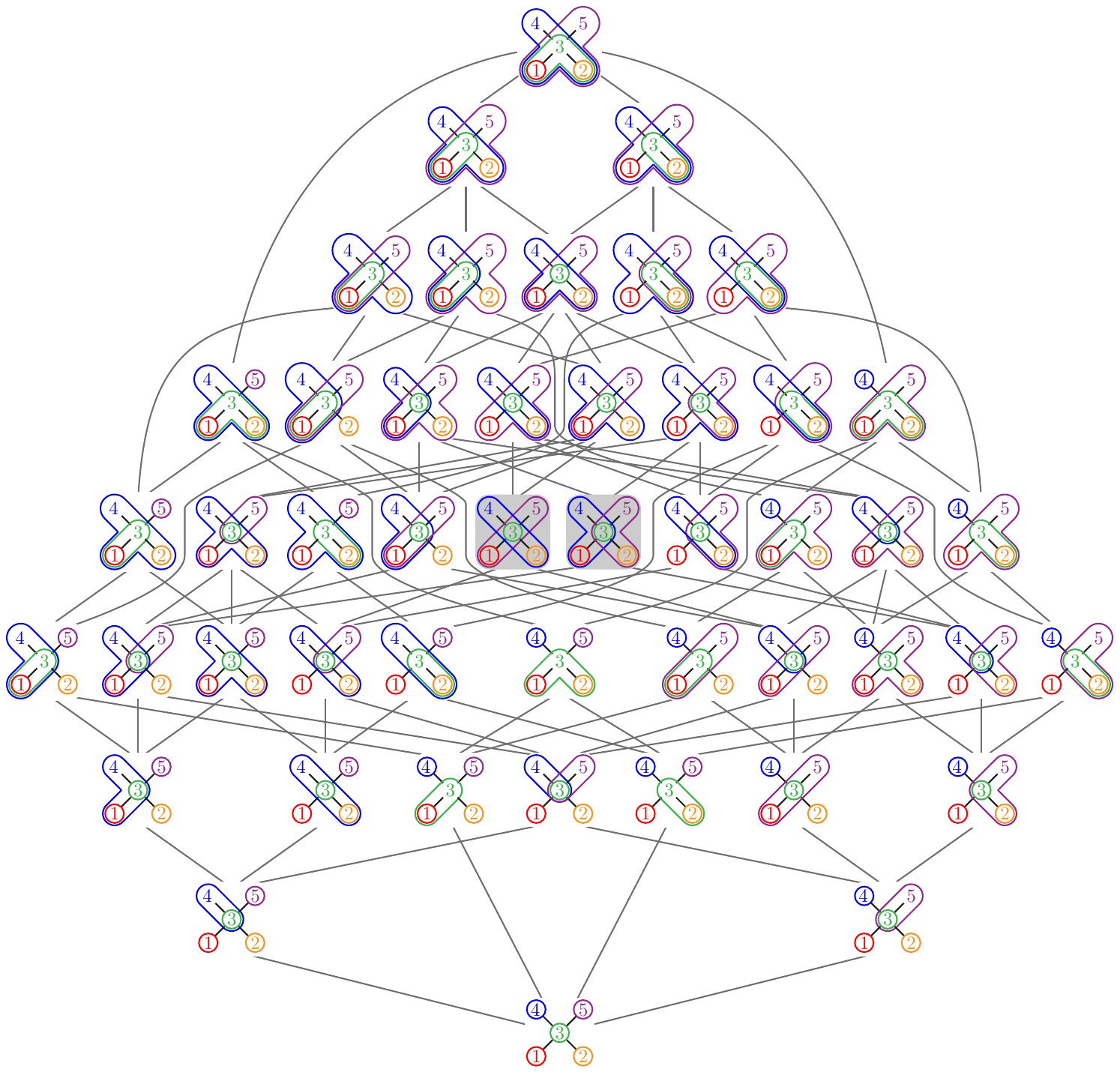}}
	\caption{The ornamentation lattice~$\Orn(\Xgraph)$. (Gray = cyclic).}
	\label{fig:ornamentationsX}
\end{figure}

\begin{figure}
	\centerline{\includegraphics[scale=.68]{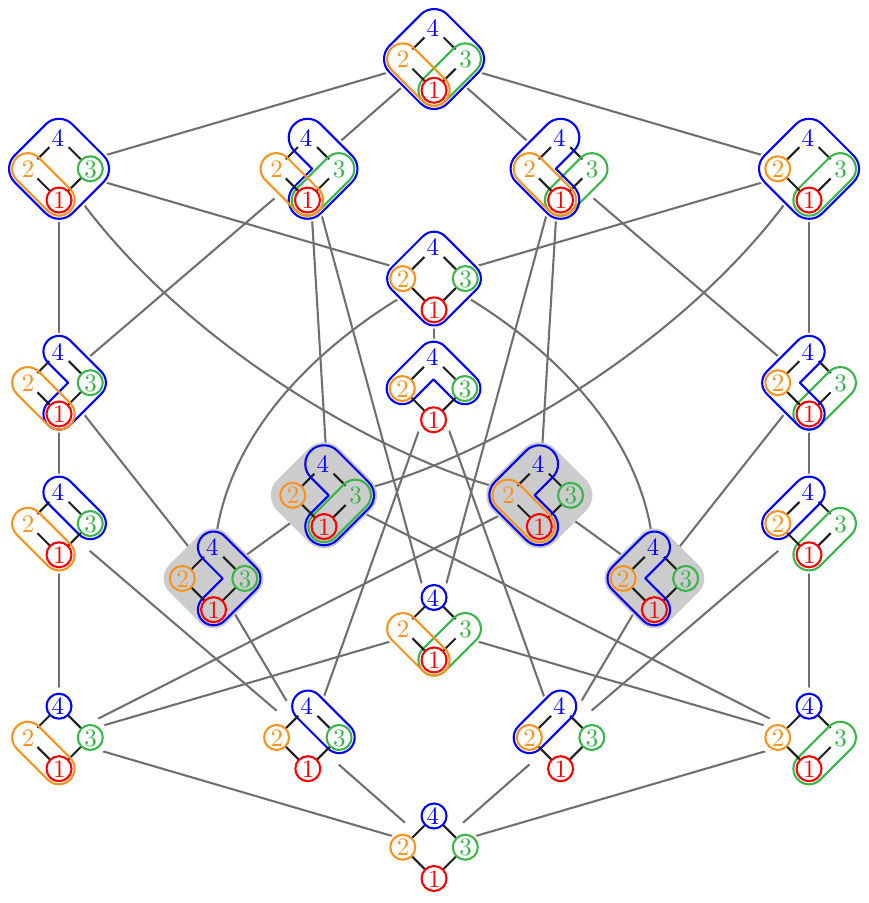}}
	\caption{The ornamentation lattice~$\Orn(\Dgraph)$. (Gray = cyclic).}
	\label{fig:ornamentationsD}
\end{figure}

\begin{figure}
	\centerline{\includegraphics[scale=.68,valign=c]{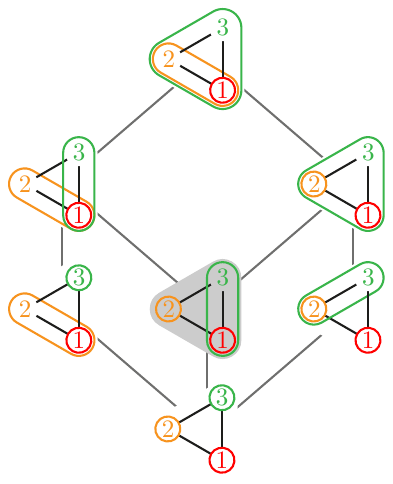}\qquad\includegraphics[scale=.68,valign=c]{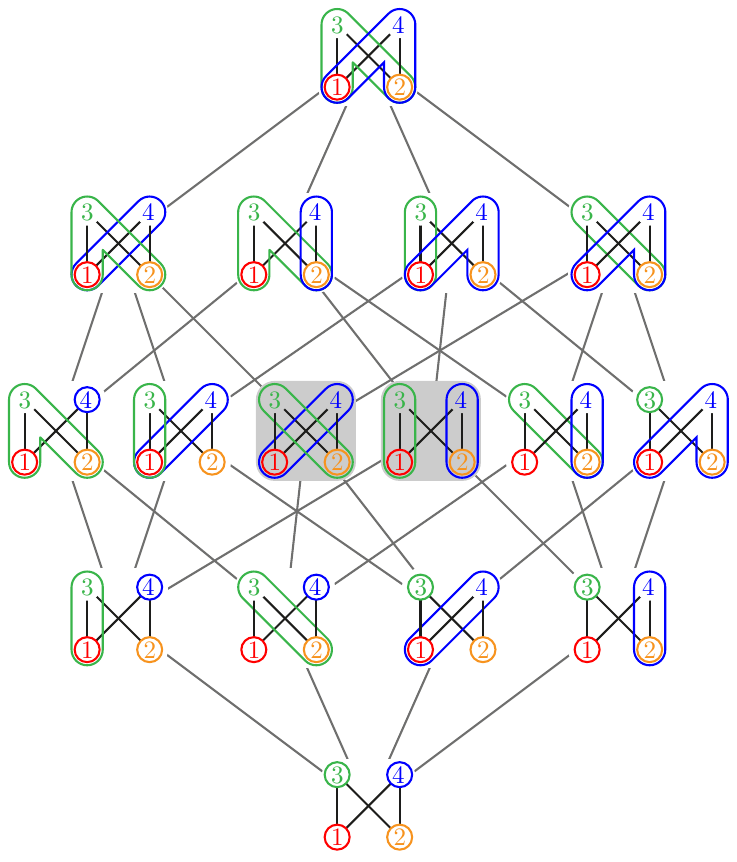}}
	\caption{The ornamentation lattices~$\Orn(\Kgraph)$ and~$\Orn(\Tgraph)$. (Gray = cyclic).}
	\label{fig:ornamentationsKT} 
\end{figure}

The following definition, first introduced for rooted trees in~\cite{DefantSack}, is illustrated in \cref{fig:ornamentationsNI,fig:ornamentationsAY,fig:ornamentationsX,fig:ornamentationsD,fig:ornamentationsKT}.
It will also be extended beyond directed graphs in~\cite{Sack}.

\begin{definition}
\label{def:Orn}
Let~$D$ be a directed graph on~$V$.
An \defn{ornament} at a vertex $v$ of~$D$ is a subset~$U \subseteq V$ such that there is a directed path from any $u \in U$ to~$v$ in the subgraph of~$D$ induced by~$U$.
An \defn{ornamentation} of~$D$ is a map $O$ on $V$ which assigns an ornament~$O(v)$ at each vertex~$v \in V$ such that~$u \in O(v) \implies O(u) \subseteq O(v)$ for all~$u,v \in V$.
We denote by~$\Orn(v \in D)$ the set of ornaments at a vertex~$v$ of~$D$ and by~$\Orn(D)$ the set of ornamentations of~$D$.
\end{definition}

In our pictures, we represent an ornamentation~$O$ of~$D$ by drawing a bubble around~$O(v)$ of the same color as~$v$ for each~$v \in V$.
Alternatively, each bubble is the ornament at its largest vertex.
See \cref{fig:ornamentationsNI,fig:ornamentationsAY,fig:ornamentationsX,fig:ornamentationsD,fig:ornamentationsKT}.

\begin{example}
\label{exm:ornamentations}
We can clearly define different ornamentations by
\begin{enumerate}[(i)]
\item sending each vertex~$v \in V$ to the singleton~$\{v\}$, 
\item sending each vertex~$v \in V$ to the set~$\lessin{D}{v}$ of vertices with a path to~$v$ in~$D$,
\item given a vertex~$v \in V$ and an ornament~$U \in \Orn(v \in D)$, sending~$v$ to~$U$ and any other vertex~${u \in V \ssm \{v\}}$ to the singleton~$\{u\}$.
\end{enumerate}
\end{example}

\begin{lemma}
\label{lem:unionOrnaments}
\begin{enumerate}[(i)]
\item If~$U \in \Orn(u \in D)$ and~$U' \in \Orn(u' \in D)$ with~$u \in U'$, then~${U \cup U' \in \Orn(u' \in D)}$.
\item If~$U,U' \in \Orn(v \in D)$, then~$U \cup U' \in \Orn(v \in D)$, while $U \cap U'$ is not necessarily~in~${\Orn(v \in D)}$.
\end{enumerate}
\end{lemma}

\begin{proof}
For Part~(i), let~$w \in U \cup U'$. If~$w \in U'$, then there is a path from~$w$ to~$u'$ in~$U'$, hence in~$U \cup U'$.
If~$w \in U$, then there is a path from~$w$ to~$u$ in~$U$ and a path from~$u$ to~$u'$ in~$U'$, hence a path from~$w$ to~$u'$ in~$U \cup U'$.
Part~(ii) is a specialization of Part~(i) to~$u = v$ and~$u' = v$.
Finally, if~$D$ has two disjoint directed paths~$P_1$ and~$P_2$ from a vertex~$u$ to a vertex~$v$ such that~$(u,v) \notin D$, then~$P_1, P_2 \in \Orn(v \in D)$ while~$P_1 \cap P_2 = \{u,v\} \notin \Orn(v \in D)$.
See \cref{fig:ornamentationsD,lem:meetOrnT}.
\end{proof}

\begin{theorem}
\label{thm:OrnMeetJoin}
The set~$\Orn(D)$ of ornamentations of~$D$ is a lattice under componentwise inclusion, meaning~$O_1 \le O_2$ if and only if~$O_1(v) \subseteq O_2(v)$ for all~$v \in V$.
For any two ornamentations~$O_1$ and~$O_2$ of~$D$,
\begin{itemize}
\item $(O_1 \meet O_2)(v)$ is the inclusion maximal ornament at~$v$ contained in~$O_1(v) \cap O_2(v)$,
\item $(O_1 \join O_2)(v)$ is the inclusion minimal subset~$U$ of~$V$ containing~$v$ and such that~$u \in U$ implies~$O_1(u) \cup O_2(u) \subseteq U$.
\end{itemize}
\end{theorem}

\begin{proof}
Given two ornamentations~$O_1$ and~$O_2$ of~$D$, consider the map~$O_\meet$ on~$V$ where $O_\meet(v)$ is the inclusion maximal ornament at~$v$ contained in~$O_1(v) \cap O_2(v)$.
Note that $O_\meet(v)$ is well-defined by \cref{lem:unionOrnaments}, and that~$O_\meet(v) \in \Orn(v \in D)$ for any~$v \in V$ by definition.
Consider now~$u,v \in V$ such that~$u \in O_\meet(v)$.
Since~$O_\meet(u) \in \Orn(u \in D)$ and~$O_\meet(v) \in \Orn(v \in D)$ with~${u \in O_\meet(v)}$, \cref{lem:unionOrnaments} gives~$O_\meet(u) \cup O_\meet(v) \in \Orn(v \in D)$.
As~$u \in O_\meet(v) \subseteq O_1(v) \cap O_2(v)$ and~$O_1$ and~$O_2$ are ornamentations of~$D$, we have~$O_1(u) \subseteq O_1(v)$ and~$O_2(u) \subseteq O_2(v)$, so~${O_\meet(u) \subseteq O_1(u) \cap O_2(u) \subseteq O_1(v) \cap O_2(v)}$, hence~$O_\meet(u) \cup O_\meet(v) \subseteq O_1(v) \cap O_2(v)$.
By maximality of~$O_\meet(v)$, we conclude that ${O_\meet(u) \cup O_\meet(v) \subseteq O_\meet(v)}$, and therefore~$O_\meet(u) \subseteq O_\meet(v)$.
We conclude that~$O_\meet$ is an ornamentation of~$D$.
Moreover, $O_\meet$ is clearly the meet of~$O_1$ and~$O_2$ in componentwise inclusion.

Given two ornamentations~$O_1$ and~$O_2$ of~$D$, consider now the map~$O_\join$ on~$V$ where~$O_\join(v)$ is the inclusion minimal subset~$U$ of~$V$ containing~$v$ and such that~$u \in U$ implies~$O_1(u) \cup O_2(u) \subseteq U$.
Note that~$O_\join(v)$ is well-defined as it can be constructed by induction, starting from~$U = \{v\}$ and adding inductively~$O_1(u) \cup O_2(u)$ for all~$u \in U$ until it stabilizes.
Since~$O_1(u) \cup O_2(u) \in \Orn(u \in D)$ by \cref{lem:unionOrnaments}, this inductive construction maintains the invariant that there is a path from~$u$ to~$v$ in~$U$ for any~$u \in U$.
Hence, we obtain that~$O_\join(v) \in \Orn(v \in D)$ for any~$v \in V$.
Consider now~$u,v \in V$ such that~$u \in O_\join(v)$.
Since~$u \in O_\join(v)$ and~$w \in O_\join(v)$ implies~$O_1(w) \cup O_2(w) \subseteq O_\join(v)$, we obtain that~$O_\join(u) \subseteq O_\join(v)$ by minimality of~$O_\join(u)$.
We conclude that~$O_\join$ is an ornamentation of~$D$.
Moreover, $O_\join$ is clearly the join of~$O_1$ and~$O_2$ in componentwise inclusion.
\end{proof}

\begin{remark}
The minimal ornamentation of~$D$ sends each vertex~$v \in V$ to the singleton~$\{v\}$, and the maximal ornamentation sends each vertex~$v \in V$ to the set~$\lessin{D}{v}$ of vertices with a path to~$v$ in~$D$.
\end{remark}

\begin{example}
The ornamentation lattice of the graph with a single edge is a $2$-element chain.
More generally, the ornamentation lattice of a graph with no path of length $2$ is a boolean lattice (see \eg \cref{fig:ornamentationsNI}\,(left) and \ref{fig:ornamentationsKT}\,(right)).
In fact, the boolean lattices are the only distributive ornamentation lattices.
\end{example}

\begin{example}
The ornamentation lattice of the $n$-element path is isomorphic to the Tamari lattice on binary trees with $n$ nodes (see \eg \cref{fig:ornamentationsNI}\,(right)).
The isomorphism is identical to the construction of bracket vectors in~\cite{HuangTamari}.
\end{example}

%%%%%%%%%%%%%%%%%%%%%%%%%%%%%%%%%%%%%%

\subsection{Reorientations}
\label{subsec:reorientations}

\begin{figure}[b]
	\centerline{\begin{tabular}{c} \includegraphics[scale=.68,valign=c]{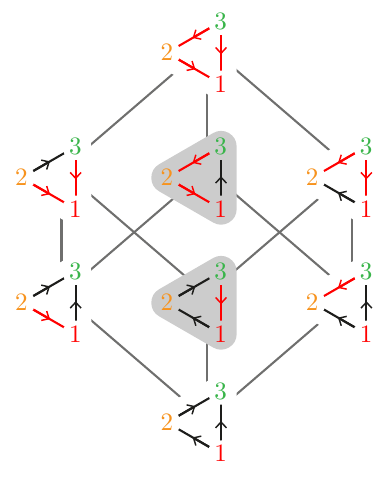} \\ \includegraphics[scale=.68,valign=c]{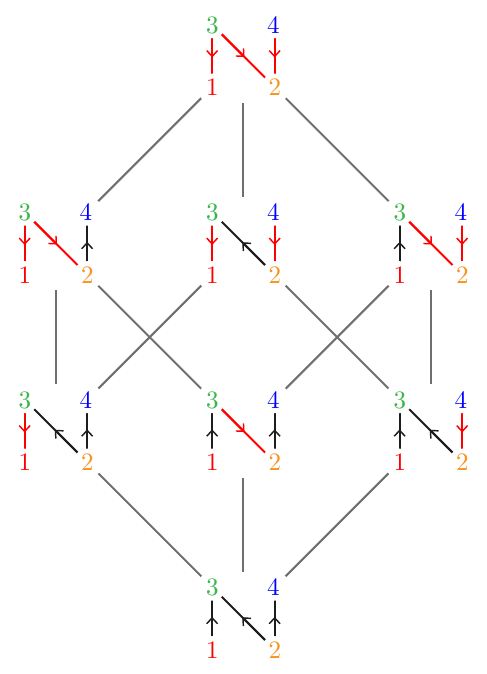} \end{tabular}\qquad\includegraphics[scale=.68,valign=c]{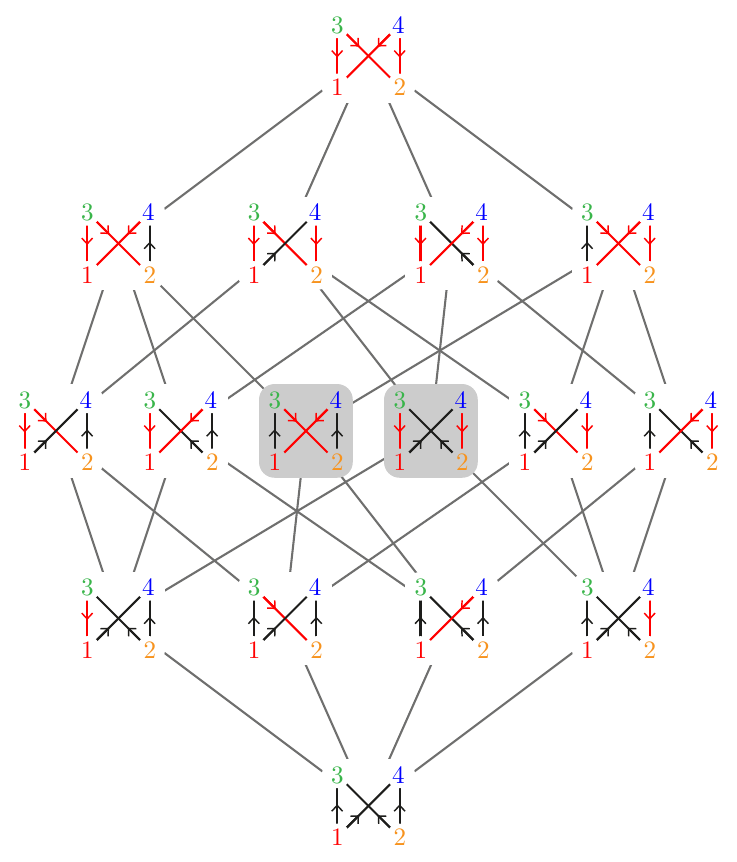}}
	\caption{The reorientation lattices~$\Reori(\tc(\Kgraph))$, $\Reori(\tc(\Ngraph))$ and~$\Reori(\tc(\Tgraph))$. (Gray = cyclic). See also \cref{fig:acyclicReorientationsKT}.}
	\label{fig:reorientationsKNT}
\end{figure}

We now consider reorientations of graphs, as illustrated in \cref{fig:reorientationsKNT}.

\begin{definition}
\label{def:Reori}
A \defn{reorientation}~$R$ of a directed graph~$E$ is a directed graph obtained from~$E$ by reversing a subset~$\rev(R)$ of arrows of~$E$.
The \defn{reorientation lattice}~$\Reori(E)$ of~$E$ is the boolean lattice of reorientations of~$E$, ordered by inclusion of subsets of reversed arrows, meaning $R_1 \le R_2$ if and only if~${\rev(R_1) \subseteq \rev(R_2)}$.
\end{definition}

In this paper, we actually need to consider the reorientation lattice~$\Reori(\tc(D))$ of the transitive closure~$\tc(D)$ of the directed graph~$D$.
Recall that, for all~$u,v \in V$, the transitive closure~$\tc(D)$ has an edge~$(u,v)$ if and only if~$D$ has a directed path from~$u$ to~$v$.
There is a natural surjection from the reorientation lattice~$\Reori(\tc(D))$ to the ornamentation lattice~$\Orn(D)$.

\begin{definition}
\label{def:Reori2Orn}
For a reorientation~$R$ of~$\tc(D)$, we define a map~$\orn{R}$ on~$V$ which associates to each vertex~$v \in V$ the inclusion maximal ornament of~$D$ at~$v$ contained in the subset of vertices~$u \in V$ with a directed path to~$v$ in $\rev(R)$ (meaning in the subgraph of~$D$ consisting of the arrows which are reversed in~$R$).
\end{definition}

\begin{lemma}
\label{lem:Reori2Orn1}
For any reorientation~$R$ of~$\tc(D)$, the map~$\orn{R}$ is an ornamentation of~$D$.
\end{lemma}

\begin{proof}
By definition, $\orn{R}(v) \in \Orn(v \in D)$ for all~$v \in V$.
Consider now~$u,v \in V$ such that~${u \in \orn{R}(v)}$.
Since~$\orn{R}(u) \in \Orn(u \in D)$ and~$\orn{R}(v) \in \Orn(v \in D)$ with~${u \in \orn{R}(v)}$, \cref{lem:unionOrnaments} ensures that $\orn{R}(u) \cup \orn{R}(v) \in \Orn(v \in D)$.
Moreover, by path concatenation, there is a path from any vertex of~$\orn{R}(u) \cup \orn{R}(v)$ to~$v$ in~$\rev(R)$.
By maximality of~$\orn{R}(v)$, we conclude that~$\orn{R}(u) \cup \orn{R}(v) \subseteq \orn{R}(v)$, and so~${\orn{R}(u) \subseteq \orn{R}(v)}$.
\end{proof}

\begin{lemma}
\label{lem:Reori2Orn2}
The map~$R \mapsto \orn{R}$ is order-preserving (meaning that~$R_1 \le R_2 \implies \orn{R_1} \le \orn{R_2}$).
\end{lemma}

\begin{proof}
If~$R_1 \le R_2$, then $\rev(R_1) \subseteq \rev(R_2)$, hence any path in~$\rev(R_1)$ is a path in~$\rev(R_2)$, and so~$\orn{R_1}(v) \subseteq \orn{R_2}(v)$ for any~$v \in V$, hence~$\orn{R_1} \le \orn{R_2}$.
\end{proof}

\begin{definition}
\label{def:Orn2Reori}
For an ornamentation~$O$ of~$D$, we define a reorientation~$\reori{O}$ of~$\tc(D)$ where for each~$(u,v) \in \tc(D)$, we have~$(u,v) \in \rev(\reori{O})$ if and only if~$u \in O(v)$.
\end{definition}

\begin{lemma}
\label{lem:Orn2Reori1}
The map~$O \mapsto \reori{O}$ is order-preserving.
\end{lemma}

\begin{proof}
If~$O_1 \le O_2$, then $O_1(v) \subseteq O_2(v)$ for all~$v \in V$, thus~$\rev(\reori{O_1}) \subseteq \rev(\reori{O_2})$, hence~${\reori{O_1} \le \reori{O_2}}$.
\end{proof}

\begin{lemma}
\label{lem:Orn2Reori2}
For any ornamentation~$O$ of~$D$, we have~$\orn{\reori{O}} = O$. In other words, the map~${O \mapsto \reori{O}}$ is a section of the map~$R \mapsto \orn{R}$.
\end{lemma}

\begin{proof}
Let~$u,v \in V$ such that there is a path from~$u$ to~$v$ in~$\rev(\reori{O})$.
Let~$u = w_0, w_1, \dots, w_p = v$ denote the vertices along this path.
We have~$(w_{i-1},w_i) \in \rev(\reori{O})$ hence~$w_{i-1} \in O(w_i)$ for each~$1 \le i \le p$.
Since~$O$ is an ornamentation, an immediate induction shows that~$w_i \in O(v)$ for all~$0 \le i \le p$.
We conclude that $u \in O(v)$ if and only if there is a path from~$u$ to~$v$ in~$\rev(\reori{O})$.
Since~$O(v)$ is an ornament of~$D$ at~$v$, we conclude that~$\orn{\reori{O}}(v) = O(v)$.
\end{proof}

To sum up, we obtained the following statement.

\begin{proposition}
\label{prop:Reori2Orn}
The map~$R \mapsto \orn{R}$ is an order-preserving surjection from the reorientation lattice~$\Reori(\tc(D))$ to the ornamentation lattice~$\Orn(D)$.
\end{proposition}

\begin{remark}
\label{rem:Orn2ReoriMiddle}
In general, the reorientation~$\reori{O}$ is neither minimal nor maximal among the reorientations~$R$ of~$\tc(D)$ such that~$\orn{R} = O$.
For instance, the following three reorientations belong to the fiber of the ornamentation $O = \!\!\raisebox{-.25cm}{\includegraphics[scale=.8]{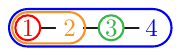}}$:
\[
\raisebox{-.23cm}{\includegraphics[scale=.8]{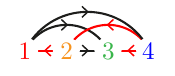}}\!\! < \reori{O} = \!\!\!\!\raisebox{-.23cm}{\includegraphics[scale=.8]{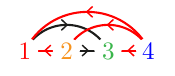}}\!\! < \!\!\raisebox{-.23cm}{\includegraphics[scale=.8]{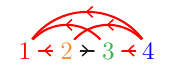}}
\]
\end{remark}

\begin{remark}
\label{rem:Reori2OrnIntervals1}
The fibers of the surjection~$R \mapsto \orn{R}$ do not all admit a minimum nor a maximum.
For instance, 
\begin{itemize}
\item the fiber of \raisebox{-.25cm}{\includegraphics[scale=.8]{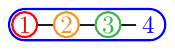}} has two minima \!\!\raisebox{-.23cm}{\includegraphics[scale=.8]{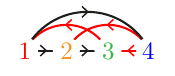}}\!\! and \!\!\raisebox{-.23cm}{\includegraphics[scale=.8]{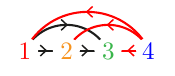}}\!\!\!\!,
\item the fiber of \raisebox{-.25cm}{\includegraphics[scale=.8]{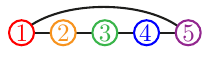}} has two maxima \!\!\raisebox{-.23cm}{\includegraphics[scale=.8]{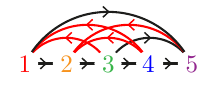}}\!\! and \!\!\raisebox{-.23cm}{\includegraphics[scale=.8]{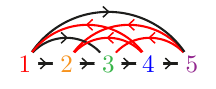}}\!\!\!\!.
\end{itemize}
We will see later in \cref{prop:Reori2OrnMaxTree} that the fibers always admit a maximum when~$D$ is a directed tree.
\end{remark}

\begin{remark}
\label{rem:Reori2OrnSemilatticeMap}
The surjection~$R \mapsto \orn{R}$ does not preserve neither the meet nor the join semilattice structure, even when~$D$ is a path.
For instance,

\centerline{
\begin{tabular}{c@{$\;\join\;$}c@{\;}c@{\;}c@{\qquad\text{and}\qquad}c@{$\;\meet\;$}c@{\;}c@{\;}c@{}c}
	\raisebox{-.2cm}{\includegraphics[scale=.8]{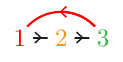}} &
	\raisebox{-.2cm}{\includegraphics[scale=.8]{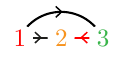}} &
	$=$ &
	\raisebox{-.2cm}{\includegraphics[scale=.8]{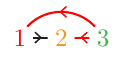}} &
	\raisebox{-.2cm}{\includegraphics[scale=.8]{reorientationI3}} &
	\raisebox{-.2cm}{\includegraphics[scale=.8]{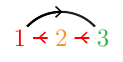}} &
	$=$ &
	\raisebox{-.2cm}{\includegraphics[scale=.8]{reorientationI2}} &
	,
	\\
	\raisebox{-.2cm}{\includegraphics[scale=.8]{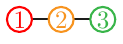}} &
	\raisebox{-.2cm}{\includegraphics[scale=.8]{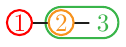}} &
	$\ne$ &
	\raisebox{-.2cm}{\includegraphics[scale=.8]{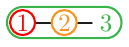}} &
	\raisebox{-.2cm}{\includegraphics[scale=.8]{ornamentationI3}} &
	\raisebox{-.2cm}{\includegraphics[scale=.8]{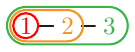}} &
	$\ne$ &
	\raisebox{-.2cm}{\includegraphics[scale=.8]{ornamentationI2}} &
	.
\end{tabular}
}

\medskip\noindent
To correct this defect, we can restrict our attention to transitively closed reorientations of~$\tc(D)$.
\end{remark}

\begin{definition}
\label{def:transitivelyClosed}
A reorientation~$R$ of~$\tc(D)$ is 
\begin{itemize}
\item \defn{transitively closed} if~$(u,v), (v,w) \in \rev(R)$ implies $(u,w) \in \rev(R)$,
\item \defn{transitively coclosed} if~$(u,v), (v,w) \in \tc(T) \ssm \rev(R)$ implies $(u,w) \in \tc(T) \ssm \rev(R)$,
\item \defn{transitively biclosed} if it is both transitively closed and transitively coclosed. 
\end{itemize}
We denote by~$\Rcl(\tc(D))$, $\Rco(\tc(D))$, and~$\Rbi(\tc(D))$ the sets of transitively closed, transitively coclosed, and transitively biclosed reorientations of $\tc(D)$ respectively.
\end{definition}

\begin{remark}
\label{rem:biclosedCyclic}
A transitively biclosed reorientation of~$\tc(D)$ is not necessarily acyclic, even when~$D$ is a tree.
For instance,
\[
\includegraphics[scale=.8,valign=c]{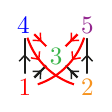}
\qquad\text{and}\qquad
\includegraphics[scale=.8,valign=c]{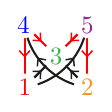}
\]
are cyclic biclosed reorientations of $\tc(\Xgraph)$, where $\Xgraph$ is the graph in \cref{fig:ornamentationsX}.
We will see in \cref{prop:AReoriLatticeUT} that all transitively biclosed reorientations are acyclic when~$T$ is an unstarred tree.
\end{remark}

\begin{lemma}
\label{lem:rcl-subsemilattice}
The set~$\Rcl(\tc(D))$ (resp.~$\Rco(\tc(D))$) of transitively closed (resp.~coclosed) reorientations induces a meet (resp.~join) subsemilattice of~$\Reori(\tc(D))$, which is a lattice.
\end{lemma}

\begin{proof}
As the intersection of transitively closed relations is transitively closed, $R_1, R_2 \in \Rcl(\tc(D))$ implies $R_1 \meet R_2 \in \Rcl(\tc(D))$.
As~$\tc(D)$ is transitively closed, the maximal reorientation of~$\tc(D)$ is transitively closed, and so $\Rcl(\tc(D))$ is a bounded meet semilattice, hence a lattice.
The proof for~$\Rco(\tc(D))$ is symmetric.
\end{proof}

\begin{remark}
\label{rem:biclosedNotLattice}
Observe however that the subposet of~$\Reori(\tc(D))$ induced by the set~$\Rbi(\tc(D))$ of transitively biclosed reorientations may fail to be a lattice.
For instance, consider the following four transitively biclosed reorientations of~$\tc(\Dgraph)$:
\[
R_1 = \!\! \includegraphics[scale=.8,valign=c]{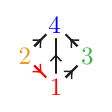}
\qquad
R_2 = \!\! \includegraphics[scale=.8,valign=c]{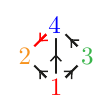}
\qquad
R_3 = \!\! \includegraphics[scale=.8,valign=c]{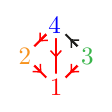}
\qquad
R_4 = \!\! \includegraphics[scale=.8,valign=c]{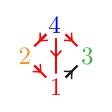}
\]
Then~$R_1 < R_3$, $R_1 < R_4$, $R_2 < R_3$, and~$R_2 < R_4$ and there is no transitively biclosed reorientation~$R$ such that~$R_1 < R$, $R_2 < R$, $R < R_3$, and~$R < R_4$.
We will see in \cref{prop:biclosedLatticeT} that~$\Rbi(\tc(T))$ is a lattice when~$T$ is a directed tree.
\end{remark}

\begin{proposition}
\label{prop:ReoriTC2Orn}
For any ornamentation~$O$ of~$D$, the reorientation~$\reori{O}$ is transitively closed.
Hence, the map~$R \mapsto \orn{R}$ restricts to a surjection from the transitively closed reorientation lattice~$\Rcl(\tc(D))$ to the ornamentation lattice~$\Orn(D)$.
\end{proposition}

\begin{proof}
Let~$u,v,w \in V$ such that~$(u,v), (v,w) \in \rev(\reori{O})$.
Then~$u \in O(v)$ and~$v \in O(w)$, and so~$u \in O(w)$ since~$O$ is an ornamentation.
Hence~$(u,w) \in \rev(\reori{O})$.
\end{proof}

\begin{remark}
\label{rem:Orn2ReoriNotTransitivelyCoclosed}
Note that~$\reori{O}$ is not always transitively coclosed.
For instance,
\[
O = \!\!\includegraphics[scale=.8,valign=c]{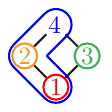} \qquad \reori{O} = \!\!\includegraphics[scale=.8,valign=c]{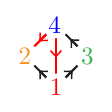}
.
\]
We will see in \cref{lem:Orn2ReoriT} that $\reori{O}$ is transitively biclosed when~$D$ is a directed tree.
\end{remark}

When restricted to transitively closed reorientations of~$\tc(D)$, the surjection~$R \mapsto \orn{R}$ behaves much nicer with respect to \cref{rem:Orn2ReoriMiddle,rem:Reori2OrnIntervals1,rem:Reori2OrnSemilatticeMap}.

\begin{proposition}
\label{prop:Reori2OrnMinTC}
For any ornamentation~$O$ of~$D$, the reorientation~$\reori{O}$ is the minimum transitively closed reorientation~$R$ of~$\tc(D)$ with~$\orn{R} = O$.
\end{proposition}

\begin{proof}
Consider any transitively closed reorientation~$R$ of~$\tc(D)$ such that~$\orn{R} = O$.
Let~$u,v \in V$ be such that~$(u,v) \in \rev(\reori{O})$, that is, such that~$u \in O(v)$.
By definition of~$\orn{R}(v) = O(v)$, there is a directed path from~$u$ to~$v$ in~$\rev(R)$.
As $R$ is transitively closed, this implies that~$(u,v) \in \rev(R)$.
We conclude that~$\rev(\reori{O}) \subseteq \rev(R)$, and so~$\reori{O} \le R$.
\end{proof}

\begin{remark}
\label{rem:Reori2OrnIntervals2}
Note that the second example of \cref{rem:Reori2OrnIntervals1} shows that the set of transitively closed reorientations~$R$ of~$\tc(D)$ such that~$\orn{R} = O$ does not admit a maximum in general.
We will see in \cref{prop:ReoriTC2OrnIntervalT} that it does if~$D$ is a directed tree.
\end{remark}

\begin{proposition}
\label{prop:Reori2OrnMeetSemilaticeTC}
The map~$R \mapsto \orn{R}$ is a meet semilattice morphism from the transitively closed reorientation lattice~$\Rcl(\tc(D))$ to the ornamentation lattice~$\Orn(D)$.
\end{proposition}

\begin{proof}
Given two transitively closed reorientations~$R_1$ and~$R_2$ of $\tc(D)$, we want to show that $\orn{R_1 \meet R_2} = \orn{R_1} \meet \orn{R_2}$.
Since~$R \mapsto \orn{R}$ is order-preserving (\cref{lem:Reori2Orn2}) and~$R_1 \meet R_2 \le R_1$ and~$R_1 \meet R_2 \le R_2$, we have~$\orn{R_1 \meet R_2} \le \orn{R_1}$ and~$\orn{R_1 \meet R_2} \le \orn{R_2}$, hence~$\orn{R_1 \meet R_2} \le \orn{R_1} \meet \orn{R_2}$.

Conversely, consider~$u \in (\orn{R_1} \meet \orn{R_2})(v)$.
There exists a path~$u = w_0, w_1, \dots, w_p = v$ in $(\orn{R_1} \meet \orn{R_2})(v) \subseteq \orn{R_1}(v) \cap \orn{R_2}(v)$ (by \cref{thm:OrnMeetJoin}).
By definition of $\orn{R_1}$ and $\orn{R_2}$, there is a path from each $w_i$ to $v$ in both $\rev(R_1)$ and $\rev(R_2)$.
Since~$R_1$ and~$R_2$ are both transitively closed, each arrow $(w_i,v)$ of $\tc(D)$ is in both~$\rev(R_1)$ and~$\rev(R_2)$,
and therefore in $\rev(R_1) \cap \rev(R_2) = \rev(R_1 \meet R_2)$ (by \cref{lem:rcl-subsemilattice}).
So, $\{u = w_0, w_1, \dots, w_p = v\} \in \Orn(v \in D)$ and each $w_i$ has a directed path to $v$ in $\rev(R_1) \cap \rev(R_2)$.
By maximality of $\orn{R_1 \meet R_2}(v)$, we deduce $u \in \orn{R_1 \meet R_2}(v)$.
Thus, we have shown that $\orn{R_1} \meet \orn{R_2} \le \orn{R_1 \meet R_2}$, and we conclude that~$\orn{R_1 \meet R_2} = \orn{R_1} \meet \orn{R_2}$.
\end{proof}

\begin{remark}
In contrast, the first example of \cref{rem:Reori2OrnSemilatticeMap} shows that the map~$R \mapsto \orn{R}$ is not necessarily a join semilattice morphism from the transitively closed reorientation lattice~$\Rcl(\tc(D))$ to the ornamentation lattice~$\Orn(D)$, even when~$D$ is a path.
\end{remark}

\begin{remark}
Note that the image of the transitively coclosed reorientations of~$\tc(D)$ by the map~$R \mapsto \orn{R}$ does not cover all ornamentations of~$D$.
For instance, for the graph~$\Dgraph$ of \cref{fig:ornamentationsD}, the fibers under~$R \mapsto \orn{R}$ of the two ornamentations
\[
\includegraphics[scale=.8,valign=c]{cyclicOrnamentationD1}
\qquad
\includegraphics[scale=.8,valign=c]{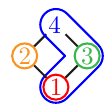}
\]
are singletons, respectively containing the reorientations
\[
\includegraphics[scale=.8,valign=c]{cyclicReorientationD1}
\qquad
\includegraphics[scale=.8,valign=c]{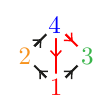}
\]
which are not transitively coclosed.
\end{remark}

%%%%%%%%%%%%%%%%%%%%%%%%%%%%%%%%%%%%%%

\subsection{Sourcings}
\label{subsec:sourcings}

We now assume that~$V$ is totally ordered, say that~$V = [n]$ with the natural order.
We consider the following maps on hypergraphs (they are usually called orientations of hypergraphs, we have preferred the term sourcings of hypergraphs to avoid confusion with the reorientations of graphs).

\begin{definition}
\label{def:Sour}
A \defn{sourcing}~$S$ of a hypergraph~$\HH$ on~$V$ is a map~$S : \HH \to V$ such that~$S(H) \in H$ for all hyperedges~$H \in \HH$.
The \defn{sourcing lattice}~$\Sour(\HH)$ is the lattice of sourcings of~$\HH$ ordered componentwise, meaning~$S_1 \le S_2$ if and only if~$S_1(H) \le S_2(H)$ for all~$H \in \HH$.
In other words, $\Sour(\HH)$ is the Cartesian product of the subchains of~$V$ induced by the hyperedges~$\HH$.
\end{definition}

We now assume that~$D$ is an increasing graph on~$V = [n]$, meaning that~$u < v$ for any edge~$(u,v)$ of~$D$, and we consider sourcings of the path hypergraph of~$D$.

\begin{definition}
\label{def:pathHypergraph}
The \defn{path hypergraph} of a directed graph~$D$ is the collection~$\PP(D)$ of vertex sets of directed paths in~$D$.
\end{definition}

We first observe that any sourcing~$S$ of~$\PP(D)$ can be lifted to a reorientation~$\reori{S}$ of~$\tc(D)$

\begin{definition}
\label{def:Sour2Reori}
Consider a sourcing~$S$ of~$\PP(D)$.
We denote by~$\rev(S)$ the set of pairs~$(u,v) \in [n]^2$ such that there is a directed path~$P \in \PP(D)$ from~$u$ to~$v$ with~$S(P) = v$.
We denote by~$\reori{S}$ the reorientation of~$\tc(D)$ defined by~$\rev(\reori{S}) = \rev(S)$.
\end{definition}

\begin{lemma}
\label{lem:Sour2Reori}
The map~$S \mapsto \reori{S}$ is order-preserving.
\end{lemma}

\begin{proof}
Consider two sourcings~$S_1 \le S_2$ of~$\PP(D)$.
For any~$(u,v) \in \rev(S_1)$, there is a path~$P \in \PP(D)$ from~$u$ to~$v$ with~$S_1(P) = v$.
We have~$v = S_1(P) \le S_2(P) \le \max(P) = v$, and so~$S_2(P) = v$, hence~$(u,v) \in \rev(S_2)$.
Thus,~$\rev(\reori{S_1}) = \rev(S_1) \subseteq \rev(S_2) = \rev(\reori{S_2})$, and so~${\reori{S_1} \le \reori{S_2}}$.
\end{proof}

In contrast, we cannot define a map from the reorientation lattice~$\Reori(\tc(D))$ to the sourcing lattice~$\Sour(\PP(D))$, because the subgraph of a reorientation of~$\tc(D)$ induced by the vertices of a path of~$\PP(D)$ might not have a source.
We will fix this issue by restricting to acyclic reorientations and acyclic sourcings in \cref{sec:acyclic}.

Observe also that $S \mapsto \reori{S}$ is not injective.
However, we can still use~$S \mapsto \reori{S}$ to define a map from all sourcings of~$\PP(D)$ to all ornamentations of~$D$.

\begin{definition}
\label{def:Sour2Orn}
For a sourcing~$S$ of~$\PP(D)$, we define a map~$\orn{S}$ on~$V$ which associates to each vertex~$v \in V$ the inclusion maximal ornament of~$D$ at~$v$ contained in the subset of vertices~$u \in V$ with a directed path to~$v$ in $\rev(S)$.
\end{definition}

\begin{lemma}
\label{lem:Sour2Orn1}
For any sourcing~$S$ of~$\PP(D)$, we have~$\orn{S} = \orn{\reori{S}}$.
\end{lemma}

\begin{proof}
Immediate from \cref{def:Reori2Orn,def:Sour2Reori,def:Sour2Orn}.
\end{proof}

\begin{lemma}
\label{lem:Sour2Orn2}
For any sourcing~$S$ of~$\PP(D)$, the map~$\orn{S}$ is an ornamentation of~$D$.
\end{lemma}

\begin{proof}
Follows from~\cref{lem:Reori2Orn1,lem:Sour2Orn1}.
\end{proof}

\begin{lemma}
\label{lem:Sour2Orn3}
The map~$S \mapsto \orn{S}$ is order-preserving.
\end{lemma}

\begin{proof}
Follows from~\cref{lem:Reori2Orn2,lem:Sour2Reori,lem:Sour2Orn1}.
\end{proof}

\begin{definition}
\label{def:Orn2Sour}
For an ornamentation~$O$ of~$D$, we define a sourcing~$\sour{O}$ of~$\PP(D)$ where the source~$\sour{O}(P)$ for a path~$P$ from~$u$ to~$v$ in~$D$ is the maximal~$w \in P$ such that~$u \in O(w)$.
\end{definition}

\begin{lemma}
\label{lem:Orn2Sour1}
The map~$O \mapsto \sour{O}$ is order-preserving.
\end{lemma}

\begin{proof}
Assume that~$O_1 \le O_2$, so that~$O_1(v) \subseteq O_2(v)$ for all~$v \in V$.
Consider any path~$P \in \PP(D)$ from~$u$ to~$v$, and let~$w \eqdef \sour{O_1}(P)$.
Then we have~$u \in O_1(w) \subseteq O_2(w)$, and so~$\sour{O_2}(P) \ge w$ by definition.
Therefore, $\sour{O_1}(P) \le \sour{O_2}(P)$ for any path~$P \in \PP(D)$, hence~${\sour{O_1} \le \sour{O_2}}$.
\end{proof}

\begin{lemma}
\label{lem:Orn2Sour2}
The following are equivalent for any ornamentation~$O$ of~$D$ and vertices~$u,v \in [n]$:
\begin{enumerate}[(i)]
\item $(u,v) \in \rev(\sour{O})$,
\item there is a path from~$u$ to~$v$ in~$\rev(\sour{O})$,
\item $u \in O(v)$.
\end{enumerate}
\end{lemma}

\begin{proof}
\uline{\textsl{(i) $\Rightarrow$ (ii):}}
Nothing to prove.

\medskip\noindent
\uline{\textsl{(ii) $\Rightarrow$ (iii):}}
Assume that there is a path from~$u$ to~$v$ in~$\rev(\sour{O})$.
Denote by~${u = w_0, w_1, \dots, w_p = v}$ the vertices along this path.
For each $1 \le i \le p$, as~$(w_{i-1}, w_i) \in \rev(\sour{O})$, there is a path~$P_i \in \PP(D)$ from~$w_{i-1}$ to~$w_i$ with~$\sour{O}(P_i) = w_i$ by \cref{def:Sour2Reori}, hence~$w_{i-1} \in O(w_i)$ by \cref{def:Orn2Sour}.
Since~$O$ is an ornamentation, an immediate induction shows that~$w_i \in O(v)$ for all~$0 \le i \le p$.
We conclude that $u \in O(v)$

\medskip\noindent
\uline{\textsl{(iii) $\Rightarrow$ (i):}}
Assume that~$u \in O(v)$.
Then there is a path~$P \in \PP(D)$ from~$u$ to~$v$ in~$O(v)$.
Since~${u \in O(v)}$, we get~$\sour{O}(P) = v$ by \cref{def:Orn2Sour}.
Hence,~$(u,v) \in \rev(\sour{O})$ by \cref{def:Sour2Reori}.
\end{proof}

\begin{lemma}
\label{lem:Orn2Sour3}
For any ornamentation~$O$ of~$D$, we have~$\orn{\sour{O}} = O$. In other words, the map~${O \mapsto \sour{O}}$ is a section of the map~$S \mapsto \orn{S}$.
\end{lemma}

\begin{proof}
By \cref{lem:Orn2Sour2}, there is a path from~$u$ to~$v$ in~$\rev(\sour{O})$ if and only if~$u \in O(v)$.
Since~$O(v)$ is an ornament of~$D$ at~$v$, we conclude by maximality that~$\orn{\sour{O}}(v) = O(v)$.
\end{proof}

\begin{lemma}
\label{lem:Orn2Sour4}
For any ornamentation~$O$ of~$D$, we have~$\reori{O} = \reori{\sour{O}}$.
\end{lemma}

\begin{proof}
Immediate from \cref{def:Orn2Reori,def:Sour2Reori,lem:Orn2Sour2}.
\end{proof}

\begin{remark}
Note however that there are sourcings~$S$ of~$\PP(D)$ with~$\reori{S} \ne \reori{\orn{S}}$.
For instance, on the graph~\includegraphics[scale=.8,valign=c]{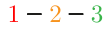}, for the sourcing~$S$ given by~$S(\{1,2\}) = 1$, $S(\{2,3\}) = 2$ and~$S(\{1,2,3\}) = 3$, we have $\rev(S) = \{(1,3)\}$ so that~$\orn{S} = \raisebox{-.2cm}{\includegraphics[scale=.8]{ornamentationI1}}$ and
\[
\reori{S} = \raisebox{-.2cm}{\includegraphics[scale=.8]{reorientationI1}} \ne \raisebox{-.2cm}{\includegraphics[scale=.8]{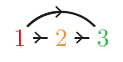}} = \reori{\orn{S}}.
\]
\end{remark}

To sum up, we obtained the following statement.

\begin{proposition}
\label{prop:Sour2Orn}
The map~$S \mapsto \orn{S}$ is an order-preserving surjection from the sourcing lattice~$\Sour(\PP(D))$ to the ornamentation lattice~$\Orn(D)$.
\end{proposition}

\begin{remark}
The map~$S \mapsto \orn{S}$ is obviously not injective.
For instance, $\orn{S}$ is the maximal ornamentation of~$D$ for any sourcing~$S$ of~$\PP(D)$ such that~$S(\{u,v\}) = v$ for each edge~$(u,v)$ of~$D$.
In contrast, we will see in \cref{sec:acyclic} that~$S \mapsto \orn{S}$ is injective on acyclic sourcings of~$\PP(D)$.
\end{remark}

%%%%%%%%%%%%%%%%%%%%%%%%%%%%%%%%%%%%%%
%%%%%%%%%%%%%%%%%%%%%%%%%%%%%%%%%%%%%%
%%%%%%%%%%%%%%%%%%%%%%%%%%%%%%%%%%%%%%

\section{Acyclic reorientations, acyclic sourcings, and acyclic ornamentations}
\label{sec:acyclic}

We now focus our attention on the acyclic case.
We assume here that~$D$ is an increasing graph on~$V = [n]$, meaning that $u < v$ for any edge~$(u,v)$ of~$D$. In particular, $D$ is acyclic.
We then consider the acyclic reorientation poset~$\AReori(\tc(D))$ (\cref{subsec:acyclicReorientations}), the acyclic sourcing poset~$\ASour(\PP(D))$ (\cref{subsec:acyclicSourcings}), and the acyclic ornamentation poset~$\AOrn(D)$ (\cref{subsec:acyclicOrnamentations}), and define natural surjective poset morphisms:
\[
\begin{array}{rcccccl}
	\fS_n & \surjection & \AReori(\tc(D)) & \surjection & \ASour(\PP(D)) & \bijection & \AOrn(D) \\
	\pi & \longmapsto & \areori{\pi} \qquad R & \longmapsto & \asour{R} \qquad S & \longmapsto & \aorn{S}
\end{array}
.
\]

%%%%%%%%%%%%%%%%%%%%%%%%%%%%%%%%%%%%%%

\subsection{Acyclic reorientations}
\label{subsec:acyclicReorientations}

\begin{figure}[b]
	\centerline{\includegraphics[scale=.68]{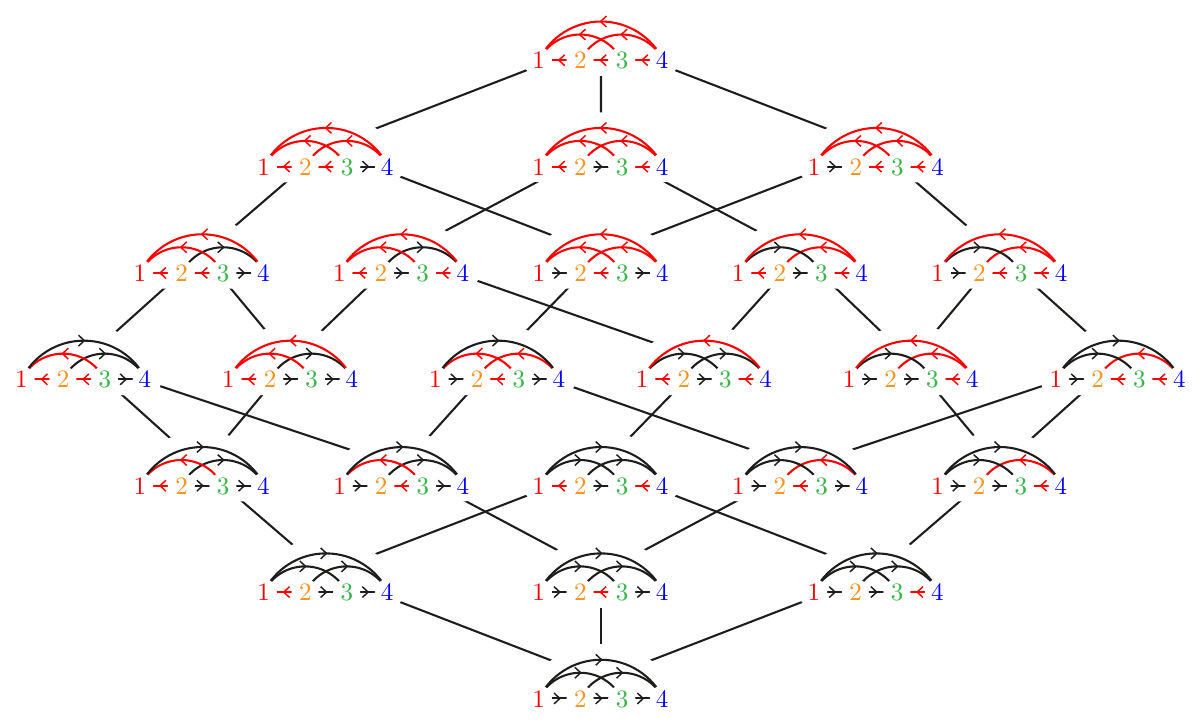}}
	\caption{The acyclic reorientation poset~$\AReori(\tc(\Igraph))$. (It is isomorphic to the weak order.)}
	\label{fig:acyclicReorientationsI}
\end{figure}

\begin{figure}
	\centerline{\includegraphics[scale=.68]{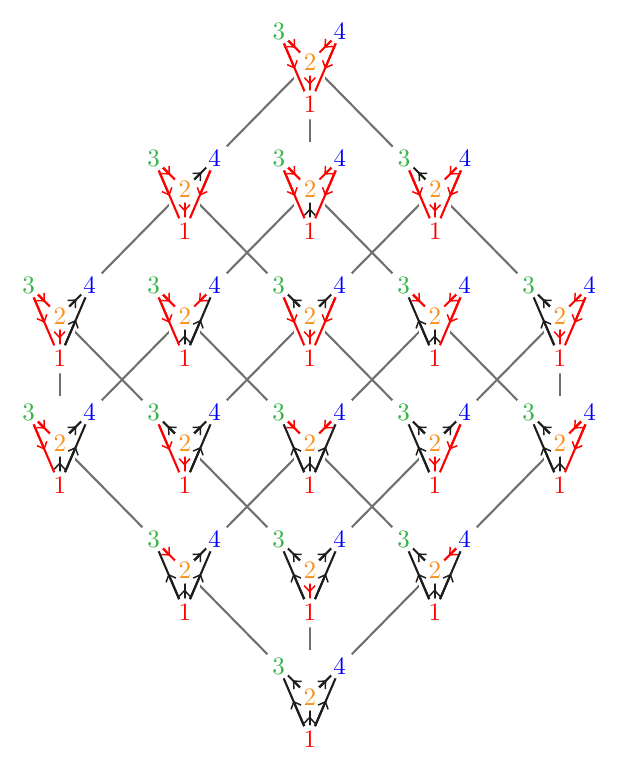} \qquad \includegraphics[scale=.68]{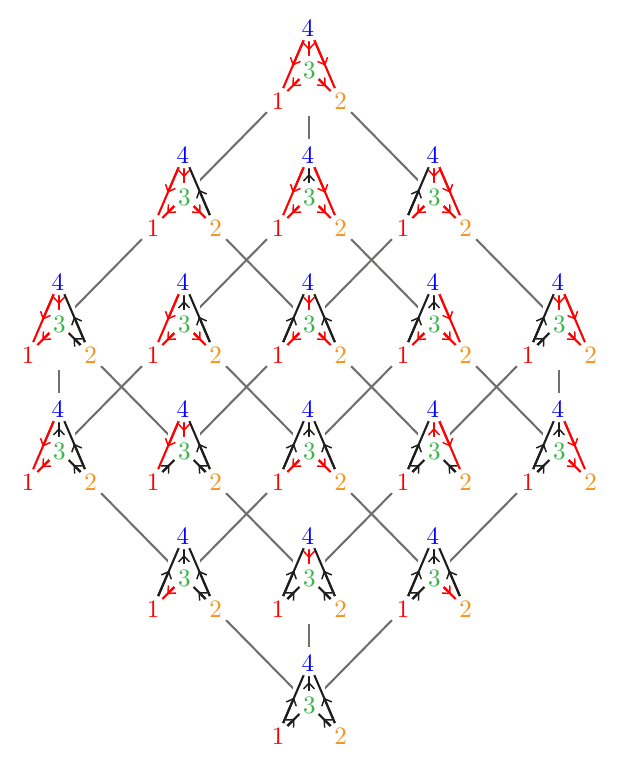}}
	\caption{The acyclic reorientation posets~$\AReori(\tc(\Ygraph))$ and~$\AReori(\tc(\Agraph))$. (They are both lattices, see \cref{prop:AReoriLatticeUT}.)}
	\label{fig:acyclicReorientationsAY}
\end{figure}

\begin{figure}
	\centerline{\includegraphics[scale=.68]{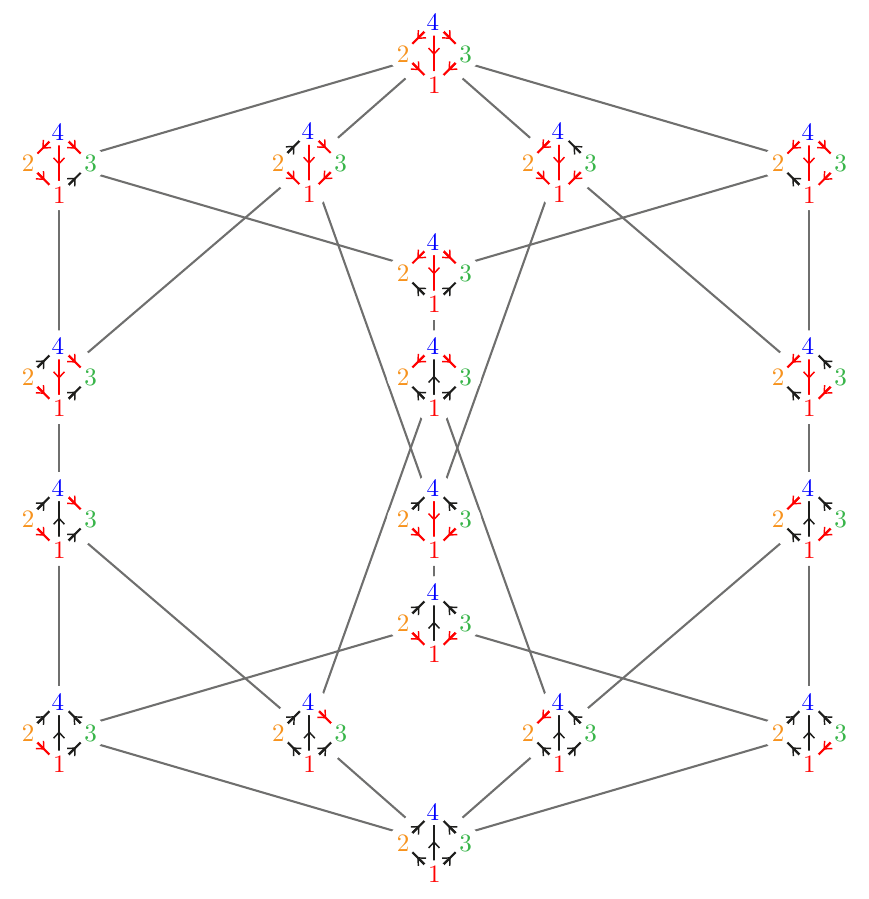}}
	\caption{The acyclic reorientation poset~$\AReori(\tc(\Dgraph))$. (Not a lattice, see \cref{fig:MacNeilleAcyclicReorientationsD}.)}
	\label{fig:acyclicReorientationsD}
\end{figure}

\begin{figure}
	\centerline{\includegraphics[scale=.68,valign=c]{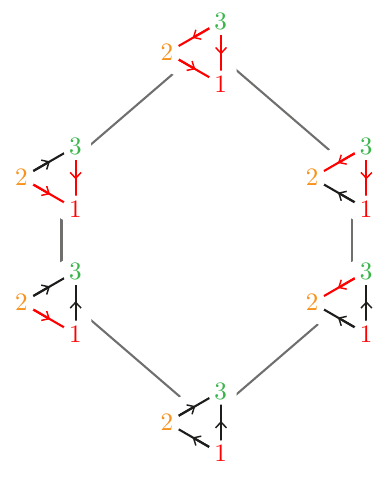}\qquad\includegraphics[scale=.68,valign=c]{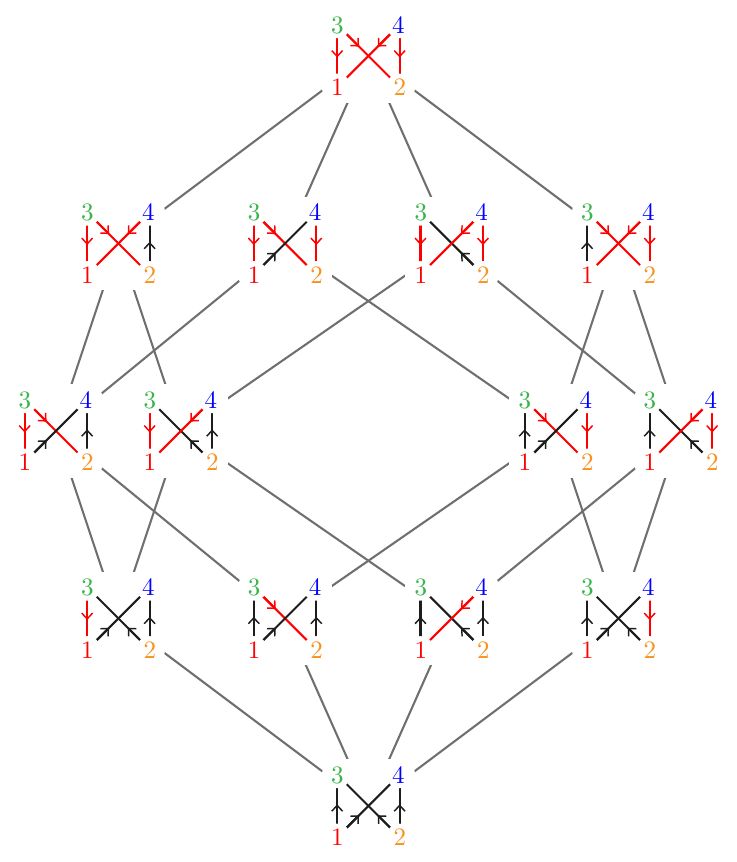}}
	\caption{The acyclic reorientation posets~$\AReori(\tc(\Kgraph))$ and~$\AReori(\tc(\Tgraph))$.}
	\label{fig:acyclicReorientationsKT}
\end{figure}

Acyclic reorientations of directed acyclic graphs is a classical subject, see for instance~\cite{Greene,GreeneZaslavsky,Pilaud-acyclicReorientationLattices} and \cref{fig:acyclicReorientationsAY,fig:acyclicReorientationsD,fig:acyclicReorientationsKT} for illustrations.

\begin{definition}
\label{def:AReori}
A reorientation of a directed graph~$E$ is \defn{acyclic} if it contains no directed cycle.
The \defn{acyclic reorientation poset}~$\AReori(E)$ is the subposet of the reorientation lattice~$\Reori(E)$ induced by acyclic reorientations.
\end{definition}

\begin{remark}
Let us recall the polytopal interpretation of the acyclic reorientation poset of~$E$.
Denote by~$(\b{e}_i)_{i \in [n]}$ the standard basis of~$\R^n$.
The \defn{graphical zonotope} of~$E$ is the Minkowski sum~$\simplex_E \eqdef \sum_{(u,v) \in E} \simplex_{uv}$, where~$\simplex_{uv}$ is the segment with endpoints~$\b{e}_u$ and~$\b{e}_v$.
The face lattice of~$\simplex_E$ is described by ordered partitions of~$E$.
In particular, the graph of~$\simplex_E$, oriented in the direction~$\b{\omega} \eqdef (n, n-1, \dots, 2, 1) - (1, 2, \dots, n-1, n) = (n-1, n-3, \dots, 3-n, 1-n)$, is isomorphic to the Hasse diagram of the acyclic reorientation poset~$\AReori(E)$.
\end{remark}

Note that the acyclic reorientation poset~$\AReori(E)$ is not always a lattice.
The lattice theory of acyclic reorientation posets was studied in detail in~\cite{Pilaud-acyclicReorientationLattices}.
For our purposes, we will only need the following statements.

\begin{proposition}[{\cite[Thm.~1]{Pilaud-acyclicReorientationLattices}}]
\label{prop:AReoriLat}
The acyclic reorientation poset~$\AReori(E)$ is a lattice if and only if the transitive reduction of any induced subgraph of~$E$ is a forest.
\end{proposition}

\begin{proposition}[{\cite[Thm.~9]{Pilaud-acyclicReorientationLattices}}]
\label{prop:AReoriLatJoinMeet}
When the acyclic reorientation poset~$\AReori(E)$ is a lattice, the meet and join of two acyclic reorientations~$R_1$ and~$R_2$ of~$E$ are given by
\begin{align*}
\rev(R_1 \join R_2) & = E \;\cap\; \tc \!\big( \rev(R_1) \cup \rev(R_2) \big) \\
\text{and}\qquad
E \ssm \rev(R_1 \meet R_2) & = E \;\cap\; \tc \!\big( E \ssm (\rev(R_1) \cap \rev(R_2)) \big).
\end{align*}
\end{proposition}

We now recall the standard map from permutations to acyclic reorientations.

\begin{definition}
\label{def:Perm2AReori}
For a permutation~$\pi$ of~$[n]$, we denote by~$\areori{\pi}$ the reorientation of~$\tc(D)$ where $(u,v) \in \tc(D)$ is reoriented in~$\areori{\pi}$ if~$\pi^{-1}(u) > \pi^{-1}(v)$.
\end{definition}

\begin{proposition}
\label{prop:Perm2AReori}
The map~$\pi \mapsto \areori{\pi}$ is an order-preserving surjection from the weak order on~$\fS_n$ to the acyclic reorientation poset~$\AReori(\tc(D))$.
\end{proposition}

\begin{proof}
Recall that the weak order on~$\fS_n$ is defined by the inclusion of inversion sets, where the inversion set of~$\pi \in \fS_n$ is~$\inv(\pi) \eqdef \set{(u,v)}{u < v \text{ and } \pi^{-1}(u) > \pi^{-1}(v)}$.
By definition, we have~$\rev(\areori{\pi}) = \inv(\pi) \cap \tc(D)$, hence~$\pi \mapsto \areori{\pi}$ is order-preserving.
It is surjective since the fiber of any reorientation~$R$ is the set of linear extensions of~$R$, which is non-empty when~$R$~is~acyclic.
\end{proof}

%%%%%%%%%%%%%%%%%%%%%%%%%%%%%%%%%%%%%%

\subsection{Acyclic sourcings}
\label{subsec:acyclicSourcings}

Acyclic sourcings of hypergraphs were introduced in~\cite{BenedettiBergeronMachacek,BergeronPilaud} as combinatorial models for the vertex sets of hypergraphic polytopes.

\begin{definition}
\label{def:ASour}
A sourcing~$S$ of a hypergraph~$\HH$ is \defn{acyclic} if there are no distinct~$H_0, \dots, H_k \in \HH$ such that~$S(H_{i-1}) \in H_i \ssm \{S(H_i)\}$ for all~$i \in [k]$ and~$S(H_k) \in H_0 \ssm \{S(H_0)\}$.
The \defn{acyclic sourcing poset}~$\ASour(\HH)$ is the subposet of the sourcing lattice~$\Sour(\HH)$ induced by acyclic sourcings.
\end{definition}

\begin{remark}
Let us recall the polytopal interpretation of the acyclic sourcing poset of a hypergraph~$\HH$.
We still denote by~$(\b{e}_i)_{i \in [n]}$ the standard basis of~$\R^n$.
The \defn{hypergraphic polytope} of~$\HH$ is the Minkowski sum
\(
\simplex_\HH \eqdef \sum_{H\in \HH} \simplex_H\,,
\)
where $\simplex_H$ is the simplex given by the convex hull of the points $\b{e}_h$ for~$h \in H$.
The face lattice of~$\simplex_\HH$ was described combinatorially in terms of acyclic orientations of~$\HH$ in~\cite{BenedettiBergeronMachacek}.
In particular, it is proved in~\cite{Gelinas} that the transitive closure of the graph of~$\simplex_\HH$, oriented in the direction~$\b{\omega}$, is isomorphic to the acyclic sourcing poset~$\ASour(\HH)$.
Note that the graph of~$\simplex_\HH$, oriented in the direction~$\b{\omega}$, is not always transitively reduced, hence not always isomorphic to the Hasse diagram of the acyclic sourcing poset~$\ASour(\HH)$.
\end{remark}

\begin{remark}
\label{rem:ASourLat}
Note that the acyclic sourcing poset of~$\HH$ is not always a lattice.
Characterizing the hypergraphs whose acyclic sourcing poset is a lattice seems to be a difficult question in general.
Such characterizations exist for specific families of hypergraphs, notably for graphs~\cite{Pilaud-acyclicReorientationLattices} (see \cref{prop:AReoriLat}), for interval hypergraphs~\cite{BergeronPilaud}, and (partially for) graph associahedra~\cite{BarnardMcConville}.
We will further study this question for subhypergraphs of the path hypergraph of increasing trees in \cref{sec:intreevalHypergraphicPosets}.
\end{remark}

We start with two simple observations about \cref{def:Sour2Reori} .

\begin{lemma}
\label{lem:revTCAcyclic}
If~$S$ is an acyclic sourcing of~$\PP(D)$, then the set~$\rev(S)$ defined in \cref{def:Sour2Reori} is transitive.
\end{lemma}

\begin{proof}
Assume that~$(u,v), (v,w) \in \rev(S)$.
There are directed paths~$P, Q$ from~$u$ to~$v$ and from~$v$ to~$w$ such that~$S(P) = v$ and~$S(Q) = w$ respectively.
Their concatenation~$PQ$ is a path from~$u$ to~$w$ with~$S(PQ) = w$ by acyclicity of~$S$ (indeed, $PQ$ and~$P$ form a cycle in~$S$ if~$S(PQ) \in P \ssm \{v\}$, and $PQ$ and~$Q$ form a cycle in~$S$ if~$S(PQ) \in Q \ssm \{w\}$).
Hence,~$(u,w) \in \rev(S)$.
\end{proof}

\begin{remark}
\label{rem:Sour2ReoriIncompatibleAcyclicity}
The reorientation~$\reori{S}$ of~$\tc(D)$ defined in \cref{def:Sour2Reori} by~$\rev(\reori{S}) = \rev(S)$ is not always acyclic even when~$S$ is acyclic.
For instance, consider the directed graph~$D$ and the reorientations~$R$ of~$\tc(D)$ and the sourcing~$S$ of~$\PP(D)$ given by
\[
D = \! \includegraphics[scale=.8,valign=c]{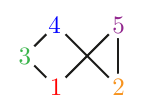}
\qquad
R = \! \includegraphics[scale=.8,valign=c]{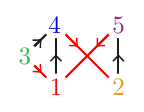}
\qquad\text{and}\qquad
\begin{array}{c}
S(13) = 3, \;\; S(134) = 3 \;\; S(15) = 5, \\
S(24) = 4, \;\; S(25) = 2, \;\; S(34) = 3.
\end{array}
\]
Note that~$S$ is acyclic and that~$\reori{S} = R$ is cyclic.
\end{remark}

We have observed in \cref{subsec:sourcings} that there is no natural surjection from all reorientations of~$\tc(D)$ to all sourcings of~$\PP(D)$.
We now recall that there is such a surjection when focusing on acyclic reorientations and acyclic sourcings.

\begin{definition}
\label{def:AReori2ASour}
For an acyclic reorientation~$R$ of~$\tc(D)$, we denote by~$\asour{R}$ the sourcing on~$\PP(D)$ where~$\asour{R}(P)$ is the source of~$P$ in~$R$, meaning the vertex~$u \in P$ such that~$(u,v) \in R$ for all~$v \in P \ssm \{u\}$ (it is well-defined since any path in~$D$ induces a clique in~$\tc(D)$).
\end{definition}

\begin{lemma}
\label{lem:AReori2ASour1}
For any acyclic reorientation~$R$ of~$\tc(D)$, we have~$\rev(\asour{R}) \subseteq \rev(R)$.
\end{lemma}

\begin{proof}
Let~$u,v \in [n]$ be such that~$(u,v) \in \rev(\asour{R})$. 
Then there is a path~$P \in \PP(D)$ from~$u$ to~$v$ such that~$\asour{R}(P) = v$.
We conclude that~$(v,u) \in R$, and so~$(u,v) \in \rev(R)$.
\end{proof}

\begin{lemma}
\label{lem:AReori2ASour2}
For any acyclic reorientation~$R$ of~$\tc(D)$, the sourcing~$\asour{R}$ is acyclic.
\end{lemma}

\begin{proof}
If~$P_0, \dots, P_k \in \PP(D)$ are such that~$\asour{R}(P_{i-1}) \in P_i \ssm \{\asour{R}(P_i)\}$, then~$\asour{R}(P_k), \dots, \asour{R}(P_0)$ is a directed path in~$R$, and so~$R$ has no edge from~$\asour{R}(P_0)$ to~$\asour{R}(P_k)$ by acyclicity of~$R$, hence~${\asour{R}(P_k) \notin P_0 \ssm \{\asour{R}(P_0)\}}$.
\end{proof}

\begin{lemma}
\label{lem:AReori2ASour3}
The map~$R \mapsto \asour{R}$ is order-preserving.
\end{lemma}

\begin{proof}
If~$R_1 \le R_2$, then~$\rev(R_1) \subseteq \rev(R_2)$, hence the source of any path~$P$ in~$R_1$ is smaller than the source of~$P$ in~$R_2$, hence~$\asour{R_1} \le \asour{R_2}$.
\end{proof}

\begin{definition}
\label{def:ASour2AReori}
Let~$S$ be an acyclic sourcing of~$\PP(D)$.
We denote by $\arr(S)$ the pairs~$(u,v) \in [n]^2$ such that there is~$P \in \PP(D)$ with~$u \in P$ and~$v = S(P)$.
Let~$\areori{S}$ be the reorientation of~$\tc(D)$ defined by~$\rev(\areori{S}) = \tc(\arr(S)) \cap \tc(D)$.
\end{definition}

\begin{remark}
\label{rem:ASour2AReori}
Note that~$\rev(S) \subseteq \rev(\areori{S})$, but the inclusion might be strict.
In fact, the map~$S \mapsto \areori{S}$ is not order-preserving.
For instance, consider the directed graph~$D$, the two reorientations~$R_1$ and~$R_2$ of~$\tc(D)$, and the two sourcings~$S_1$ and~$S_2$ of~$S_1$ of~$\PP(D)$ given by
\begin{gather*}
D = \! \includegraphics[scale=.8,valign=c]{graphR}
\qquad\qquad
R_1 = \! \includegraphics[scale=.8,valign=c]{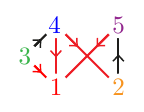}
\qquad\text{and}\qquad
R_2 = \! \includegraphics[scale=.8,valign=c]{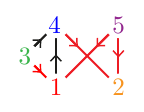}
\\
S_1(13) = 3, \;\; S_1(134) = 3, \;\; S_1(15) = 5, \;\; S_1(24) = 4, \;\; S_1(25) = 2, \;\; S_1(34) = 3, \\
S_2(13) = 3, \;\; S_1(134) = 3, \;\; S_2(15) = 5, \;\; S_2(24) = 4, \;\; S_2(25) = 5, \;\; S_2(34) = 3.
\end{gather*}
Note that~$R_1$, $R_2$, $S_1$ and~$S_2$ are all acyclic, and that~$R_1 = \areori{S_1}$ and~$R_2 = \areori{S_2}$ and conversely~$S_1 = \asour{R_1}$ and~$S_2 = \asour{R_2}$.
Moreover, one can check that~$R_1 \not< R_2$ while~$S_1 < S_2$.
\end{remark}

\begin{lemma}
\label{lem:ASour2AReori1}
For any acyclic sourcing~$S$ of~$\PP(D)$, the reorientation~$\areori{S}$ of~$\tc(D)$ is acyclic.
\end{lemma}

\begin{proof}
If~$\areori{S}$ contains a cycle, so does~$\arr(S)$, and so~$S$ is cyclic by \cref{def:ASour}.
\end{proof}

\begin{lemma}
\label{lem:ASour2AReori2}
For any acyclic sourcing~$S$ of~$\PP(D)$, we have~$\asour{\areori{S}} = S$. In other words, the map~$S \mapsto \areori{S}$ is a section of the map~$R \mapsto \asour{R}$.
\end{lemma}

\begin{proof}
Pick~$P \in \PP(D)$ and let~$v = S(P)$.
For any~$u \in P$, we have~$(u,v) \in \arr(S)$.
If~$u < v$, then~$(u,v) \in \rev(\areori{S})$.
If~$u > v$, then~$(v,u) \notin \tc(\arr(S))$ by acyclicity of~$S$, and so~$(v,u) \notin \rev(\areori{S})$, and so~$(u,v) \in \areori{S}$ (since $P$ is a tournament in~$\areori{S}$).
Hence, $v$ is the source of~$P$ in~$\areori{S}$, and so~$\asour{\areori{S}}(P) = v = S(P)$.
We conclude that~$\asour{\areori{S}} = S$. 
\end{proof}

To sum up, we obtained the following statement.

\begin{proposition}
\label{prop:AReori2ASour}
The map~$R \mapsto \asour{R}$ is an order-preserving surjection from the acyclic reorientation poset~$\AReori(\tc(D))$ to the acyclic sourcing poset~$\ASour(\PP(D))$.
\end{proposition}

\begin{remark}
Note that we could have argued the surjectivity of the map~$R \mapsto \asour{R}$ geometrically.
Namely, using that the hypergraphic polytope of $\PP(D)$ is a deformation of graphical zonotope of $\tc(D)$.
We have introduced the map~$S \mapsto \areori{S}$ instead to remain at the level of posets.
\end{remark}

%%%%%%%%%%%%%%%%%%%%%%%%%%%%%%%%%%%%%%

\subsection{Acyclic ornamentations}
\label{subsec:acyclicOrnamentations}

Remember from \cref{def:Reori2Orn} (resp.~\cref{def:Sour2Orn}) that we have defined an ornamentation~$\orn{R}$ (resp.~$\orn{S}$) from any reorientation~$R$ of~$\tc(D)$ (resp.~sourcing~$S$ of~$\PP(D)$).
We now observe that these maps behave well on acyclic reorientations of~$\tc(D)$ and acyclic sourcings of~$\PP(D)$, which enables us to define acyclic ornamentations of~$D$.

\begin{lemma}
\label{lem:ASour2AOrn1}
The map~$S \mapsto \orn{S}$ is injective on acyclic sourcings.
\end{lemma}

\begin{proof}
Let~$S_1 \ne S_2$ be two distinct acyclic sourcings of~$\PP(D)$.
There exists~$P \in \PP(D)$ such that~$u \eqdef S_1(P) \ne S_2(P) \defeq v$.
We can assume by symmetry that~$u < v$.
If~$u \in \orn{S_1}(v)$, then there is a path from~$u$ to~$v$ in~$\rev(S_1)$.
Let~$u = w_0, w_1, \dots, w_k = v$ denote the vertices of this path.
For each~$i \in [k]$, we have~$(w_{i-1},w_i) \in \rev(S_1)$, hence there is a path~$P_i$ from~$w_{i-1}$ to~$w_i$ with~$S_1(P_i) = w_i$.
Together with~$P$, this contradicts the acyclicity of~$S_1$.
We conclude that~${u \notin \orn{S_1}(v)}$.
Consider now the subpath~$Q$ of~$P$ from~$u$ to~$v$.
Since~$S_2$ is acyclic, we have~$S_2(Q) = S_2(P) = v$.
Hence,~$(u,v) \in \rev(S_2)$, which implies that~$u \in \orn{S_2}(v)$.
We conclude that~$\orn{S_1} \ne \orn{S_2}$.
\end{proof}

\begin{lemma}
\label{lem:acyclicSimplifiesDef}
For any acyclic reorientation~$R$ of~$\tc(D)$ (resp.~acyclic sourcing~$S$ of~$\PP(D)$) and any~$u,v \in [n]$ with~$u \in \orn{R}(v)$ (resp.~$\orn{S}(v)$), we have~$(u,v) \in \rev(R)$ (resp.~$\rev(S)$).
\end{lemma}

\begin{proof}
Since~$R$ is acyclic, $\rev(R)$ is transitive.
Similarly, since~$S$ is acyclic, $\rev(S)$ is transitive by \cref{lem:revTCAcyclic}.
The statement thus follows directly from the definition of~$\orn{R}(v)$ (resp.~$\orn{S}(v)$).
\end{proof}

\begin{lemma}
\label{lem:AReori2ASour2AOrn}
For any acyclic reorientation~$R$ of~$\tc(D)$, we have~$\orn{R} = \orn{\asour{R}}$.
\end{lemma}

\begin{proof}
Since~$R$ is acyclic, we have~$\rev(\asour{R}) \subseteq \rev(R)$ by \cref{lem:AReori2ASour1}, hence~$\orn{\asour{R}}(v) \subseteq \orn{R}(v)$ for any~$v \in [n]$.
Conversely, consider~$u,v \in [n]$ such that~$u \in \orn{R}(v)$.
Let~$P$ be a path from~$u$ to~$v$ in~$\orn{R}(v)$.
By \cref{lem:acyclicSimplifiesDef}, we have~$(w,v) \in \rev(R)$ for all~$w \in P$.
Therefore, $v$ is the source of~$P$ in~$R$, that is, $v = \asour{R}(P)$.
We thus obtain that~$(u,v) \in \rev(\asour{R})$.
We conclude that~$\rev(R) \subseteq \rev(\asour{R})$, hence that~$\orn{R}(v) \subseteq \orn{\asour{R}}(v)$ by maximality of~$\orn{\asour{R}}(v)$.
\end{proof}

\begin{lemma}
\label{lem:AOrn2}
The following conditions are equivalent for an ornamentation~$O$ of~$D$:
\begin{enumerate}[(i)]
\item there exists an acyclic reorientation~$R$ of~$\tc(D)$ such that~$\orn{R} = O$,
\item there exists an acyclic sourcing~$S$ of~$\PP(D)$ such that~$\orn{S} = O$.
\end{enumerate}
\end{lemma}

\begin{proof}
\uline{\textsl{(i) $\Rightarrow$ (ii):}}
If~$R$ is an acyclic reorientation of~$\tc(D)$ such that~$\orn{R} = O$, then~$S \eqdef \asour{R}$ is an acyclic sourcing of~$\PP(D)$ such that~$\orn{S} = \orn{\asour{R}} = \orn{R} = O$ by \cref{lem:AReori2ASour2AOrn}.

\medskip\noindent
\uline{\textsl{(ii) $\Rightarrow$ (i):}}
If~$S$ is an acyclic sourcing of~$\PP(D)$ such that~$\orn{S} = O$, then~$R \eqdef \areori{S}$ is an acyclic reorientation of~$\tc(D)$ such that~$\orn{R} = \orn{\areori{S}} = \orn{\asour{\areori{S}}} = \orn{S} = O$ by \cref{lem:ASour2AReori2,lem:AReori2ASour2AOrn}.
\end{proof}

\begin{remark}
\label{rem:Orn2Sour&Orn2ReoriIncompatibleAcyclicity}
Careful, the conditions of \cref{lem:AOrn2} are not equivalent to $\reori{O}$ and/or~$\sour{O}$ being acyclic.
For instance, consider the two ornamentations
\[
O_1 = \! \includegraphics[scale=.8,valign=c]{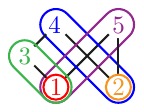}
\qquad\text{and}\qquad
O_2 = \! \includegraphics[scale=.8,valign=c]{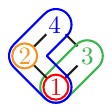}
.
\]
Note that~$O_1 = \orn{\reori{34251}}$ and~$O_2 = \orn{\reori{3412}}$ so that they both fulfill the conditions of \cref{lem:AOrn2}.
However,
\begin{itemize}
\item the reorientation~$\reori{O_1}$ is cyclic since it contains the cycle~$1, 4, 2, 5, 1$,
\item the sourcing~$\sour{O_2}$ is cyclic since~$\sour{O_2}(\{3,4\}) = 3$ and~$\sour{O_2}(\{1,3,4\}) = 4$.
\end{itemize}
\end{remark}

\begin{definition}
\label{def:AOrn}
An ornamentation~$O$ of~$D$ is \defn{acyclic} if the equivalent conditions of \cref{lem:AOrn2} are satisfied.
The \defn{acyclic ornamentation poset}~$\AOrn(D)$ is the subposet of the ornamentation lattice~$\Orn(D)$ induced by acyclic ornamentations.
\end{definition}

\begin{remark}
\label{rem:COrn}
We will see in \cref{subsec:unstarredTrees} that all the ornamentations of an unstarred tree are acyclic.
This is false in general, even for increasing trees.
For instance, the cyclic ornamentations of the directed graphs of \cref{fig:ornamentationsX,fig:ornamentationsD,fig:ornamentationsKT} are
\[
\includegraphics[scale=.8,valign=c]{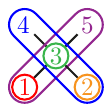}\;\;
\includegraphics[scale=.8,valign=c]{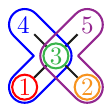}\;\;
\includegraphics[scale=.8,valign=c]{cyclicOrnamentationD1}
\includegraphics[scale=.8,valign=c]{cyclicOrnamentationD2}\;\;
\includegraphics[scale=.8,valign=c]{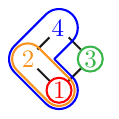}
\includegraphics[scale=.8,valign=c]{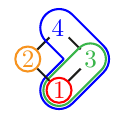}\;\;
\includegraphics[scale=.8,valign=c]{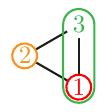}\;\;
\includegraphics[scale=.8,valign=c]{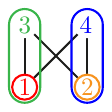}\;\;
\includegraphics[scale=.8,valign=c]{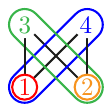}
\]
since their fibers under~$R \mapsto \orn{R}$ are all singletons, respectively containing the cyclic reorientations
\[
\includegraphics[scale=.8,valign=c]{cyclicReorientationX1}\;\;
\includegraphics[scale=.8,valign=c]{cyclicReorientationX2}\;\;
\includegraphics[scale=.8,valign=c]{cyclicReorientationD1}\;
\includegraphics[scale=.8,valign=c]{cyclicReorientationD2}\;
\includegraphics[scale=.8,valign=c]{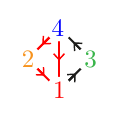}\;
\includegraphics[scale=.8,valign=c]{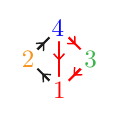}\;
\includegraphics[scale=.8,valign=c]{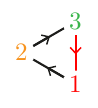}\;\;
\includegraphics[scale=.8,valign=c]{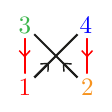}\;\;
\includegraphics[scale=.8,valign=c]{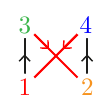}
\]
The acyclic ornamentations are shaded in gray in \cref{fig:ornamentationsX,fig:ornamentationsD,fig:ornamentationsKT}.
\end{remark}

\begin{remark}
\label{rem:AOrnLat}
We will see in \cref{subsec:unstarredTrees} that the acyclic ornamentation poset of an unstarred tree  is a lattice.
This is false in general, even for increasing trees.
For instance, the acyclic ornamentation posets of the directed graphs~$\Xgraph$ of \cref{fig:ornamentationsX}, $\Dgraph$ of \cref{fig:ornamentationsD}, and $\Tgraph$ of \cref{fig:ornamentationsKT} are not lattices.
In contrast, the acyclic ornamentation poset of the directed graph~$\Kgraph$ of \cref{fig:ornamentationsKT} is a lattice.
We do not have a complete characterization of the directed graphs whose acyclic ornamentation poset is a lattice.
\end{remark}

\begin{definition}
\label{def:moreNotations}
To keep our notations coherent, we denote by~$\aorn{R} \eqdef \orn{R}$ the acyclic ornamentation of~$D$ corresponding to an acyclic reorientation~$R$ of~$\tc(D)$, and by~$\aorn{S} \eqdef \orn{S}$ the acyclic ornamentation of~$D$ corresponding to an acyclic sourcing~$S$ of~$\PP(D)$.
Conversely, for an acyclic ornamentation~$O$ of~$D$, we denote by~$\asour{O}$ the corresponding acyclic sourcing of~$\PP(D)$ (\cref{lem:ASour2AOrn1}), and we define an acyclic reorientation of~$\tc(D)$ by~$\areori{O} \eqdef \areori{\asour{O}}$.
Be careful that~$\asour{O} \ne \sour{O}$ and~$\areori{O} \ne \reori{O}$ in general (as~$\sour{O}$ and~$\reori{O}$ might not be acyclic even if~$O$ is acyclic, see \cref{rem:Orn2Sour&Orn2ReoriIncompatibleAcyclicity}).
\end{definition}

\begin{proposition}
\label{prop:ASour2AOrn}
The map~$S \mapsto \aorn{S}$ is a poset isomorphism from the acyclic sourcing \linebreak poset~$\ASour(\PP(D))$ to the acyclic ornamentation poset~$\AOrn(D)$.
\end{proposition}

\begin{proof}
The map~$S \mapsto \aorn{S}$ is injective by \cref{lem:ASour2AOrn1} and surjective by \cref{def:AOrn}.
Moreover, it is order-preserving as it is the restriction of the map $S \to \orn{S}$, which is order-preserving by \cref{lem:Sour2Orn3}.
We are only left to show that $S \mapsto \aorn{S}$ is order-reflecting.

Let $S_1$ and~$S_2$ be two acyclic sourcings of~$\PP(D)$ such that~$\aorn{S_1} \leq \aorn{S_2}$.
Suppose that~$S_1 \not\leq S_2$ and let~$P \in \PP(D)$ such that~$u \eqdef S_2(P) < S_1(P) \defeq v$.
By acyclicity of~$S_1$ (resp.~$S_2$) we can assume that~$P$ ends at~$v$ (resp.~begins at~$u$).
The acyclicity of~$S_1$ further implies that~$S_1(Q) = v$ for any subpath~$Q$ of~$P$ ending at~$v$.
Hence,~$u \in P \subseteq \aorn{S_1}(P) \subseteq \aorn{S_2}(P)$.
But then, there exists a path~$P'$ from~$u$ to~$v$ such that~$S_2(P') = v$, contradicting the acyclicity of~$S_2$.
Thus, we necessarily have~$S_1 \leq S_2$.
\end{proof}

%%%%%%%%%%%%%%%%%%%%%%%%%%%%%%%%%%%%%%
%%%%%%%%%%%%%%%%%%%%%%%%%%%%%%%%%%%%%%
%%%%%%%%%%%%%%%%%%%%%%%%%%%%%%%%%%%%%%

\section{Directed trees}
\label{sec:trees}

In this section, we focus on the specific case when the underlying directed graph~$D$ is a directed tree~$T$ (but not necessarily a rooted tree), as for example in \cref{fig:ornamentationsNI,fig:ornamentationsAY,fig:ornamentationsX}.
We start with a few basic observations (\cref{subsec:basicObservationsT}).
We then describe the join and meet irreducible ornamentations of~$T$ and prove that the ornamentation lattice~$\Orn(T)$ is semidistributive (\cref{subsec:semidistributivityT}).
We then revisit some structural properties of the reorientations of~$\tc(T)$ (\cref{subsec:maximalReorientationT,subsec:transitivelyBiclosedReorientationsT}).
Finally, we describe the MacNeille completions of the acyclic reorientation poset~$\AReori(\tc(T))$ and of the acyclic sourcing poset~$\ASour(\PP(T))$ (\cref{subsec:MacNeilleT}).

%%%%%%%%%%%%%%%%%%%%%%%%%%%%%%%%%%%%%%

\subsection{Basic observations}
\label{subsec:basicObservationsT}

Let~$T$ be a directed tree on~$V$.
For~$u,v \in V$, we write~$u \le_T v$ if the unique path between~$u$ and~$v$ in~$T$ is a directed path from~$u$ to~$v$, and we denote the vertices along this path by~$[u,v]_T \eqdef \set{w \in T}{u \le_T w \le_T v}$.
We let~$\lessin{T}{v} \eqdef \set{u \in V}{u \le_T v}$ denote the set of vertices which admit a directed path towards~$v$ in~$T$.
We begin with a few observations.

\begin{lemma}
\label{lem:meetOrnT}
If~$T$ is a directed tree, and~$U,U' \in \Orn(v \in T)$, then~$U \cap U' \in \Orn(v \in T)$.
\end{lemma}

\begin{proof}
In a directed tree~$T$, a nonempty subset $U \subseteq \lessin{T}{v}$ is an ornament of~$T$ at~$v$ if and only if for all $u \in U$, the (vertex set of the) unique path from $u$ to $v$ in $T$ is contained in $U$.
This property is closed under intersection.
\end{proof}

\begin{corollary}
\label{coro:meetOrnT}
The meet of two ornamentations~$O_1$ and~$O_2$ of a directed tree~$T$ is given by
\[
(O_1 \meet O_2)(v) = O_1(v) \cap O_2(v).
\]
\end{corollary}

We now describe the cover relations in the ornamentation lattice~$\Orn(T)$.

\begin{lemma}
\label{lem:coverRelationsOrnT}
Two ornamentations~$O_1$ and~$O_2$ of a directed tree~$T$ form a cover relation~$O_1 \lessdot O_2$ if and only if there are~$u, v \in V$ such that~$u \notin O_1(v)$ and $O_2(v) = O_1(u) \cup O_1(v)$ while~$O_1(w) = O_2(w)$ for all~${w \in V \ssm \{v\}}$.
Moreover, $u$ and~$v$ are uniquely determined by~$O_1$ and~$O_2$.
\end{lemma}

\begin{proof}
Assume first that~$O_1$ and~$O_2$ are two ornamentations of~$T$ such that there are~$u, v \in V$ with~$u \notin O_1(v)$ and~$O_2(v) = O_1(u) \cup O_1(v)$ while~$O_1(w) = O_2(w)$ for all~${w \in V \ssm \{v\}}$.
We clearly have~$O_1 \le O_2$.
Moreover, consider any ornamentation~$O$ of~$T$ such that~$O_1 \le O \le O_2$.
Then~$O_1(w) = O(w) = O_2(w)$ for all~${w \in V \ssm \{v\}}$ and~$O_1(v) \subseteq O(v) \subseteq O_2(v)$.
If~$O_1 \ne O$, then there is~$x \in O(v) \ssm O_1(v)$.
Note that~$x \in O(v) \ssm O_1(v) \subseteq O_2(v) \ssm  O_1(v) = O_1(u)$.
Hence, the directed path from~$x$ to~$v$ in~$T$ passes through~$u$, and we have~$u \in O(v)$.
We conclude that~$O_1(u) = O(u) \subseteq O(v)$, and thus that~$O = O_2$.

Conversely, assume that~$O_1$ and~$O_2$ are two ornamentations of~$T$ such that~$O_1 \lessdot O_2$ is a cover relation in~$\Orn(T)$.
Consider~$v \in V$ such that~$O_1(v) \ne O_2(v)$ and~$O_1(w) = O_2(w)$ for any~$w \in O_1(v)$.
Then the map~$O$ on~$V$ defined by $O(v) = O_1(v)$ and~$O(w) = O_2(w)$ for all~$w \ne v$ is an ornamentation of~$T$.
Moreover, $O_1 \le O < O_2$.
Hence, $O_1 = O$.
Consider now~$u$ maximal in~$T$ such that~$u \in O_2(v) \ssm O_1(v)$.
Then the map~$O$ on~$V$ defined by~$O(v) = O_1(v) \cup O_1(u)$ and~$O(w) = O_1(w)$ for all~$w \ne v$ is an ornamentation of~$T$.
Moreover, $O_1 < O \le O_2$.
Hence, $O_2 = O$.
We have thus found~$u$ and~$v$ as required, and they are determined by~$O_1$ and~$O_2$.
\end{proof}

\begin{remark}
\label{rem:coverRelationsOrn}
The description of \cref{lem:coverRelationsOrnT} is not sufficient to have a cover relation of the ornamentation lattice~$\Orn(D)$ of any directed graph~$D$.
See \eg \cref{fig:ornamentationsKT}\,(left).
The general description is slightly more involved.
Namely, two ornamentations~$O_1$ and~$O_2$ of~$D$ form a cover relation~$O_1 \lessdot O_2$ if and only if there are~$u, v \in V$ such that~$u \notin O_1(v)$, $O_2(v) = O_1(u) \cup O_1(v)$ while~$O_1(w) = O_2(w)$ for all~${w \in V \ssm \{v\}}$, and~$O_1(w) = O_1(u)$ for any~$w \in O_1(u)$ with an edge towards~$O_1(v)$. A proof may be found in~\cite{Sack}.
\end{remark}

Finally, we will need the following convenient statement.

\begin{lemma}
\label{lem:Reori2OrnT}
For a directed tree~$T$, an acyclic reorientation~$R$ of~$\tc(T)$ and vertices $u \le_T v$, we have~$u \in \orn{R}(v)$ if and only if~$(u',v) \in \rev(R)$ for all~$u \le_T u' <_T v$.
\end{lemma}

\begin{proof}
Since~$R$ is acyclic, for any $u <_T v$, there is a directed path from~$u$ to~$v$ in~$\rev(R)$ if and only if~$(u,v) \in \rev(R)$.
By \cref{def:Reori2Orn}, $\orn{R}(v)$ is the maximal ornament of~$T$ at~$v$ contained in the subset of vertices~$u \in V$ with a directed path to~$v$ in $\rev(R)$.
Hence, $u \in \orn{R}(v)$ if and only if~$u' \in \rev(R)$ for all~$u'$ along the path from~$u$ to~$v$ in~$T$.
\end{proof}

%%%%%%%%%%%%%%%%%%%%%%%%%%%%%%%%%%%%%%

\subsection{Semidistributivity and canonical join representations in the ornamentation lattice}
\label{subsec:semidistributivityT}

Let us start with a quick recollection on semidistributivity.
We consider a finite lattice~$(L, \le, \meet, \join)$.

\begin{definition}
An element~$x \in L$ is called \defn{join} (resp.~\defn{meet}) \defn{irreducible} if it covers (resp.~is covered by) a unique element denoted~$x_\star$ (resp.~$x^\star$).
We denote by $\JI$ (resp.~$\MI$) the subposet of~$L$ induced by the set of join (resp.~meet) irreducible elements of~$L$.
\end{definition}

\begin{definition}
A \defn{join representation} of~$x \in L$ is a subset~${J \subseteq L}$ such that~$x = \bigJoin J$.
Such a representation is \defn{irredundant} if~$x \ne \bigJoin J'$ for any strict subset~$J' \subsetneq J$.
The irredundant join representations in~$L$ are antichains of~$L$, and are ordered by containment of the lower sets of their elements (\ie~$J \le J'$ if and only if for any~$y \in J$ there exists~$y' \in J'$ such that~$y \le y'$ in~$L$).
The \defn{canonical join representation} of~$x$, denoted~$\CJR(x)$, is the minimal irredundant join representation of~$x$ for this order, when it exists.
\end{definition}

Note that when it exists, $\CJR(x)$ is an antichain of~$\JI$.
The following statement characterizes the lattices where canonical join representations exist.

\begin{proposition}[{\cite[Thm.~2.24 \& Thm.~2.56]{FreeseNation}}]
\label{prop:semidistributive}
A finite lattice~$L$ is \defn{join semidistributive} when the following equivalent conditions hold:
\begin{enumerate}[(i)]
\item $x \join y = x \join z$ implies $x \join (y \meet z) = x \join y$ for any~$x, y, z \in L$,
\item for any cover relation~$x \lessdot y$ in~$L$, the set \[K_\join(x,y) \eqdef \set{z \in L}{z \not\le x \text{ but } z \le y} = \set{z \in L}{x \join z = y}\] has a unique minimal element~$k_\join(x,y)$ (which is then automatically join irreducible),
\item any element of~$L$ admits a canonical join representation.
\end{enumerate}
Moreover, 
\begin{itemize}
\item the join irreducible elements of~$L$ are precisely the elements~$k_\join(x,y)$ for all cover relations~$x \lessdot y$ in~$L$,
\item the canonical join representation of~$y \in L$ is~$\CJR(y) = \set{k_\join(x, y)}{x \lessdot y}$.
\end{itemize}
\end{proposition}

Note that in a finite join semidistributive lattice~$L$, we can associate to any meet irreducible element~$m$ of~$L$ a join irreducible element~$\kappa_\join(m) \eqdef k_\join(m, m^\star)$ of~$L$.

The \defn{meet semidistributivity} property, the maps~$K_\meet$, $k_\meet$ and~$\kappa_\meet$, and the \defn{canonical meet representation}~$\CMR(x)$ are all defined dually.
A lattice~$L$ is \defn{semidistributive} if it is both meet and join semidistributive.
In this case, the maps~$\kappa_\join$ and~$\kappa_\meet$ define inverse bijections between~$\MI$ and~$\JI$.

\medskip
We now consider semidistributivity in ornamentation lattices.
We first describe the join and meet irreducible ornamentations of a directed tree~$T$.
Recall that for~$v \in V$, we denote by $\lessin{T}{v} \eqdef \set{u \in V}{u \le_T v}$ the set of vertices which admit a directed path towards~$v$ in~$T$.

\begin{definition}
\label{def:irreducibleOrnamentationsT}
For a directed path~$P$ from~$u$ to~$v$ in a directed tree~$T$, we denote by~$J_P$ and~$M_P$ the two maps on~$V$ defined by
\[
J_P(w) \eqdef \begin{cases} P & \text{if } w = v \\ \{w\} & \text{otherwise} \end{cases}
\qquad\text{and}\qquad
M_P(w) \eqdef \begin{cases} \lessin{T}{w} \ssm \lessin{T}{u} & \text{if } u <_T w \le_T v \\ \lessin{T}{w} & \text{otherwise} \end{cases}
.
\]
\end{definition}

\begin{lemma}
For any~$P \in \PP(T)$, the maps~$J_P$ and~$M_P$ are both ornamentations of~$T$.
\end{lemma}

\begin{proof}
The claim for~$J_P$ follows from \cref{exm:ornamentations}\,(iii).
Consider now~$u', v' \in V$ such that~$u' \in M_P(v')$.
If~$u <_T v' \le_T v$, then~$M_P(u') = \lessin{T}{u'} \ssm \lessin{T}{u} \subseteq \lessin{T}{v'} \ssm \lessin{T}{u} = M_P(v')$ (observe that the first equality holds even if $u \not\le_T u'$, since in this case $\lessin{T}{u'} \ssm \lessin{T}{u} = \lessin{T}{u'}$).
Otherwise, $M_P(u') \subseteq \lessin{T}{u'} \subseteq \lessin{T}{v'} = M_P(v')$.
\end{proof}

\begin{lemma}
\label{lem:semidistributive}
Let~$O_1, O_2$ be two ornamentations of a directed tree~$T$ which form a cover relation~$O_1 \lessdot O_2$, let~$u,v \in V$ such that~$u \notin O_1(v)$ and $O_2(v) = O_1(u) \cup O_1(v)$ while~$O_1(w) = O_2(w)$ for all~${w \in V \ssm \{v\}}$ (see \cref{lem:coverRelationsOrnT}), and let~$P$ denote the directed path from~$u$ to~$v$ in~$T$.
Then~$J_P$ (resp.~$M_P$) is the unique minimal element of~$\set{O \in \Orn(T)}{O \not\le O_1 \text{ but } O \le O_2}$ (resp.~maximal element of~$\set{O \in \Orn(T)}{O_1 \le O \text{ but } O_2 \not\le O}$).
\end{lemma}

\begin{proof}
Observe first that~$J_P \not\le O_1$ (since~$u \in P = J_P(v)$ while~$u \notin O_1(v)$) but~$J_P \le O_2$ (since~${P \subseteq O_2(v)}$). 
Consider now an ornamentation~$O$ of~$T$ such that~$O \not\le O_1$ but~$O \le O_2$.
Since~$O_1(w) = O_2(w)$ for~$w \ne v$, we must have~$u \in O(v)$, hence~$P \subseteq O(v)$.
We conclude that~$J_P \le O$.

Observe now that~$O_1 \le M_P$ (as~$u \notin O_1(v)$, we have~$\lessin{T}{u} \cap O_1(w) = \varnothing$ for all~$u <_T w \le_T v$) but~$O_2 \not\le M_P$ (because~$u \in O_2(v)$ while~$u \notin M_P(v)$).
Consider now an ornamentation~$O$ of~$T$ such that~$O_1 \le O$ but~$O_2 \not\le O$.
Note that~$O_2(w) = O_1(w) \subseteq O(w)$ for all~$w \ne v$.
Assume that~$u \in O(w)$ for some~$u <_T w \le_T v$.
As~$w \in O_1(v) \subseteq O(w)$, we have~$u \in O(v)$, hence~$O_1(u) \subseteq O(u) \subseteq O(v)$.
We thus obtain that~$O_2(v) = O_1(u) \cup O_1(v) \subseteq O(v)$.
This contradicts~$O_2 \not\le O$.
We thus obtained that~$u \notin O(w)$, hence~$\lessin{T}{u} \cap O(w) = \varnothing$ for all~$u <_T w \le_T v$.
We conclude that~$O \le M_P$.
\end{proof}

\begin{theorem}
\label{thm:OrnSemidistributiveT}
For a directed tree~$T$, 
\begin{itemize}
\item the ornamentation lattice~$\Orn(T)$ is semidistributive,
\item the join (resp.~meet) irreducible ornamentations of~$T$ are precisely the ornamentations~$J_P$ (resp.~$M_P$) for all directed paths~$P$ in~$T$,
\item the canonical join (resp.~meet) representation of an ornamentation~$O$ is~$O = \bigJoin_P J_P$ (resp.~$O = \bigMeet_P M_P$) where the join (resp.~meet) ranges over all paths~$P = [u,v]$ described in \cref{lem:coverRelationsOrnT} for the cover relations~$O' \lessdot O$ (resp.~$O \lessdot O'$),
\item we have~$\kappa_\join(M_P) = J_P$ and~$\kappa_\meet(J_P) = M_P$ for any~$P \in \PP(T)$.
\end{itemize}
\end{theorem}

\begin{proof}
This directly follows from \cref{prop:semidistributive,lem:semidistributive}.
\end{proof}

\begin{remark}
Observe that the ornamentation lattice~$\Orn(\Dgraph)$ of the diamond graph~$\Dgraph$ represented in \cref{fig:ornamentationsD} is neither join nor meet semidistributive, since
\begin{multline*}
\includegraphics[scale=.8,valign=c]{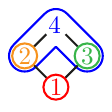} \join \includegraphics[scale=.8,valign=c]{cyclicOrnamentationD1} = \includegraphics[scale=.8,valign=c]{ornamentationD1} \join \includegraphics[scale=.8,valign=c]{cyclicOrnamentationD2} = \includegraphics[scale=.8,valign=c]{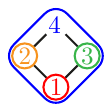}
\\[-.2cm]
 \ne \includegraphics[scale=.8,valign=c]{ornamentationD1} = \includegraphics[scale=.8,valign=c]{ornamentationD1} \join \includegraphics[scale=.8,valign=c]{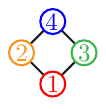} = \includegraphics[scale=.8,valign=c]{ornamentationD1} \join \Big( \includegraphics[scale=.8,valign=c]{cyclicOrnamentationD1} \meet  \includegraphics[scale=.8,valign=c]{cyclicOrnamentationD2} \Big)
\end{multline*}
and
\begin{multline*}
\includegraphics[scale=.8,valign=c]{cyclicOrnamentationD1} \meet \includegraphics[scale=.8,valign=c]{ornamentationD1} = \includegraphics[scale=.8,valign=c]{cyclicOrnamentationD1} \meet \includegraphics[scale=.8,valign=c]{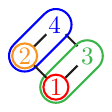} = \includegraphics[scale=.8,valign=c]{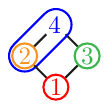}
\\[-.2cm]
 \ne \includegraphics[scale=.8,valign=c]{cyclicOrnamentationD1} = \includegraphics[scale=.8,valign=c]{cyclicOrnamentationD1} \meet \includegraphics[scale=.8,valign=c]{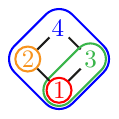} = \includegraphics[scale=.8,valign=c]{cyclicOrnamentationD1} \meet \Big( \includegraphics[scale=.8,valign=c]{ornamentationD1} \join  \includegraphics[scale=.8,valign=c]{ornamentationD4} \Big)
\end{multline*}
\end{remark}

%%%%%%%%%%%%%%%%%%%%%%%%%%%%%%%%%%%%%%

\subsection{Maximal reorientation}
\label{subsec:maximalReorientationT}

We now revisit \cref{rem:Reori2OrnIntervals1,rem:Reori2OrnIntervals2} for directed trees.

\begin{definition}
\label{def:Reori2OrnMaxT}
Consider an ornamentation~$O$ of a directed tree~$T$.
Let~$\maxreori{O}$ denote the reorientation of~$\tc(T)$ where for each path~$u, u', \dots, v$ in~$T$ (note that $u'$ might coincide with~$v$), the edge~$(u,v)$ is in~$\rev(\maxreori{O})$ if and only if there is no~$w \in V$ such that~$u \notin O(w)$ while~$u', v \in O(w)$.
\end{definition}

\begin{proposition}
\label{prop:Reori2OrnMaxTree}
For any ornamentation~$O$ of a directed tree~$T$, the reorientation~$\maxreori{O}$ is the maximum reorientation~$R$ of~$\tc(T)$ with~$\orn{R} = O$.
\end{proposition}

\begin{proof}
We first prove that~$\orn{\maxreori{O}} = O$.
Fix~$u,v \in V$.
Observe first that~$u \in O(v)$ implies~$(u,v) \in \rev(\maxreori{O})$, since for any~$w \in V$, the fact that~$v \in O(w)$ implies~$u \in O(v) \subseteq O(w)$. Hence, by definition,~$u \in \orn{\maxreori{O}}(v)$.
Conversely, assume that~$u \notin O(v)$ and consider such an~$u$ which is the closest to~$v$.
Consider the path~$u, u', \dots, v$ in~$T$.
For any~$v'$ between~$u'$ and~$v$, we have~$u \notin O(v)$ while~$u', v' \in O(v)$, which implies that~$(u,v') \notin \rev(\maxreori{O})$.
Hence, there is no path from~$u$ to~$v$ in~$\rev(\maxreori{O})$, and thus~$u \notin \orn{\maxreori{O}}(v)$.
We conclude that~$u \in O(v) \iff u \in \orn{\maxreori{O}}(v)$, that is~$O = \orn{\maxreori{O}}$.

Consider now any reorientation~$R$ of~$\tc(T)$ such that~$\orn{R} = O$.
Assume that there is~$u,v \in V$ such that~$(u,v) \in \rev(R) \ssm \rev(\maxreori{O})$.
Consider the path~$u, u', \dots, v$ in~$T$, and let~$w \in V$ be such that~$u \notin O(w)$ while~$u', v \in O(w)$.
Since~$\orn{R} = O$ and~$v \in O(w)$, there is a path from~$v$ to~$w$ in~$\rev(R)$, hence from~$u$ to~$v$ in~$\rev(R)$ since~$(u,v) \in \rev(R)$.
Moreover, since~$(u,u') \in T$ and~$u' \in O(w)$, the set~$O(w) \cup \{u\}$ is an ornament of~$T$ at~$w$.
Hence, $O(w) \cup \{u\}$ contradicts the maximality of~$\orn{R}(w) = O(w)$.
We conclude that if~$\orn{R} = O$, then~$\rev(R) \subseteq \rev(\maxreori{O})$, hence~$R \le \maxreori{O}$.
\end{proof}

\begin{proposition}
\label{prop:ReoriTC2OrnIntervalT}
For any ornamentation~$O$ of a directed tree~$T$, the reorientation~$\maxreori{O}$  of~\cref{def:Reori2OrnMaxT} is transitively closed.
Hence, the transitively closed reorientations~$R$ of~$\tc(T)$ with~${\orn{R} = O}$ form an interval of the transitively closed reorientation lattice~$\Rcl(\tc(T))$.
\end{proposition}

\begin{proof}
Assume that there is a directed path~$u, u', \dots, v$ in~$T$ such that~$(u,v) \notin \rev(\maxreori{O})$.
Let~$w \in V$ be such that $u \notin O(w)$ but~$u',v \in O(w)$.
Then, for any~$v'$ between~$u'$ and~$v$, we have~$v' \in O(w)$, hence~$(u,v') \not \in \rev(\maxreori{O})$.
We thus obtained that~$(u,v'), (v',v) \in \rev(\maxreori{O})$ implies~$(u,v) \in \rev(\maxreori{O})$, that is, $\rev(\maxreori{O})$ is transitively closed.
We conclude that the fiber of~$O$ is an interval, since it has a minimum by \cref{prop:Reori2OrnMinTC} and a maximum by \cref{prop:Reori2OrnMaxTree}.
\end{proof}

%%%%%%%%%%%%%%%%%%%%%%%%%%%%%%%%%%%%%%

\subsection{Transitively biclosed reorientations}
\label{subsec:transitivelyBiclosedReorientationsT}

We now revisit \cref{rem:biclosedNotLattice,rem:Orn2ReoriNotTransitivelyCoclosed} for directed trees.

\begin{proposition}
\label{prop:biclosedLatticeT}
For a directed tree~$T$, the subposet of~$\Reori(\tc(T))$ induced by the set~$\Rbi(\tc(T))$ of transitively biclosed reorientations of~$T$ is a lattice.
\end{proposition}

\begin{proof}
Consider two transitively biclosed reorientations~$R_1$ and~$R_2$ of~$\tc(T)$.
Let~$R$ be the reorientation of~$\tc(T)$ defined by~$\rev(R) = \tc(\rev(R_1) \cup \rev(R_2))$.

The reorientation~$R$ is closed by definition, and we claim that it is also coclosed.
Indeed, we consider~$u,v,w \in V$ such that~$(u,v), (v,w) \in \tc(T)$ and~$(u,w) \in \rev(R)$, and we will show that either~$(u,v) \in \rev(R)$ or~$(v,w) \in \rev(R)$.
By definition of~$\rev(R)$, there is a path~${u = u_0, \dots, u_k = w}$ such that~$(u_{i-1}, u_i) \in \rev(R_1) \cup \rev(R_2)$ for all~$i \in [k]$.
As~$(u,v), (v,w) \in \tc(T)$ and~$T$ is a tree, there is~$i \in [k]$ such that~$u_{i-1} \le_T v \le_T u_i$.
If $v = u_{i-1}$ or $v = u_i$, the claim follows. 
Otherwise, assume by symmetry that~$(u_{i-1}, u_i) \in \rev(R_1)$.
Since~$R_1$ is coclosed, either~$(u_{i-1}, v) \in \rev(R_1)$ or~$(v, u_i) \in \rev(R_1)$.
We thus obtain that either~$(u,v) \in \rev(R)$ or~$(v,w) \in \rev(R)$, hence that~$R$ is indeed coclosed.

Since any transitively closed reorientation~$R'$ such that~$R' > R_1$ and~$R' > R_2$ clearly satisfies~$R' \ge R$, we conclude that~$R$ is the join of~$R_1$ and~$R_2$ among transitively biclosed reorientations.
Thus, $\Rbi(\tc(T))$ is a finite bounded join semilattice, hence a lattice.
\end{proof}

\begin{lemma}
\label{lem:Orn2ReoriT}
For an ornamentation~$O$ of a directed tree~$T$, the reorientation~$\reori{O}$ is transitively biclosed.
\end{lemma}

\begin{proof}
We have already seen in \cref{prop:ReoriTC2Orn} that~$\reori{O}$ is transitively closed.
Let~$u, v, w \in V$ be such that~$(u,v)$ and~$(v,w)$ are in~$\tc(T)$,
and that~$(u,w) \in \rev(\reori{O})$. Then~$u \in O(w)$. As~$T$ is a tree, the complete path from~$u$ to~$w$ is in~$O(w)$, hence~$v \in O(w)$ and thus $(v,w) \in \rev(\reori{O})$.
We conclude that~$\reori{O}$ is transitively coclosed.
\end{proof}

%%%%%%%%%%%%%%%%%%%%%%%%%%%%%%%%%%%%%%

\subsection{MacNeille completions}
\label{subsec:MacNeilleT}

We first briefly recall the definition of the MacNeille completion of a poset.

\begin{definition}
\label{def:MacNeilleCompletion}
A \defn{completion} of a finite poset~$P$ is a lattice~$L$ such that there is an order-embedding~$f : P \to L$ (meaning that $f$ is injective and~$x \le_P y$ if and only if~$f(x) \le_L f(y)$).
The MacNeille completion~$C(P)$ of~$P$ is the smallest completion of~$P$, meaning that for any completion~$L$ of~$P$, there exists an order-embedding from~$C(P)$ to~$L$.
\end{definition}

\begin{proposition}[{\cite[Thm.~7.42]{DaveyPriestley}}]
\label{prop:MacNeilleCompletion}
Any finite lattice~$L$ is the MacNeille completion of its subposet induced by~$\JI \cup \MI$.
\end{proposition}

We now come back to ornamentation lattices.
As already observed on the directed tree~$\Xgraph$ of \cref{fig:ornamentationsX},
\begin{itemize}
\item there might be some transitively biclosed but cyclic reorientations of~$\tc(T)$ (\cref{rem:biclosedCyclic}),
\item the acyclic reorientation poset~$\AReori(\tc(T))$ is not always a lattice (\cref{prop:AReoriLat}),
\item there might be some cyclic ornamentations of~$T$ (\cref{rem:COrn}),
\item the acyclic ornamentation poset~$\AOrn(T)$ is not always a lattice (\cref{rem:AOrnLat}).
\end{itemize}
However, we now observe that for a directed tree~$T$, the transitively biclosed reorientation lattice~$\Rbi(\tc(T))$ (resp.~the ornamentation lattice~$\Orn(T)$) is precisely the MacNeille completion of the acyclic reorientation poset~$\AReori(\tc(T))$ (resp.~of the acyclic ornamentation lattice~$\AOrn(T)$).
For this, we need the following lemmas.

\begin{lemma}
\label{lem:irreduciblesReorientationsAcyclicT}
If~$R$ is a transitively biclosed and cyclic reorientation of~$\tc(T)$, then it covers (resp.~is covered by) at least two biclosed reorientations of~$\tc(T)$.
\end{lemma}

\begin{proof}
As~$R$ contains a cycle, it contains an induced oriented cycle~$C$ (any chord can be used to make the cycle shorter).
Since~$R$ is transitively biclosed, $C$ is an induced alternating cycle in~$\tc(T)$.
Write~$C = (u_1, v_1, \dots, u_k, v_k)$ where~$(u_i, v_i)$ and~$(u_{i+1},v_i)$ are in~$\tc(T)$ for all~$i \in [k]$.
We can moreover assume without loss of generality that for~$i \in [k]$ and any outgoing neighbor~$u'$ of~$u_i$ (resp.~incoming neighbor~$v'$ of~$v_i$) in~$R$, replacing~$u_i$ by~$u'$ (resp.~$v_i$ by~$v'$) in~$C$ does not yield an induced oriented cycle in~$R$.
For~$i \in [k]$, consider the reorientation~$R_i$ of~$\tc(T)$ obtained by reversing the edge~$(u_i,v_i)$, meaning that~$\rev(R_i) = \rev(R) \cup \{(u_i,v_i)\}$.
The reorientation~$R_i$ clearly covers~$R$, and we claimed that is still biclosed.
Indeed,
\begin{itemize}
\item If~$R_i$ were not closed, we would have~$u' \in V$ with~$(u',u_i) \in \rev(R)$ and~$(u',v_i) \notin \rev(R)$ or~$v' \in V$ with~$(v_i, v') \in \rev(R)$ and~$(u_i,v') \notin \rev(R)$. We could then replace~$u_i$ by~$u'$ or~$v_i$ by~$v'$ in~$C$ and obtained another induced oriented cycle in~$R$.
\item If~$R_i$ were not coclosed, we would have~$u' \in V$ with~$(u_i,u') \in \tc(T) \ssm \rev(R)$ and~$(u',v_i) \in \tc(T) \ssm \rev(R)$. We could then replace~$u_i$ by~$u'$ in~$C$ and obtained another induced oriented cycle in~$R$.
\end{itemize}
In both cases, this contradicts our assumption on~$C$.
Hence, we obtained~$k$ biclosed reorientations of~$\tc(T)$ covering~$R$.
Symmetrically, we can construct a biclosed reorientations~$R_i$ of~$\tc(T)$ covered by~$R$ by reversing the arrow~$(u_{i+1}, v_i)$ for any~$i \in [k]$.
\end{proof}

\begin{lemma}
\label{lem:irreduciblesOrnamentationsAcyclicT}
For any~$P \in \PP(T)$, the reorientations~$\reori{J_P}$ and~$\reori{M_P}$ of~$\tc(T)$ are both acyclic.
\end{lemma}

\begin{proof}
Let~$u \le_T v$ be the endpoints of the directed path~$P$.

Observe that~$\rev(\reori{J_P})= (P \ssm \{v\}) \times \{v\}$.
Hence, any induced cycle in~$\reori{J_P}$ must consist of an edge~$(v,y)$ and a directed path from~$y$ to~$v$ in~$T$ for some~$y \in P$.
This is impossible as the last edge of this path would be in~$\rev(\reori{J_P})$.

Observe that~$\tc(T) \ssm \rev(\reori{M_P}) = \lessin{T}{u} \times (P \ssm \{u\})$.
Any cycle in~$\reori{M_P}$ should have an edge of $\tc(T) \ssm \rev(\reori{M_P})$, hence a vertex in~$\lessin{T}{u}$ and a vertex in~$P \ssm \{u\} \subseteq T_{\le v} \ssm T_{\le u}$.
This is impossible as there is no edge of~$\reori{M_P}$ leaving~$T_{\le v} \ssm T_{\le u}$.
\end{proof}

\begin{corollary}
\label{coro:irreduciblesAcyclicT}
For a directed tree~$T$, 
\begin{itemize}
\item all join (resp.~meet) irreducible transitively biclosed reorientations of~$\tc(T)$ are acyclic,
\item all join (resp.~meet) irreducible ornamentations of~$T$ are acyclic.
\end{itemize}
\end{corollary}

\begin{proof}
It follows from \cref{lem:irreduciblesReorientationsAcyclicT} for transitively biclosed reorientations, and from \cref{lem:irreduciblesOrnamentationsAcyclicT,thm:OrnSemidistributiveT} for ornamentations.
\end{proof}

We conclude with the main statement of this section.
\begin{figure}[b]
	\centerline{\includegraphics[scale=.68]{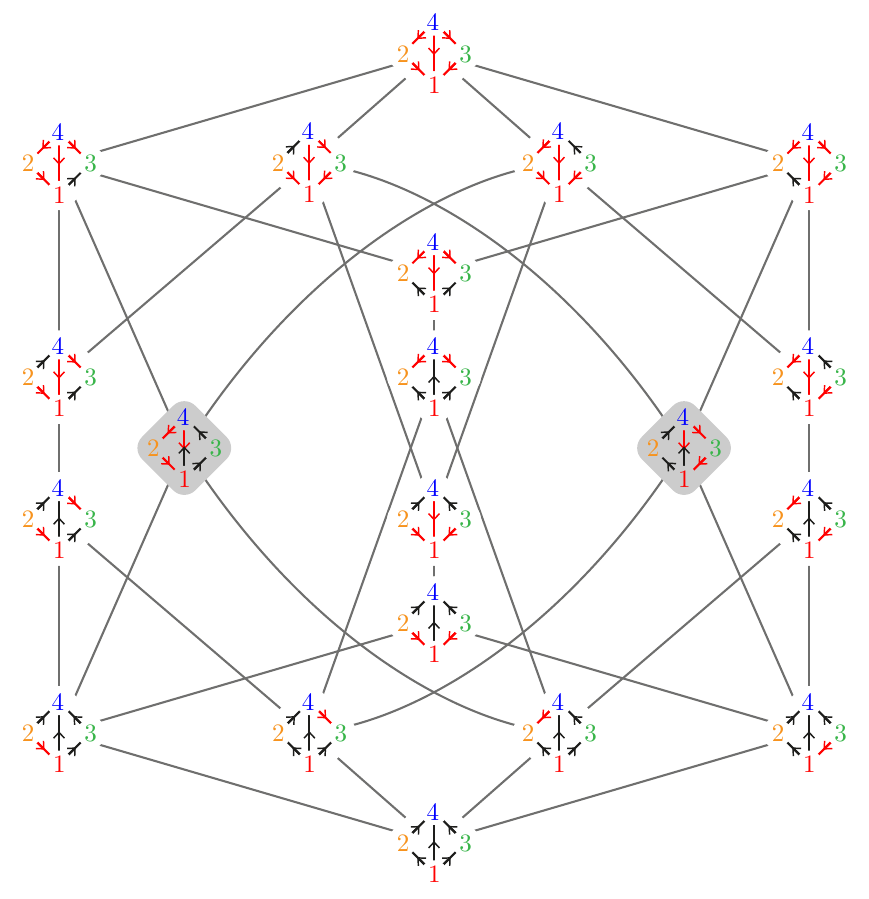}}
	\caption{The MacNeille completion of the acyclic reorientation poset~$\AReori(\tc(\Dgraph))$ of \cref{fig:acyclicReorientationsD}. (Gray = added element). See \cref{rem:MacNeille}.}
	\label{fig:MacNeilleAcyclicReorientationsD}
\end{figure}
\begin{figure}
	\centerline{\includegraphics[scale=.68]{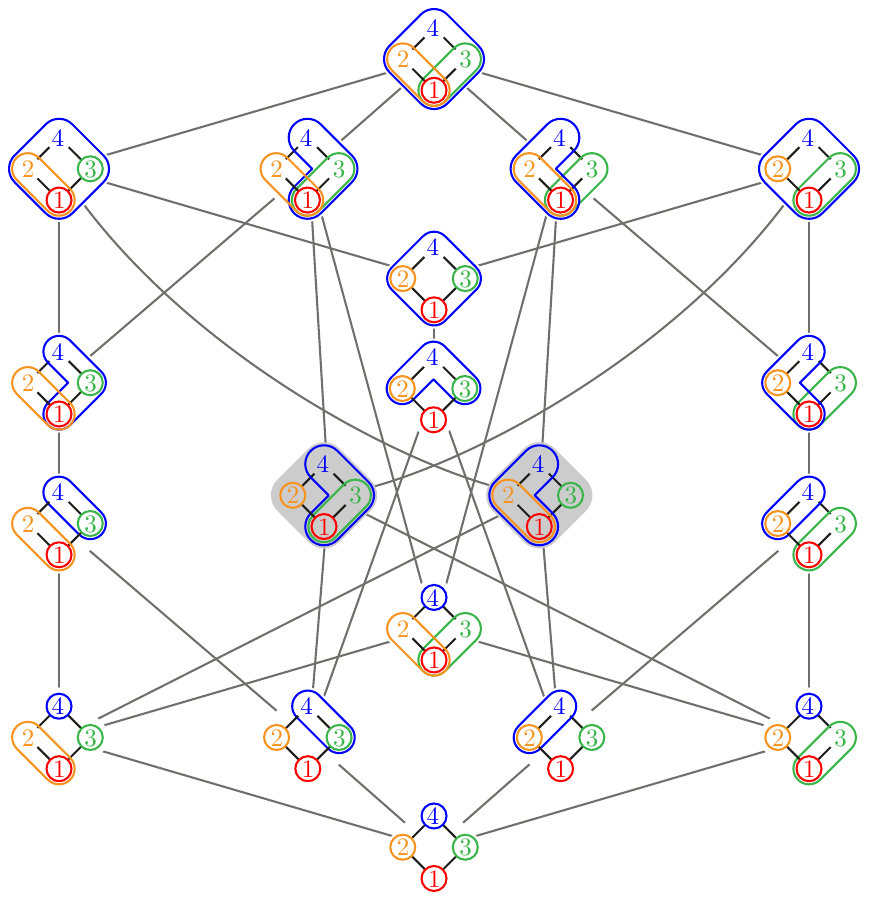}}
	\caption{The MacNeille completion of the acyclic ornamentation poset~$\AOrn(\Dgraph)$ of \cref{fig:ornamentationsD}. (Gray = added element). See \cref{rem:MacNeille}.}
	\label{fig:MacNeilleAcyclicOrnamentationsD}
\end{figure}

\begin{theorem}
\label{thm:MacNeille}
For a directed tree~$T$, 
\begin{itemize}
\item the transitively biclosed reorientation lattice~$\Rbi(\tc(T))$ is the MacNeille completion of the acyclic reorientation poset~$\AReori(\tc(T))$,
\item the ornamentation lattice~$\Orn(T)$ is the MacNeille completion of the acyclic ornamentation poset~$\AOrn(T)$ (or equivalently of the acyclic sourcing poset~$\ASour(\PP(T))$).
\end{itemize}
\end{theorem}

\begin{proof}
This follows from \cref{prop:MacNeilleCompletion,coro:irreduciblesAcyclicT}.
\end{proof}

\pagebreak
\begin{remark}
\label{rem:MacNeille}
Note that~\cref{thm:MacNeille} fails beyond directed trees.
For instance, for the directed graph~$\Dgraph$ of  \cref{fig:ornamentationsD},
\begin{itemize}
\item The transitively biclosed reorientation poset~$\Rbi(\tc(\Dgraph))$ is not even a lattice as observed in \cref{rem:biclosedNotLattice}. The MacNeille completion of the acyclic reorientation poset~$\AReori(\tc(\Dgraph))$ is illustrated in \cref{fig:MacNeilleAcyclicReorientationsD}.
\item Adding to the acyclic ornamentation poset~$\AOrn(\Dgraph)$ the two cyclic ornamentations
\[
\includegraphics[scale=.8,valign=c]{cyclicOrnamentationD3}
\quad\text{and}\quad
\includegraphics[scale=.8,valign=c]{cyclicOrnamentationD4}
\]
suffices to make it a lattice, as illustrated in \cref{fig:MacNeilleAcyclicOrnamentationsD}.
\end{itemize}
Another interesting case is the directed graph~$\Kgraph$ of \cref{fig:ornamentationsKT,fig:reorientationsKNT,fig:acyclicReorientationsKT}, for which the acyclic ornamentations already form lattice, even if there exists a cyclic ornamentation.
In contrast, \cref{thm:MacNeille} still holds for the directed graph~$\Tgraph$ of \cref{fig:ornamentationsKT}.
In general, we have no description of the MacNeille completion of the acyclic ornamentation lattice of an arbitrary directed graph.
\end{remark}

%%%%%%%%%%%%%%%%%%%%%%%%%%%%%%%%%%%%%%
%%%%%%%%%%%%%%%%%%%%%%%%%%%%%%%%%%%%%%
%%%%%%%%%%%%%%%%%%%%%%%%%%%%%%%%%%%%%%

\section{Unstarred trees and rooted trees}
\label{sec:unstarredTrees}

In this section, we discuss the specific case when the underlying directed graph~$D$ is an unstarred tree~$T$ (defined in \cref{subsec:unstarredTrees}).
The main result is the proof of \cref{thm:main2} (\cref{subsec:OrnUT}).
Finally, we discuss enumeration properties of ornamentations of brooms (\cref{subsec:brooms}) and combs (\cref{subsec:combs}).

%%%%%%%%%%%%%%%%%%%%%%%%%%%%%%%%%%%%%%

\subsection{Unstarred trees}
\label{subsec:unstarredTrees}

We first define the family of unstarred trees by the following equivalent conditions on~$\tc(T)$.
We will see further equivalent conditions in terms of the posets~$\AReori(\tc(T))$, $\ASour(\PP(T))$ and~$\AOrn(T)$ in \cref{subsec:OrnUT}.
See \cref{coro:starredCharacterizations} for a summary.

\begin{lemma}
\label{lem:starredCharacterization}
The following conditions are equivalent for a directed tree~$T$ on~$V$:
\begin{enumerate}[(i)]
\item $\tc(T)$ contains an induced alternating cycle,
\item there are distinct $a,b,c,d,e \in V$ such that~$(a,c)$, $(b,c)$, $(c,d)$, and~$(c,e)$ are all in~$\tc(T)$, while~$(a,b)$, $(b,a)$, $(d,e)$ and~$(e,d)$ are not in~$\tc(T)$,
\item there are~$u,v \in V$ such that~$u$ has at least two incoming edges in~$T$, $v$ has at least two outgoing edges in~$T$, and $T$ has a directed path from~$u$ to~$v$,
\end{enumerate}
\end{lemma}

\begin{proof}
\uline{\textsl{(i) $\Rightarrow$ (ii):}}
Let~$C$ be an induced alternating cycle in~$\tc(T)$.
Write~$C = (u_1, v_1, \dots, u_k, v_k)$ where~$(u_i, v_i)$ and~$(u_{i+1}, v_i)$ are in~$\tc(T)$ for all~$i \in [k]$.
As~$C$ is induced, $u_i$ and~$u_j$ (resp.~$v_i$ and~$v_j$) are incomparable for any~$i \ne j$.

Assume first that~$(u_1, v_k)$ is an edge of~$T$.
Let~$D$ and~$U$ denote the connected components of~$T \ssm \{(u_1, v_k)\}$ containing~$u_1$ and~$v_k$ respectively.
Let~$i \in [k]$ be maximal such that~$u_i \in D$.
Since~$u_{i+1} \in U$ and~$(u_{i+1}, v_i) \in \tc(T)$, we have~$v_i \in U$.
Hence, the edge~$(u_1, v_k)$ is an edge of the path from~$u_i$ to~$v_i$ in~$T$.
We conclude that~$C$ was not induced.
Note that the same argument shows that all edges of~$C$ are in~$\tc(T) \ssm T$.

Let~$w$ denote the last vertex before~$v_k$ along the path from~$u_1$ to~$v_k$ in $T$.
If~$w$ does not belong to the path from~$u_k$ to~$v_k$, then the cycle~$C$ must pass again through~$v_k$, contradicting inducedness.
We now distinguish two cases:
\begin{itemize}
\item If there is~$i \in [k-1]$ such that~$(w,v_i) \in \tc(T)$, then we obtain~(ii) by considering~$a = u_1$, $b = u_k$, $c = w$, $d = v_i$ and~$e = v_k$.
\item otherwise, we obtain a new induced alternating cycle by replacing~$v_k$ by~$w$ in~$C$. We thus obtain~(ii) by induction on the sum of the length of the paths between~$u_i$ and~$v_i$ in~$T$.
\qedhere
\end{itemize}

\medskip\noindent
\uline{\textsl{(ii) $\Rightarrow$ (i):}} Let~$a,b,c,d,e$ as in~(ii). Then~$(a,d,b,e)$ forms an induced alternating cycle in~$\tc(T)$.

\medskip\noindent
\uline{\textsl{(ii) $\Rightarrow$ (iii):}} Let~$a,b,c,d,e$ as in~(ii). Let $u$ (resp.~$v$) be the first (resp.~last) common vertex of the paths from~$a$ and~$b$ to~$c$ (resp.~from~$c$ to~$d$ and~$e$). Then~$u$ (resp.~$v$) has at least two incoming (resp.~outgoing) arrows, and there is a path from~$u$ to~$v$ passing through~$c$.

\medskip\noindent
\uline{\textsl{(iii) $\Rightarrow$ (ii):}} Let~$u,v$ as in~(iii). Let~$a,b$ be two incoming neighbors of~$u$, let~$d,e$ be two outgoing neighbors of~$v$, and let~$c$ be any vertex along the path from~$u$ to~$v$. Then~$a,b,c,d,e$ are as in~(ii).
\end{proof}

\begin{definition}
\label{def:starredTree}
We say that an increasing tree~$T$ on~$[n]$ is \defn{starred} if it satisfies the equivalent conditions of \cref{lem:starredCharacterization} and \defn{unstarred} otherwise.
\end{definition}

Note that rooted trees (with increasing edges all oriented toward the root) are unstarred.
For instance, the tree of \cref{fig:ornamentationsX} is starred (it is the smallest starred tree), while those of \cref{fig:ornamentationsNI,fig:ornamentationsAY} are all unstarred.

%%%%%%%%%%%%%%%%%%%%%%%%%%%%%%%%%%%%%%

\subsection{Ornamentation lattices of unstarred trees}
\label{subsec:OrnUT}

This section is devoted to the proof of \cref{thm:main2}, which we cut into pieces.

\begin{proposition}
\label{prop:AReoriLatticeUT}
For an unstarred tree~$T$, all transitively biclosed reorientations of~$\tc(T)$ are acyclic.
In particular, the acyclic reorientation poset~$\AReori(\tc(T))$ is a lattice.
\end{proposition}

\begin{proof}
Assume that there is a transitively biclosed and cyclic reorientation~$R$ of~$\tc(T)$.
As~$R$ contains a cycle, it contains an induced oriented cycle~$C$ (any chord can be used to make the cycle shorter).
Since~$R$ is transitively biclosed, $C$ is an induced alternating cycle in~$\tc(T)$.
This implies that~$T$ is starred by \cref{def:starredTree}.
\end{proof}

\begin{remark}
Note that it is also straightforward to check that an unstarred tree~$T$ satisfies the condition of~\cref{prop:AReoriLat}.
Indeed, given any subset~$U \subset [n]$, the subgraph of~$\tc(T)$ induced by~$U$ contains no induced alternating cycle, hence no induced cycles of length at least~$4$.
Hence, its transitive reduction contains no cycle.
\end{remark}

\begin{proposition}
\label{prop:AOrnLatticeUT}
For any ornamentation~$O$ of an unstarred tree~$T$, the reorientation~$\reori{O}$ of~$\tc(T)$ and the sourcing~$\sour{O}$ of~$\PP(T)$ are both acyclic.
In particular, all ornamentations of~$T$ are acyclic and~$\ASour(\PP(T)) \simeq \AOrn(T) = \Orn(T)$.
\end{proposition}

\begin{proof}
The orientation~$R_O$ is transitively biclosed by~\cref{lem:Orn2ReoriT}, hence acyclic by \cref{prop:AReoriLatticeUT}.
The proof for the acyclicity of~$\sour{O}$ is similar.
\end{proof}

\begin{proposition}
For a starred tree~$T$, there exists a transitively biclosed and cyclic reorientation \linebreak of~$\tc(T)$ and a cyclic ornamentation of~$T$.
Hence, neither the acyclic reorientation poset~$\AReori(\tc(T))$, the acyclic sourcing poset~$\ASour(\PP(T))$, nor the acyclic ornamentation poset~$\AOrn(T)$ are lattices.
\end{proposition}

\begin{proof}
Consider a starred tree~$T$, and let~$a,b,c,d,e$ be vertices of~$T$ as in \cref{lem:starredCharacterization}\,(ii).

The reorientation~$R$ of~$\tc(T)$ defined by~$\rev(R) = \set{(w,d)}{a \le_T w <_T d} \cup \set{(w,e)}{b \le_T w <_T e}$ is transitively closed because~$d \not<_T e$ and~$e \not<_T d$, transitively coclosed because~$(u,w) \in \rev(R)$ implies~$(v,w) \in \rev(R)$ for any~$u <_T v <_T w$, and cyclic because~$a, e, b, d$ forms a directed cycle.

The ornamentation~$O = \orn{R}$ of~$T$ with~$O(d) = [a,d]_T$, $O(e) = [b,e]_T$, and~$O(w) = \{w\}$ for all~$w \in V \ssm \{d,e\}$ is cyclic.
Indeed, consider any reorientation~$R'$ of~$\tc(T)$ such that~$O = \orn{R'}$.
We have~$(w,d) \in \rev(R')$ for all~$w \in [a,d]_T$ (because $w \in O(d)$) and similarly~$(w,e) \in \rev(R')$ for all~$w \in [b,e]_T$.
Therefore,~$(a,e) \notin \rev(R')$ (because~${a \notin O(e)}$ while~$(w,e) \in \rev(R')$ for all~$w \in [a,e]_T$ and similarly~$(b,d) \notin \rev(R')$.
We conclude that~$a, e, b, d$ forms a directed cycle in~$R'$ as well.
Hence, $O$ is cyclic.

We conclude by~\cref{thm:MacNeille} that~$\AReori(\tc(T))$ and~$\ASour(\PP(T)) \simeq \AOrn(T)$ are not lattices.
\end{proof}

To sum up, we proved the following statement.

\begin{corollary}
\label{coro:starredCharacterizations}
The following conditions are equivalent for a directed tree~$T$ on~$V$:
\begin{enumerate}[(i)]
\item $\tc(T)$ contains an induced alternating cycle,
\item there are distinct $a,b,c,d,e \in V$ such that~$(a,c)$, $(b,c)$, $(c,d)$, and~$(c,e)$ are all in~$\tc(T)$, while~$(a,b)$, $(b,a)$, $(d,e)$ and~$(e,d)$ are not in~$\tc(T)$,
\item there are~$u,v \in V$ such that~$u$ has at least two incoming edges in~$T$, $v$ has at least two outgoing edges in~$T$, and $T$ has a directed path from~$u$ to~$v$,
\item there exists a transitively biclosed and cyclic reorientation of~$\tc(T)$,
\item there exists a cyclic ornamentation of~$T$,
\item the acyclic reorientation poset~$\AReori(\tc(T))$ is not a lattice,
\item the acyclic sourcing poset~$\ASour(\PP(T))$ is not a lattice,
\item the acyclic ornamentation poset~$\AOrn(T)$ is not a lattice.
\end{enumerate}
\end{corollary}

We now prove that the ornamentation map is a lattice map when~$T$ is unstarred.

\begin{proposition}
\label{prop:AReori2AOrnLatticeMapUT}
For an unstarred tree~$T$, the map $R \mapsto \aorn{R}$ is a surjective lattice map, hence the lattices~$\ASour(\PP(T)) \simeq \AOrn(T) = \Orn(T)$ are lattice quotients~of~$\AReori(\tc(T))$.
\end{proposition}

\begin{proof}
The map is surjective by \cref{def:AOrn}.
Consider two acyclic reorientations~$R_1$ and~$R_2$ of $\tc(T)$.
We need to show that ${\orn{R_1 \meet R_2} = \orn{R_1} \meet \orn{R_2}}$ and~$\orn{R_1 \join R_2} = \orn{R_1} \join \orn{R_2}$.
Since~$R \mapsto \orn{R}$ is order-preserving (\cref{lem:Reori2Orn2}) and~$R_1 \meet R_2 \le R_1 \le R_1 \join R_2$ and~$R_1 \meet R_2 \le R_2 \le R_1 \join R_2$, we have~$\orn{R_1 \meet R_2} \le \orn{R_1} \le \orn{R_1 \join R_2}$ and~$\orn{R_1 \meet R_2} \le \orn{R_2} \le \orn{R_1 \join R_2}$, hence~${\orn{R_1 \meet R_2} \le \orn{R_1} \meet \orn{R_2}}$ and~${\orn{R_1} \join \orn{R_2} \le \orn{R_1 \join R_2}}$.
To prove the converse inequalities, remember the descriptions of~${O_1 \meet O_2}$ and~${O_1 \join O_2}$ from \cref{thm:OrnMeetJoin}, and of~$\rev(R_1 \meet R_2)$ and~$\rev(R_1 \join R_2)$ from~\cref{prop:AReoriLatJoinMeet}.

If $u <_T v$ but~$u \notin \orn{R_1 \meet R_2}(v)$, then there is $u'$ along the path from~$u$ to~$v$ such that~$(u',v) \notin \rev(R_1 \meet R_2)$ by \cref{lem:Reori2OrnT}.
Therefore, the path~$u' = w_0, w_1, \dots, w_k = v$ satisfies~$(w_{i-1}, w_i) \notin \rev(R_1) \cap \rev(R_2)$ for all $i \in [k]$ by~\cref{prop:AReoriLatJoinMeet}.
As~$w_{k-1}$ is on the path from~$u$ to~$v$ and~$(w_{k-1},v) \notin \rev(R_1) \cap \rev(R_2)$, we conclude by \cref{lem:Reori2OrnT} that~$u \notin \orn{R_1}(v)$ or~$u \notin \orn{R_2}(v)$, hence~$u \notin \orn{R_1}(v) \cap \orn{R_2}(v) = (\orn{R_1} \meet \orn{R_2})(v)$ (by \cref{coro:meetOrnT}, since~$T$ is a directed tree).
We conclude that~$(\orn{R_1} \meet \orn{R_2}) \le \orn{R_1 \meet R_2}$, hence that~$\orn{R_1 \meet R_2} = \orn{R_1} \meet \orn{R_2}$.

Consider now~$u \in \orn{R_1 \join R_2}(v)$.
Let~$u = w_0, w_1, \dots, w_k = v$ denote the path from~$u$ to~$v$.
By \cref{lem:Reori2OrnT}, $(w_i, v) \in \rev(R_1 \join R_2)$ for all~$i \in [k]$.
Assume that~$w_{i+1}, \dots, w_k \in (\orn{R_1} \join \orn{R_2})(v)$ for some~$i \in [k-1]$.
Since~$(w_i, v) \in \rev(R_1 \join R_2) = \tc(T) \cap \tc(\rev(R_1) \cup \rev(R_2))$, there is~${i+1 \le \ell \le k}$ such that~$(w_i, w_\ell) \in \rev(R_1) \cup \rev(R_2)$.
Assume that~$\ell$ is minimal for this property, and assume by symmetry that~$(w_i, w_\ell) \in \rev(R_1)$.
Since~$R_1$ is acyclic, the minimality of $\ell$ implies that~$(w_j, w_\ell) \in \rev(R_1)$ for all~${i \le j \le \ell}$.
By \cref{lem:Reori2OrnT}, we obtain that~$w_i \in \orn{R_1}(w_\ell)$.
By \cref{thm:OrnMeetJoin}, we therefore obtain that~${w_i \in (\orn{R_1} \join \orn{R_2})(v)}$.
Hence, by descending induction, we obtain that~$u \in (\orn{R_1} \join \orn{R_2})(v)$.
We conclude that~$\orn{R_1 \join R_2} \le \orn{R_1} \join \orn{R_2}$, hence that~${\orn{R_1 \join R_2} = \orn{R_1} \join \orn{R_2}}$.
\end{proof}

\begin{corollary}
If~$T$ is an unstarred tree, then the ornamentation lattice~$\Orn(T)$ is isomorphic to the transitive closure of the graph of the path hypergraphic polytope~$\simplex_{\PP(D)} \eqdef \sum_{P \in \PP(T)} \simplex_P$ oriented in the direction~$\omega$.
\end{corollary}

\begin{remark}
Lattice quotients of acyclic reorientation lattices were studied in detail in~\cite{Pilaud-acyclicReorientationLattices}.
We can reformulate \cref{prop:AReoriLatticeUT,prop:AOrnLatticeUT,prop:AReori2AOrnLatticeMapUT} in the language of~\cite{Pilaud-acyclicReorientationLattices}: for an unstarred~tree~$T$,
\begin{enumerate}
\item the transitive closure~$\tc(T)$ is skeletal, and so~$\AReori(\tc(T))$ is an acyclic reorientation lattice,
\item the reorientation lattice~$\AReori(\tc(T))$ is (isomorphic to) the Tamari lattice of~$\tc(T)$ defined by the $\tc(T)$-sylvester congruence,
\item the path hypergraphic polytope~$\simplex_{\PP(T)}$ is the $\tc(T)$-associahedron.
\end{enumerate}
\end{remark}

\begin{remark}
It might seem more natural to use the map~$\fS_n \to \ASour(\PP(T))$ defined by~$\pi \mapsto \asour{\areori{\pi}}$ instead of the map~$\AReori(\tc(T)) \to \ASour(\PP(T))$ defined by~$R \mapsto \asour{R}$.
However, the former is not always a lattice map, while the latter is.
For instance, the fiber of the acyclic sourcing identified with the ornamentation
\[
\includegraphics[scale=.8,valign=c]{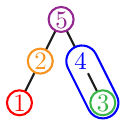}
\]
is the set of permutations~$\{41325, 43125, 14325, 14235, 12435, 12453, 14253, 41253, 41235\}$ which is not even an interval of the weak order as it has two maximal elements~$43125$ and~$41253$.
\end{remark}

%%%%%%%%%%%%%%%%%%%%%%%%%%%%%%%%%%%%%%

\subsection{Brooms}
\label{subsec:brooms}

This short section is dedicated to the enumeration of ornamentations of the following specific family of trees, illustrated in \cref{fig:broomsCombs}\,(left).

\begin{definition}
Let~$m,n \ge 0$.
The \defn{$(m,n)$-broom} is the rooted tree obtained by attaching $m$ nodes (called \defn{bristles}) to the initial vertex of a path with $n$ nodes (called \defn{handle}).
\end{definition}

\begin{remark}
The transitive closure of the $(m,n)$-broom has~$2^{n(2m+n-1)/2}$ reorientations, of which $n! \, (n+1)^m$ are acyclic, and~$(n+1)!^m \prod_{k = 1}^n k^{n+1-k}$ sourcings.
As the $(m,n)$-broom is a rooted tree, all its ornamentations are acyclic by \cref{prop:AOrnLatticeUT}, so its numbers of acyclic sourcings, acyclic ornamentations and ornamentations all coincide, and we compute them below.
\end{remark}

We denote by~$B_{m,n}$ the number of (acyclic) ornamentations of the $(m,n)$-broom.
The first few values are gathered in \cref{tab:broom}.
In~\cite{OEIS}, the first three rows of \cref{tab:broom} are the sequences \OEIS{A000108}, \OEIS{A000108}, and \OEIS{A070031} respectively, while the first three columns of \cref{tab:broom} are the sequences \OEIS{A000012}, \OEIS{A000079}, and \OEIS{A007689} respectively.
We denote by
\[
B(x,y) \eqdef \sum_{m \ge 0, \, n \ge 0} B_{m,n} \, \frac{x^m}{m!} \, y^n
\]
their generating function (exponential in~$x$ and ordinary in~$y$), and by
\[
B_m(y) \eqdef m! \, [x^m] \, B(x,y) = \sum_{n \ge 0} B_{m,n} \, y^n
\]
the coefficient of~$x^m/m!$ in~$B(x,y)$.
As~$B_{0,n} = B_{1,n-1}$ is the $n$th Catalan number~$C_n \eqdef \frac{1}{n+1}\binom{2n}{n}$, we have $B_0(y) = C(y)$ and~$B_1(y) = (C(y) - 1)/y$ where~$C(y) \eqdef \sum_{n \ge 0} C_n \, y^n = \frac{1-\sqrt{1-4y}}{2y}$.

\begin{table}
	\[
	\begin{array}{c|ccccccccccc}
	m \backslash n & 0 & 1 & 2 & 3 & 4 & 5 & 6 & 7 & 8 & \cdots \\
	\hline
	0 & 1 & 1 & 2 & 5 & 14 & 42 & 132 & 429 & 1430 & \cdots \\ % \OEIS{A000108}
	1 & 1 & 2 & 5 & 14 & 42 & 132 & 429 & 1430 & 4862 & \cdots \\ % \OEIS{A000108}
	2 & 1 & 4 & 13 & 42 & 138 & 462 & 1573 & 5434 & 19006 & \cdots \\ % \OEIS{A070031}
	3 & 1 & 8 & 35 & 134 & 492 & 1782 & 6435 & 23270 & 84422 & \cdots \\ % ?
	4 & 1 & 16 & 97 & 450 & 1878 & 7458 & 28873 & 110266 & 418030 & \cdots \\
	5 & 1 & 32 & 275 & 1574 & 7572 & 33342 & 139659 & 567590 & 2263142 & \cdots \\
	6 & 1 & 64 & 793 & 5682 & 31878 & 157122 & 717673 & 3124474 & 13177006 & \cdots \\
	7 & 1 & 128 & 2315 & 21014 & 138852 & 772302 & 3872955 & 18167270 & 81443702 & \cdots \\
	8 & 1 & 256 & 6817 & 79170 & 621318 & 3927378 & 21752953 & 110506426 & 528949870 & \cdots \\ 
	\vdots & \vdots & \vdots & \vdots & \vdots & \vdots & \vdots & \vdots & \vdots & \vdots & \ddots \\
%	\cite{OEIS} & \OEIS{A000012} & \OEIS{A000079} & \OEIS{A007689} & ?
	\end{array}
	\]
	\caption{The number $B_{m,n}$ of (acyclic) ornamentations of the $(m,n)$-broom.}
	\label{tab:broom}
\end{table}

\begin{proposition}
\label{prop:brooms1}
The numbers~$B_{m,n}$ satisfy the following recurrence relation:
\[
B_{m,n} = \sum_{k = 0}^m \sum_{\ell = 1}^n \binom{m}{k} \, C_{n-\ell} \, B_{k,\ell-1}
\]
where~$\displaystyle C_n \eqdef \frac{1}{n+1}\binom{2n}{n}$ is the $n$th Catalan number.
\end{proposition}

\begin{proof}
Consider an ornament~$U$ of the $(m,n)$-broom containing~$k \ge 1$ bristles and $\ell \ge 1$ vertices of the handle.
The ornamentations~$O$ of the $(m,n)$-broom containing~$U$ are clearly in bijection with pairs formed by an ornamentation of the $(k, \ell-1)$-broom and an ornamentation of the $(n-\ell)$-path.
Conversely, we can associate to each ornamentation of the $(m,n)$-broom the inclusion maximal ornament~$U$ of~$O$ containing at least one bristle (note that there might be no such ornament).
The formula immediately follows.
\end{proof}

\begin{proposition}
\label{prop:brooms2}
The generating function~$\displaystyle B(x,y) \eqdef \sum_{m \ge 0, \, n \ge 0} B_{m,n} \, \frac{x^m}{m!} \, y^n$ is given~by
\[
B(x,y) = \frac{e^x}{1-e^x \, y \, C(y)}
\]
where~$C(y) \eqdef \sum_{n \ge 0} C_n \, y^n = \frac{1-\sqrt{1-4y}}{2y}$.
\end{proposition}

We give two proofs of \cref{prop:brooms2}: the first is a direct translation of the decomposition of ornamentations of brooms discussed in the proof of \cref{prop:brooms1}, the second is a more pedestrian proof based on the expression of \cref{prop:brooms1}.

\begin{proof}[First proof of \cref{prop:brooms2}]
The decomposition of ornamentations of brooms discussed in the proof of \cref{prop:brooms1} tells that ornamentations of brooms can be thought of as sequences of pairs, each formed by a subset of bristles and an ornamentation of a path.
Using classical generatingfunctionology~\cite{FlajoletSedgewick}, this immediately yields the expression of the statement.
\end{proof}

\begin{proof}[Second proof of \cref{prop:brooms2}]
Recall that if
\[
U(x,y) = \sum_{m \ge 0, \, n \ge 0} U_{m,n} \, \frac{x^m}{m!} \, y^n
\qquad\text{and}\qquad
V(x,y) = \sum_{m \ge 0, \, n \ge 0} V_{m,n} \, \frac{x^m}{m!} \, y^n
\]
are two generating functions exponential in~$x$ and ordinary in~$y$, then
\[
U(x,y) \cdot V(x,y) = \sum_{m \ge 0, \, n \ge 0} W_{m,n}  \, \frac{x^m}{m!} \, y^n
\qquad\text{where}\qquad
W_{m,n} = \sum_{k = 0}^m \sum_{\ell = 0}^n \binom{m}{k} \, U_{k,\ell} \, V_{n-k,n-\ell}.
\]
Considering~$U(x,y) \eqdef B(x,y) \, y$ and~$V(x,y) \eqdef e^x \, C(y)$, we obtain by \cref{prop:brooms1}
\begin{align*}
B(x,y) \, y \, e^x \, C(y) & = \sum_{m \ge 0, \, n \ge 1}  \left( \sum_{k = 0}^m \sum_{\ell = 1}^n \binom{m}{k} \, B_{k,\ell-1} \, C_{n-\ell} \right) \, \frac{x^m}{m!} \, y^n \\
& = \sum_{m \ge 0, \, n \ge 1} B_{m,n} \, \frac{x^m}{m!} \, y^n \\
& = B(x,y) - e^x
\end{align*}
from which we derive that
\[
B(x,y) = \frac{e^x}{1-e^x \, y \, C(y)}.
\qedhere
\]
\end{proof}

\begin{proposition}
\label{prop:brooms3}
For any~$m \in \N$, we have
\[
y \, B_m(y) \, C(y) = \sum_{k = 0}^m (-1)^{m-k} \binom{m}{k} \, B_k(y),
\]
where~$\displaystyle C(y) \eqdef \sum_{n \ge 0} C_n \, y^n = \frac{1-\sqrt{1-4y}}{2y}$.
\end{proposition}

\begin{proof}
The same argument as in the proof of \cref{prop:brooms1} gives the functional equation
\[
B_m = 1 + y \, C \, \sum_{k = 0}^m \binom{m}{k} \, B_k.
\]
As the inverse of the binomial matrix~$\big[\binom{m}{k}\big]_{m,k}$ is the signed binomial matrix~$\big[(-1)^{m-k}\binom{m}{k}\big]_{m,k}$, we obtain that
\[
y \, B_m \, C = \sum_{k = 0}^m (-1)^{m-k} \binom{m}{k} \,(B_k - 1) = \sum_{k = 0}^m (-1)^{m-k} \binom{m}{k} \, B_k.
\qedhere
\]
\end{proof}

\begin{example}
For~$m = 2$, we have
\(
B_2 = 1 + y \, C \, (B_0 + 2  \, B_1 + B_2) = 1 + y \, C( C + 2 \, (C - 1) / y + B_2)
\)
from which we derive that
\(
B_2 \, (1 - y \, C) = 1 + y \, C^2 + 2 \, C \, (C - 1) = 2 \, C^2 - C. 
\)
Finally, we observe that
\(
C^2 (1+y \, C) \, ( 1 - y \, C) = C^2 \, (1 - y^2 \, C^2) = C^2 - y^2 \, C^4 = C^2 - (C - 1)^2 = 2 \, C^2 - C
\)
We conclude that~$B_2 = C^3 (1+y \, C)$.
\end{example}

\begin{figure}
	\centerline{\includegraphics[scale=.8]{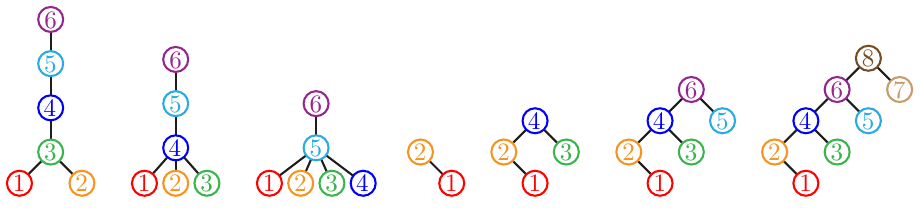}}
	\caption{The $(2,4)$-, $(3,3)$- and $(4,2)$-brooms (left) and the $1$-, $2$-, $3$- and $4$-combs (right).}
	\label{fig:broomsCombs}
\end{figure}

\subsection{Combs}
\label{subsec:combs}

This short section is dedicated to the enumeration of ornamentations of the following specific family of trees, illustrated in \cref{fig:broomsCombs}\,(right).

\begin{definition}
Let~$n \ge 0$.
The \defn{$n$-comb} is the rooted tree obtained by attaching a node (called \defn{tooth}) to each node of a path with $n$ nodes (called \defn{handle}).
\end{definition}

To fix a convention, let us denote by~$1, \dots, 2n-1$ the teeth and by~$2, \dots, 2n$ the handle of the $n$-comb, so that the parent of node~$i$ is $i+1$ if~$i$ is odd and~$i+2$ if~$i$ is even (except the node~$2n$ which is the root).
See \cref{fig:broomsCombs}.

\begin{remark}
The transitive closure of the $n$-comb has~$2^{n^2}$ reorientations, of which $n! \, (n+1)!$ are acyclic, and~$\prod_{k = 1}^n k^{n+1-k} \, (k+1)!$ sourcings.
As the $n$-comb is a rooted tree, all its ornamentations are acyclic by \cref{prop:AOrnLatticeUT}, so its numbers of acyclic sourcings, acyclic ornamentations and ornamentations all coincide, and we compute them below.
\end{remark}

We denote by~$E_n$ the number of (acyclic) ornamentations of the $n$-comb.
It is the sequence \OEIS{A000698} in \cite{OEIS}, and its first few values are
\[
1, 2, 10, 74, 706, 8162, 110410, 1708394, 29752066, 576037442, 12277827850, 285764591114, \dots
\]
In order to give a formula for~$E_n$, we introduce the following two families of combinatorial objects, illustrated in \cref{fig:comb}.

\begin{figure}
	\centerline{\includegraphics[scale=.68]{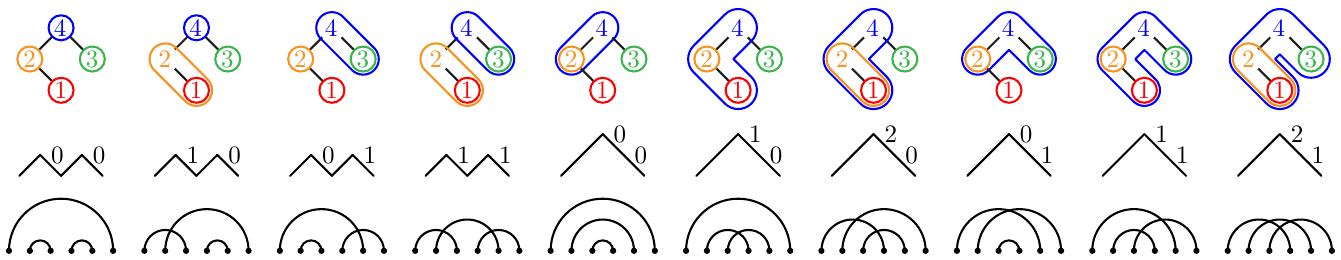}}
	\caption{The bijections of \cref{prop:comb} between (acyclic) ornamentations of the $n$-comb (top), labeled Dyck paths of semilength~$n$ (middle) and indecomposable perfect matchings of~$[2n+2]$ (bottom), for~$n = 2$.}
	\label{fig:comb}
\end{figure}

\begin{definition}
A \defn{Dyck path} of semilength~$n$ is a path with up steps~$(1,1)$ and down steps~$(1,-1)$ starting at~$(0,0)$, ending at~$(n,n)$, and never passing strictly below the horizontal axis.
A \defn{labeled Dyck path} is a Dyck path where each down step is labeled by an integer in~$[0,h]$ where $h$ is the height of its top endpoint.
\end{definition}

\begin{definition}
A \defn{perfect matching} of~$[2n]$ is a partition~$\pi$ of~$[2n]$ into $n$ pairs.
It is \defn{decomposable} if there is~$k \in [n]$ such that any pair of~$\pi$ is contained either in~$[2k]$ or in its complement~$[2n] \ssm [2k]$.
\end{definition}

\begin{proposition}
\label{prop:comb}
The following sets of combinatorial objects are in bijection:
\begin{enumerate}[(i)]
\item the (acyclic) ornamentations of the $n$-comb,
\item the labeled Dyck paths of semilength $n$,
\item the indecomposable perfect matchings of~$[2n+2]$.
\end{enumerate}
Hence, the number~$E_n$ of (acyclic) ornamentations of the $n$-comb satisfies
\[
E_n = (2n+1)!! - \sum_{k = 1}^n (2k-1)!! \, E_{n-k}
\]
where~$\displaystyle(2k-1)!! = \frac{(2k)!}{2^k \, k!} = (2k-1) \cdot (2k-3) \cdots 5 \cdot 3 \cdot 1$.
\end{proposition}

\begin{proof}
We only sketch the ideas of the bijections, which are illustrated in \cref{fig:comb}.

\medskip\noindent
\uline{\textsl{Ornamentations versus labeled Dyck paths:}}
The restriction of an ornamentation~$O$ to the handle gives an ornamentation of the $n$-path, hence a Dyck path~$D$ of semilength~$n$.
The height of the top endpoint of the $i$th down step of~$D$ is the number of ornaments of~$O$ containing the node $2i$ (\ie the $i$th node of the handle) of the $n$-comb.
The label of this $i$th down step records how many ornaments of~$O$ contain the node~$2i-1$ (\ie the $i$th tooth) of the $n$-comb.

\medskip\noindent
\uline{\textsl{Labeled Dyck paths versus indecomposable perfect matchings:}}
Reading a labeled Dyck path~$D$ of semilength~$n$, we construct a perfect matching~$M$ iteratively as follows.
We start with a free point.
If we read an up step, we add a free point to the right of~$M$.
If we read a down step labeled by~$\ell$, we add a point to the right of~$M$ and connect it to the $(\ell+1)$st free point starting from the right.
Finally, when we finished to read the Dyck path~$D$, we add a last point to the right of~$M$ and connect it to the only remaining free point.
The conditions on the labels clearly ensure that we obtain a perfect matching.
It is indecomposable since we never exhaust the free points until the very last point that we add.
\end{proof}

\begin{remark}
The number of ornamentations of the $n$-comb minus its bottommost node is the sequence \OEIS{A105616} in~\cite{OEIS}, whose first few values are
\[
1, 4, 26, 226, 2426, 30826, 451586, 7489426, 138722426, 2839238026, 63654973826, 1551919194226, \dots % 40888965122426, 1157981114051626, 35083865696279426, 1132449247218851026, 38800104353355372026, 1406432065083818193226
\]
The sequences \OEIS{A000698} and \OEIS{A105616} are the first two columns of the table \OEIS{A105615} in~\cite{OEIS}.
We have no explanation for this fact, nor candidates for families of directed trees whose numbers of ornamentations would give the other columns of this table.
\end{remark}

%sage: for n in range(6): 
%....:     D = DiGraph([[i,i+n] for i in range(1,n)] + [[i,i+1] for i in range(n,2*n-1)]) 
%....:     print(n, len(digraph_ornamentation_lattice(D))) 
%....:                                                                                                                                                                                                       
%0 1
%1 1
%2 4
%3 26
%4 226
%5 2426
%sage: for n in range(6): 
%....:     D = DiGraph([[i,i+n] for i in range(2,n)] + [[i,i+1] for i in range(n,2*n-1)]) 
%....:     print(n, len(digraph_ornamentation_lattice(D))) 
%....:                                                                                                                                                                                                       
%0 1
%1 1
%2 2
%3 10
%4 74
%5 706
%sage: for n in range(6): 
%....:     D = DiGraph([[i,i+n] for i in range(3,n)] + [[i,i+1] for i in range(n,2*n-1)]) 
%....:     print(n, len(digraph_ornamentation_lattice(D))) 
%....:                                                                                                                                                                                                       
%0 1
%1 1
%2 2
%3 5
%4 28
%5 224

%%%%%%%%%%%%%%%%%%%%%%%%%%%%%%%%%%%%%%
%%%%%%%%%%%%%%%%%%%%%%%%%%%%%%%%%%%%%%
%%%%%%%%%%%%%%%%%%%%%%%%%%%%%%%%%%%%%%

\section{Intreeval hypergraphic lattices}
\label{sec:intreevalHypergraphicPosets}

In this section, we characterize the subhypergraphs~$\II$ of the path interval~$\PP(T)$ of an increasing tree~$T$ whose acyclic sourcing poset~$\ASour(\II)$ is a lattice.
We first state the needed definitions and the characterization (\cref{subsec:characterization}), then show that these conditions are necessary (\cref{subsec:necessary}) and sufficient (\cref{subsec:sufficient}) using an important property of star sparse intreeval hypergraphs (\cref{subsec:shortCycles}).
Finally, when~$\ASour(\II)$ is a lattice, we give an explicit formula for its meets and joins (\cref{subsec:formulaMeetJoin}).

%%%%%%%%%%%%%%%%%%%%%%%%%%%%%%%%%%%%%%

\subsection{Characterization}
\label{subsec:characterization}

We first present a few definitions needed to state our characterization in \cref{thm:intreevalLattices}.

\begin{definition}
An \defn{intreeval hypergraph} of an increasing tree~$T$ is a subhypergraph of~$\PP(T)$.
In other words, any hypergraph whose hyperedges are (the vertex sets of) some directed paths in~$T$.
\end{definition}

\begin{definition} 
\label{def:pathIntersectionClose}
We say that an intreeval hypergraph~$\II$ of an increasing tree~$T$ is \defn{path intersection closed} if for all~$I,J\in \II$ with~$|I\cap J|> 1$ such that $\min(I) <_T \min(I\cap J)$ and  $\max(J) >_T \max(I\cap J)$, there exists~$K \in \II$ such that~$I \cap J \subseteq K \subseteq [\min(J) , \max(I)]_T$. %$\min(J) \le_T \min(K) \le_T \min(I\cap J)$ and ${\max(I\cap J) \le_T \max(K) \le_T \max(I)}$.
\end{definition}

\begin{remark}
Note that if~$\II$ is \defn{intersection closed} (\ie~$I , J \in \II$ and~$|I \cap J| > 1$ implies~$I \cap J \in \II$), then~$\II$ is path intersection closed.
However, the reverse implication is wrong.
See \cref{exm:intreevalHypergraphs,fig:intreevalHypergraphs}.
\end{remark}

\begin{definition}
\label{def:pathStarSparse}
Let~$\II$ be an intreeval hypergraph of an increasing tree~$T$.
Consider~$u \le_T v$, and let~$u_1, \dots, u_k$ (resp.~$v_1, \dots, v_\ell$) denote the incoming (resp.~outgoing) neighbors of~$u$ (resp.~$v$) in~$T$.
We denote by~$X(T,u,v,\II)$ the (undirected) graph on~$u_1, \dots, u_k,\; v_1, \dots, v_\ell$ with an edge joining~$u_i$ and~$v_j$ if and only if there is~$I \in \II$ containing both~$u_i$ and~$v_j$.
We say that~$\II$ is \defn{star sparse} if~$X(T,u,v,\II)$ is acyclic for all~$u,v$ connected by a directed path.
\end{definition}

\begin{example}
\label{exm:intreevalHypergraphs}
Among the $8$ intreeval hypergraphs of \cref{fig:intreevalHypergraphs},
\begin{itemize}
\item the first two are both path intersection closed and star sparse,
\item the next four are not path intersection closed but are star sparse,
\item the next one is path intersection closed but not star sparse,
\item the last one is neither path intersection closed nor star sparse.
\end{itemize}

\begin{figure}[b]
	\centerline{\includegraphics[scale=.8]{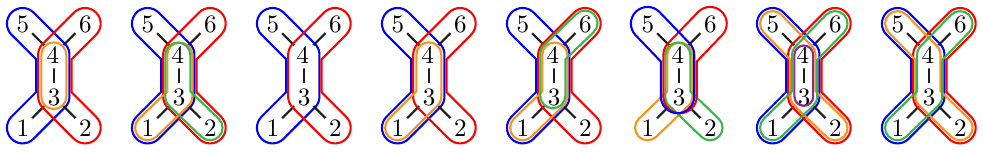}}
	\caption{Some intreeval hypergraphs. See \cref{exm:intreevalHypergraphs}.}
	\label{fig:intreevalHypergraphs}
\end{figure}
\end{example}

The goal of this section is to show the following characterization.
It will be a consequence of \cref{lem:pathIntersectionClosed,lem:pathStarSparse,prop:quasiLatticeMapIntreevalLattices}.

\begin{theorem}
\label{thm:intreevalLattices}
Let~$\II$ be an intreeval hypergraph of an increasing tree~$T$.
The acyclic sourcing poset~$\ASour(\II)$ is a lattice if and only if $\II$ is path intersection closed (\cref{def:pathIntersectionClose}) and star sparse (\cref{def:pathStarSparse}).
\end{theorem}

\begin{remark}
When~$T$ is an oriented path, 
\begin{itemize}
\item the path intersection closed condition boils down to being closed under intersection (including singletons by default),
\item the star sparse condition holds trivially,
\item the characterization of \cref{thm:intreevalLattices} was proved in~\cite[Thm.~A]{BergeronPilaud}.
\end{itemize}
\end{remark}

\begin{remark}
When~$T$ is a rooted tree,
\begin{itemize}
\item the path intersection closed condition rephrases to for all~$I,J\in \II$ with~$|I\cap J|> 1$ such that $\min(I) <_T \min(I\cap J)$ and  $\max(J) >_T \max(I\cap J)$, there exists~$K \in \II$ such that $\min(J) = \min(K)$ and ${\max(I\cap J) \le_T \max(K) \le_T \max(I)}$,
\item the star sparse condition holds trivially.
\end{itemize}
In particular, we have the following corollary.
\end{remark}

\begin{corollary}
The acyclic sourcing poset~$\ASour(\II)$ of an intersection closed intreeval hypergraph~$\II$ of a rooted tree~$T$ is a lattice.
\end{corollary}

%%%%%%%%%%%%%%%%%%%%%%%%%%%%%%%%%%%%%%

\subsection{Necessary condition}
\label{subsec:necessary}

We first prove that the two conditions of \cref{thm:intreevalLattices} are necessary.

\begin{lemma}
\label{lem:pathIntersectionClosed}
Let~$\II$ be an intreeval hypergraph of an increasing tree~$T$.
If~$\ASour(\II)$ is a lattice, then $\II$ is path intersection closed (\cref{def:pathIntersectionClose}).
\end{lemma}

\begin{proof}
The proof follows exactly the same lines as that of~\cite[Prop.~4.6]{BergeronPilaud}.
Assume by means of contradiction that~$\ASour(\II)$ is a lattice but that there are~$I,J\in \II$ with~$|I\cap J|> 1$ such that $\min(I) <_T \min(I\cap J)$ and  $\max(J) >_T \max(I\cap J)$, and there is no~$K \in \II$ such that ${I \cap J \subseteq K \subseteq [\min(J) , \max(I)]_T}$.

Let~$b$ (resp.~$c$) be the bottom (resp.~top) endpoint of~$I \cap J$, and~$a$ (resp.~$d$) be its neighbor in~$I \ssm J$ (resp.~$J \ssm I$).
Note that~$a,b,c,d$ are all distinct since~$|I \cap J| > 1$.
Let~$Y$ be a word formed by~$(I \cup J) \ssm \{a,b,c,d\}$ in any order, and~$X$ be a word formed by~$[n] \ssm (I \cup J)$ in any order.
Consider the permutations
\[
\pi_A \eqdef XbacdY
\quad
\pi_B \eqdef XacdbY
\quad
\pi_C \eqdef XdbacY
\quad
\pi_D \eqdef XcdbaY
\quad
\pi_E \eqdef XadcbY
\quad
\pi_F \eqdef XadbcY
\]
and let~$S_A \eqdef \asour{\pi_A}^\II, \dots, S_F \eqdef \asour{\pi_F}^\II$ denote the corresponding acyclic sourcings of~$\II$, where~$\asour{\pi}^\II$ is defined by~$\asour{\pi}^\II(I) \eqdef \pi\bigl(\min \set{j \in [n]}{\pi(j) \in I} \bigr)$.
See \cref{exm:pathIntersectionClosed,fig:pathIntersectionClosedStarSparse}\,(left) for an illustration of these permutations and acyclic sourcings.

We claim that~$S_E = S_F$.
Indeed, if~$K \in \II$ is such that~$S_E(K) \ne S_F(K)$, then~$K \supseteq \{b,c\}$ but~${K \cap (X \cup \{a,d\}) = \varnothing}$.
This would imply that~${I \cap J = [b,c] \subseteq K \subseteq [\min(J) , \max(I)]_T}$, contradicting our assumption on~$I$ and~$J$.

Observe now that
\[
\pi_A < \pi_C
\qquad
\pi_A < \pi_D
\qquad
\pi_B < \pi_D
\qquad
\pi_B < \pi_E
\qquad
\pi_F < \pi_C
\]
in weak order (note that~$\pi_B \not< \pi_C$ though).
As~$\pi \mapsto \aorn{\pi}$ is order-preserving and~$S_E = S_F$, we conclude that
\[
S_A < S_C
\qquad
S_A < S_D
\qquad
S_B < S_E = S_F < S_C
\qquad
S_B < S_D.
\]

As~$\ASour(\II)$ is a lattice, we have an acyclic sourcing~$S_G$ of~$\II$ such that
\[
S_A \le S_G
\qquad
S_B \le S_G
\qquad
S_G \le S_C
\qquad
S_G \le S_D.
\]
Since~$S_G$ is acyclic and~$\pi \mapsto \asour{\pi}$ is surjective, there is a permutation~$\pi_G$ such that~$S_G = \asour{\pi}$.
Let
\[
g \eqdef \pi_G(\min\set{i \in [n]}{\pi_G(i) \in I \cup J}).
\]
We discuss four cases:
\begin{enumerate}[(i)]
\item if~$\min(I) \le_T g <_T b$, then~$S_G(I) = g < b = S_A(I)$, contradicting that~$S_A < S_G$,
\item if~$\min(J) \le_T g <_T c$, then~$S_G(J) = g < c = S_B(J)$, contradicting that~$S_B < S_G$,
\item if~$c <_T g \le \max(J)$, then~$S_G(J) = g > c = S_D(J)$, contradicting that~$S_G < S_D$,
\item if~$b <_T g \le_T \max(I)$, then~$S_G(I) = g > b = S_C(I)$, contradicting that~$S_G < S_C$.
\end{enumerate}
Note that the case $g = b$ (resp.~$g = c$) is covered by Case~(ii) (resp.~(iii)), hence these four cases cover all possibilities.
As these four cases are all excluded, we reached a contradiction.
\end{proof}

\pagebreak

\begin{example}
\label{exm:pathIntersectionClosed}
We illustrate the construction of the proof of \cref{lem:pathIntersectionClosed} on the third intreeval hypergraph of \cref{fig:intreevalHypergraphs}.
Let~$I \eqdef \{1,3,4,5\}$ and~$J \eqdef \{2,3,4,6\}$ so that~$a = 1$, $b = 3$, $c = 4$, $d = 6$, $X = \varnothing$, $Y = 25$ and we have the following permutations and sources:
\[
\begin{array}{c|cccccc}
	Z & A & B & C & D & E & F \\
	\hline
	\pi_Z & 314625 & 146325 & 631425 & 463125 & 164325 & 163425 \\
	S_Z(\blue{1345}) & 3 & 1 & 3 & 4 & 1 & 1 \\
	S_Z(\red{2346}) & 3 & 4 & 6 & 4 & 6 & 6
\end{array}
\]
See \cref{fig:pathIntersectionClosedStarSparse}\,(left).
The reader can check that~$S_A < S_C$, $S_A < S_D$, $S_B < S_E = S_F < S_C$, and~$S_B < S_D$.
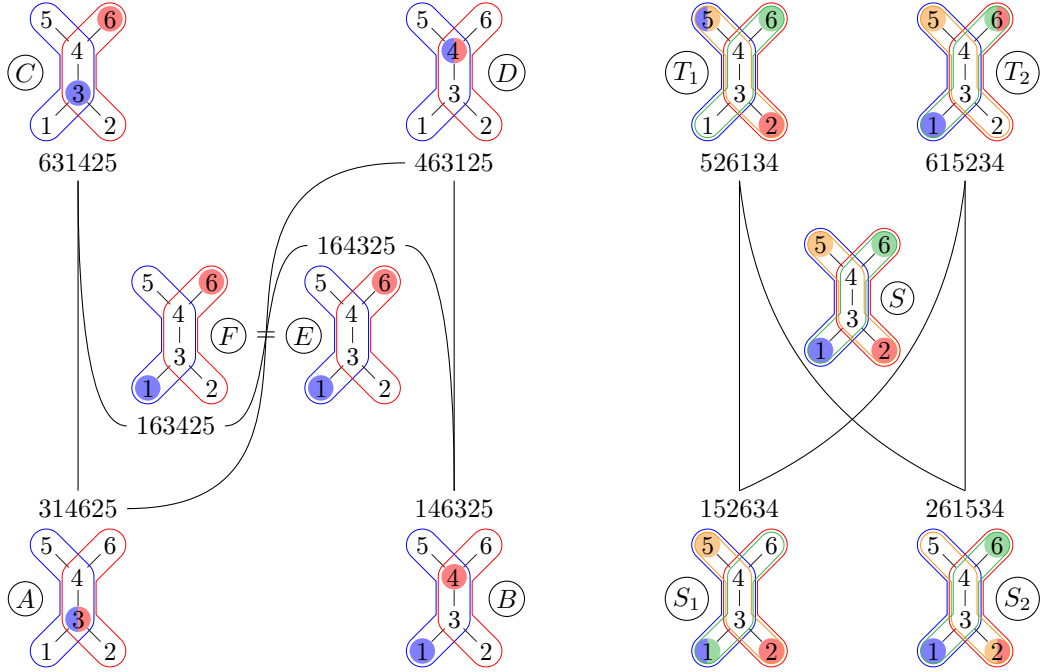
\begin{figure}
	\centerline{
		\input{figures/pathIntersectionClosed}
		\hspace{1.5cm}
		\input{figures/starSparse}
	}
	\caption{The permutations and sourcings of \cref{exm:pathIntersectionClosed} (left) and of \cref{exm:pathStarSparse} (right), illustrating the proofs of \cref{lem:pathIntersectionClosed,lem:pathStarSparse}. The black arrows indicate relations in weak order. The source of each hyperedge is colored in the color of the hyperedge.}
	\label{fig:pathIntersectionClosedStarSparse}
\end{figure}
\end{example}

\begin{lemma}
\label{lem:pathStarSparse}
Let~$\II$ be an intreeval hypergraph of an increasing tree~$T$.
If~$\ASour(\II)$ is a lattice, then $\II$ is star sparse (\cref{def:pathStarSparse}).
\end{lemma}

\begin{proof}
Assume by means of contradiction that~$\ASour(\II)$ is a lattice but that there are~$u \le_T v$ such that~$X(T,u,v,\II)$ contains a cycle~$u_1, v_1, u_2, v_2, \dots, u_p, v_p$, where~$p \ge 2$ (the indices are understood modulo~$p$).
Without loss of generality, we can assume that this cycle is induced.
Note that~$u_i <_T v_j$ for all~$i,j \in [p]$ since we consider increasing trees, and that there is no~$I \in \II$ containing either two distinct~$u_i$'s or two distinct~$v_j$'s since they are incomparable in~$T$.

Let~$W$ be a word formed by the other vertices~$[n] \ssm \{u_0, v_0, \dots, u_p, v_p\}$ in any order.
For~$i \in [p]$, define the permutations
\[
\sigma^i \eqdef u_i v_i u_{i+1} v_{i+1} \dots u_{i+p-1} v_{i+p-1} W
\qquad\text{and}\qquad
\tau^i \eqdef v_i u_{i+1} v_{i+1} \dots u_{i+p-1} v_{i+p-1} u_i W
\]
and the acyclic sourcings~$S^i \eqdef \asour{\sigma^i}$ and~$T^i \eqdef \asour{\tau^i}$ of~$\II$.

Let~$S$ denote the sourcing on~$\II$ defined by
\[
S(I) = 
\begin{cases}
	u_i & \text{if~$u_i \in I$ and~$v_{i-1} \notin I$} \\
	v_i & \text{if~$v_i \in I$ and~$u_{i-1} \notin I$} \\
	w & \text{in all other cases, where~$w$ is the first letter of~$W$ that appears in~$I$}.
\end{cases}
\]
See \cref{exm:pathStarSparse,fig:pathIntersectionClosedStarSparse}\,(right) for an illustration of these permutations and sourcings.

Note that~$S$ is cyclic since it contains the cycle~$u_1, v_1, u_2, v_2, \dots, u_p, v_p$ that we started from.
Observe also that for any~$i \in [p]$ and any~$I \in \II$,
\begin{itemize}
\item $S^i(I) = S(I)$ except if~$\{u_i, v_{i-1}\} \subseteq I$ in which case we have~$S^i(I) = u_i < v_{i-1} = S(I)$,
\item $T^i(I) = S(I)$ except if~$\{u_i, v_i\} \subseteq I$ in which case we have~$S(I) = u_i < v_i = T^i(I)$.
\end{itemize}
In particular, we obtain that~$S^i < S < T^i$ for all~$i \in [p]$.

As~$\ASour(\II)$ is a lattice, we have an acyclic sourcing~$T$ of~$\II$ such that~$S^i \le T \le T^i$ for all~$i \in [p]$.
For any~$I \in \II$, there is~$i \in [p]$ such that~$u_i \notin I$.
For this~$i$, we have~$S(I) = S^i(I) \le T(I) \le T^i(I) = S(I)$, hence~$S(I) = T(I)$.
We conclude that~$S = T$ contradicting the acyclicity of~$T$.
\end{proof}

\begin{example}
\label{exm:pathStarSparse}
We illustrate the construction of the proof of \cref{lem:pathStarSparse} on the last intreeval hypergraph of \cref{fig:intreevalHypergraphs}.
Here, we have~$u = 3$, $v = 4$, $p = 2$ and the cycle in~$X(T, u, v, \II)$ is~$1, 5, 2, 6$.
Therefore, we have the following permutations and ornamentations:
\[
\begin{array}[b]{c|cc}
	i & 1 & 2 \\ 
	\hline
	\sigma^i & 152634 & 261534 \\
	S^i(\blue{1345}) & 1 & 1 \\
	S^i(\orange{2345}) & 5 & 2 \\
	S^i(\red{2346}) & 2 & 2 \\
	S^i(\green{1346}) & 1 & 6 \\
\end{array}
\qquad
\begin{array}[b]{c|c}
	S(\blue{1345}) & 1 \\
	S(\orange{2345}) & 5 \\
	S(\red{2346}) & 2 \\
	S(\green{1346}) & 6 \\
\end{array}
\qquad
\begin{array}[b]{c|cc}
	i & 1 & 2 \\ 
	\hline
	\tau^i & 526134 & 615234 \\
	T^i(\blue{1345}) & 5 & 1 \\
	T^i(\orange{2345}) & 5 & 5 \\
	T^i(\red{2346}) & 2 & 6 \\
	T^i(\green{1346}) & 6 & 6 \\
\end{array}
\]
See \cref{fig:pathIntersectionClosedStarSparse}\,(right).
The reader can check that~$S^i < S < T^i$ for all~$i \in [2]$.
\end{example}

%%%%%%%%%%%%%%%%%%%%%%%%%%%%%%%%%%%%%%

\subsection{Star sparse implies no long cycles}
\label{subsec:shortCycles}

In this section, we observe the following important property of the star sparse intreeval hypergraphs.

\begin{proposition}
\label{prop:pathStarSparse}
Let~$\II$ be a star sparse (\cref{def:pathStarSparse}) intreeval hypergraph of an increasing tree~$T$.
Then any minimal cycle in any sourcing of~$S$ has length~$2$.
\end{proposition}

\begin{proof}
Assume by means of contradiction that there is a sourcing~$S$ of~$\II$ with a minimal cycle~$I_1, \dots, I_p$ where~$p \ge 3$ (the indices are understood modulo~$p$), and let~$w_i \eqdef S(I_i)$ for all~$i \in [p]$.
In other words, we have
\begin{enumerate}[(a)]
\item $w_i \in I_{i+1} \ssm \{w_{i+1}\}$ for all~$i \in [p]$ (because~$I_1, \dots, I_p$ form a cycle in~$S$), and
\item $w_i \in I_j \iff j \in \{i,i+1\}$ for all~$i,j \in [p]$ (because this cycle is minimal).
\end{enumerate}
Moreover, we can assume that there is no sourcing~$S' < S$ with a minimal cycle of length at least~$3$.

\begin{figure}[b]
	\centerline{
	\begin{tikzpicture}
	\coordinate (Z) at (-1,1.5);
	\coordinate (Y) at (-1,1.2);
	\coordinate (A) at (0,0.0);
	\coordinate (B) at (0,0.3);
	\coordinate (D) at (1,1.5);
	\coordinate (C) at (1,1.2);
	\coordinate (E) at (2,0.0);
	\coordinate (F) at (2,0.3);
	\coordinate (H) at (3,1.5);
	\coordinate (G) at (3,1.2);

	\node[brown] at (-.55,1.5) {\scriptsize $w_{2j-2}$};
	\node at (-.55,1.2) {\scriptsize $w'_{2j-2}$};
	\node[red] at (0.45,0.0) {\scriptsize $w_{2j-1}$};
	\node at (0.45,0.3) {\scriptsize $w'_{2j-1}$};
	\node[orange] at (1.3,1.5) {\scriptsize $w_{2j}$};
	\node at (1.3,1.2) {\scriptsize $w'_{2j}$};
	\node[green] at (2.45,0.0) {\scriptsize $w_{2j+1}$};
	\node at (2.45,0.3) {\scriptsize $w'_{2j+1}$};
	\node[blue] at (3.45,1.5) {\scriptsize $w_{2j+2}$};
	\node at (3.45,1.2) {\scriptsize $w'_{2j+2}$};

	\node[red] at (-.6,.6) {\scriptsize $I_{2j-1}$};
	\node[orange] at (.7,.6) {\scriptsize $I_{2j}$};
	\node[green] at (1.4,.6) {\scriptsize $I_{2j+1}$};
	\node[blue] at (2.8,.6) {\scriptsize $I_{2j+2}$};
	
	\node at (-2,.75) {$\cdots$};
	\draw[brown] (-1.5,.75) to [out=10,in=270] ([xshift=-.2] Y) -- ([xshift=-.2] Z);
	\draw[red] ([xshift=-.2] A.center) -- ([xshift=-.2] B.center) to [out=90,in=270] ([xshift=.2] Y) -- ([xshift=.2] Z);
	\draw[orange] ([xshift=.2] A) -- ([xshift=.2] B) to [out=90,in=270] ([xshift=-.2] C) -- ([xshift=-.2] D);
	\draw[green] ([xshift=-.2] E) -- ([xshift=-.2] F) to [out=90,in=270] ([xshift=.2] C) -- ([xshift=.2] D);
	\draw[blue] ([xshift=.2] E) -- ([xshift=.2] F) to [out=90,in=270] ([xshift=-.2] G) -- ([xshift=-.2] H);
	\draw[purple] (3.5,.75) to [out=170,in=270] ([xshift=.2] G) -- ([xshift=.2] H);
	\node at (4,.75) {$\cdots$};

	\node[fill,circle,inner sep=1pt] at (Z) {};
	\node[fill,circle,inner sep=1pt] at (Y) {};
	\node[fill,circle,inner sep=1pt] at (A) {};
	\node[fill,circle,inner sep=1pt] at (B) {};
	\node[fill,circle,inner sep=1pt] at (D) {};
	\node[fill,circle,inner sep=1pt] at (C) {};
	\node[fill,circle,inner sep=1pt] at (E) {};
	\node[fill,circle,inner sep=1pt] at (F) {};
	\node[fill,circle,inner sep=1pt] at (H) {};
	\node[fill,circle,inner sep=1pt] at (G) {};
	
	\end{tikzpicture}
	}
	\caption{Illustration of the notations in the proof of \cref{prop:pathStarSparse}.}
	\label{fig:pathStarSparse}
\end{figure}
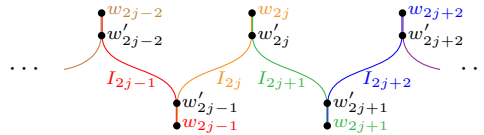

As~$\bigcup_{j \in [p] \ssm \{i,i+1\}} I_j$ is connected by~(a) and contains both~$w_{i-1} \in I_{i-1}$ and~$w_{i+1} \in I_{i+2}$ but not~$w_i$ by~(b), the (not necessarily directed) path joining~$w_{i-1}$ to~$w_{i+1}$ in~$T$ cannot contain~$w_i$.
Hence, the paths from~$w_i$ to~$w_{i-1}$ and to~$w_{i+1}$ respectively share their vertex adjacent to $w_i$, which we denote by~$w'_i$.
Since~$w_i$ and~$w_{i+1}$ both belong to the path~$I_{i+1}$, they are connected by a directed path in~$T$.
We thus obtain that the vertices~$w_1, \dots, w_p$ form an alternating cycle in~$\tc(T)$, and thus that~$p = 2q$.
By symmetry, assume that~$w_i$ has two incoming (resp.~outgoing) arrows in this cycle if~$i$ is even (resp.~odd).
See \cref{fig:pathStarSparse}.

Let~$j \in [q]$.
If~$w'_{2j}$ is not in~$I_{2j+2}$, we can consider the sourcing~$S'$ of~$\II$ defined by~$S'(I_{2j}) = w'_{2j}$ and~$S'(I) = S(I)$ for all~$I \in \II$.
We have~$S' < S$ by definition, and~$I_1, \dots, I_p$ still form a minimal cycle in~$S'$, contradicting the minimality of~$S$.
We conclude that~$w'_{2j}$ belongs to~$I_{2j+2}$, hence to the path from~$w_{2j+1}$ to~$w_{2j+2}$.
This implies in particular that~$w'_{2j} \le w'_{2j+2}$.
We thus obtain that~$w'_2 \le w'_4 \le \dots \le w'_{2q} \le w'_2$, and thus that they all coincide.
We denote this vertex by~$v$.

We obtained that~$I \eqdef \bigcap_{i \in [q]} I_i$ is a non-empty (as it contains~$v$) directed path (as intersection of directed paths in a tree~$T$).
Its final vertex is~$v$, and we denote its initial vertex by~$u$.
For~$j \in [q]$, we denote by~$u_j$ the neighbor of~$u$ in the path from~$w_{2j-1}$ to~$u$, and we denote by~$v_j \eqdef w_{2j}$ which is a neighbor of~$v$.
We obtained that~$u_j \in I_{2j-1} \cap I_{2j}$ while~$v_j \in I_{2j} \cap I_{2j+1}$, hence that~$u_1, v_1, \dots, u_p, v_p$ form a cycle in~$X(T, u, v, \II)$, showing that~$T$ is not star sparse.
\end{proof}

\begin{remark}
\cref{prop:pathStarSparse} fails if the intreeval hypergraph~$\II$ is not star sparse.
For instance, if~$\II = \PP(\Xgraph)$ is the collection of all paths of the graph~$\Xgraph$ of \cref{fig:ornamentationsX}, we have sourcings with minimal cycles of length~$4$.
\end{remark}

%%%%%%%%%%%%%%%%%%%%%%%%%%%%%%%%%%%%%%

\subsection{Sufficient condition}
\label{subsec:sufficient}

To prove the reverse implication of \cref{thm:intreevalLattices}, we will use a quasi-lattice map as defined in~\cite[Rem.~4.14]{BergeronPilaud}.

\begin{definition}[{\cite[Rem.~4.14]{BergeronPilaud}}]
\label{def:quasiLatticeMap}
A \defn{quasi-lattice map} is a surjective order-preserving map~$f$ from a lattice~$L$ to a poset~$P$ such that
\begin{itemize}
\item the fiber of any~$p \in P$ is an interval~$f^{-1}(p) = [\projDown{p}, \projUp{p}]$ of~$L$,
\item $\projDown{p} \le \projUp{q}$ for any~$p \le q$ in~$P$.
\end{itemize}
\end{definition}

\begin{proposition}[{\cite[Rem.~4.14]{BergeronPilaud}}]
\label{prop:quasiLatticeMap}
If~$f : L \to P$ is a quasi-lattice map, then~$P$ is a lattice. 
Moreover, for any~$p,q \in P$,
\[
p \join q = f(\projDown{p} \join \projDown{q})
\qquad\text{and}\qquad
p \meet q = f(\projUp{p} \meet \projUp{q}).
\]
\end{proposition}

\begin{proof}
We repeat the proof here for completeness.
Let~$p,q,r,s \in P$ such that~$p, q \le r, s$.
Then~$\projDown{p}, \projDown{q} \le \projUp{r},  \projUp{s}$.
Hence,~$\projDown{p}, \projDown{q} \le \projDown{p} \join \projDown{q} \le \projUp{r}, \projUp{s}$.
Since~$f$ is order-preserving, we obtain that~$f(\projUp{p}) = p, f(\projUp{q}) = q \le f(\projDown{p} \join \projDown{q}) \le f(\projUp{r}) = r, f(\projUp{s}) = s$.
This implies that~$p$ and~$q$ admit a join~$p \join q \le f(\projDown{p} \join \projDown{q})$.
Moreover, as~$p, q \le p \join q$, we have~$\projDown{p}, \projDown{q} \le \projUp{p \join q}$, hence $\projDown{p} \join \projDown{q} \le \projUp{p \join q}$.
Since~$f$ is order-preserving, we obtain that~$f(\projDown{p} \join \projDown{q}) \le f(\projUp{p \join q}) = p \join q$.
We conclude that~$p \join q = f(\projDown{p} \join \projDown{q})$.
The proof for the meet is symmetric.
\end{proof}

\begin{remark}
Recall that a lattice map is a surjective order-preserving map~$f$ from a lattice~$L$ to a poset~$P$ such that
\begin{itemize}
\item the fiber of any~$p \in P$ is an interval~$f^{-1}(p) = [\projDown{p}, \projUp{p}]$ of~$L$,
\item $\projDown{p} \le \projDown{q}$ and~$\projUp{p} \le \projUp{q}$ for any~$p \le q$ in~$P$.
\end{itemize}
Clearly, a lattice map is a quasi-lattice map, but the reverse is not true.
While a quasi-lattice map already ensures that~$P$ is a lattice, it is much weaker than a lattice map.
\end{remark}

We now consider the following map from the ornamentation lattice~$\Orn(T)$ to the sourcing poset~$\Sour(\II)$.
Note that $\sour{O}^\II$ is just the restriction to~$\II$ of the sourcing~$\sour{O}$ of~$\PP(T)$ defined in \cref{def:Orn2Sour}.

\begin{definition}
\label{def:Orn2SourIntreeval}
Let~$\II$ be an intreeval hypergraph of an increasing tree~$T$.
For an ornamentation~$O$ of~$T$, we define a sourcing~$\sour{O}^\II$ of~$\II$ where the source~$\sour{O}^\II(I)$ for a path~$I \in \II$ from~$u$ to~$v$ in~$T$ is the maximal~$w \in I$ such that~$u \in O(w)$.
\end{definition}

\begin{lemma}
\label{lem:Orn2SourIntreeval1}
The map~$O \mapsto \sour{O}^\II$ is order-preserving.
\end{lemma}

\begin{proof}
Follows from \cref{lem:Orn2Sour1} as $\sour{O}^\II$ is just the restriction to~$\II$ of the sourcing~$\sour{O}$ of~$\PP(T)$ defined in \cref{def:Orn2Sour}.
\end{proof}

\begin{lemma}
\label{lem:Orn2SourIntreeval2}
For any ornamentation~$O$ of~$T$, all cycles of the sourcing~$\sour{O}^\II$ have length at least~$3$.
\end{lemma}

\begin{proof}
Let~$v_1 \eqdef \sour{O}^\II(I_1)$ and~$v_2 \eqdef \sour{O}^\II(I_2)$ for some~$I_1,I_2 \in \II$, and assume by symmetry that~${v_1 \le v_2}$.
If~$v_1 \in I_2 \ssm \{v_2\}$, then~$v_1 \in O(v_2)$, hence~$O(v_1) \subseteq O(v_2)$.
By definition of~$\sour{O}^\II$, we thus obtain that~$\min(I_1) \in O(v_1) \subseteq O(v_2)$.
If~$v_2 \in I_1$, this implies~$v_1 = \sour{O}^\II(I_1) \ge v_2$, hence~${v_1 = v_2}$.
\end{proof}

\begin{corollary}
\label{coro:Orn2SourIntreeval3}
Let~$T$ be an increasing tree, $\II$ be a star sparse intreeval hypergraph of~$T$, and~$O$ be an ornamentation of~$T$.
Then the sourcing~$\sour{O}^\II$ is acyclic.
\end{corollary}

\begin{proof}
Follows from \cref{lem:Orn2SourIntreeval2,prop:pathStarSparse}.
\end{proof}

\begin{definition}
Let~$\II$ be an intreeval hypergraph of an increasing tree~$T$.
For an acyclic sourcing~$S$ of~$\II$ and any~$I \in \II$, we define two ornamentations~$\minorn{S}{I}$ and~$\maxorn{S}{I}$ of~$T$ by
\[
\minorn{S}{I}(v) \eqdef \begin{cases} I \cap \lessin{T}{v} & \text{if } v = S(I) \\ \{v\} & \text{otherwise} \end{cases}
\qquad\text{and}\qquad
\maxorn{S}{I}(v) \eqdef  \begin{cases} \lessin{T}{v} \ssm (I \cap \lessin{T}{S(I)}) & \text{if } S(I) < v \in I \\ \lessin{T}{v} & \text{otherwise} \end{cases}
.
\]
Finally, we define two ornamentations~$\minorn{S}{\II}$ and~$\maxorn{S}{\II}$ of~$T$ by
\[
\minorn{S}{\II} \eqdef \bigJoin_{I \in \II} \minorn{S}{I}
\qquad\text{and}\qquad
\maxorn{S}{\II} \eqdef \bigMeet_{I \in \II} \maxorn{S}{I}
.
\]
\end{definition}

Note that~$\minorn{S}{I}$ and~$\maxorn{S}{I}$ are indeed ornamentations of~$T$ by \cref{exm:ornamentations}\,(iii), hence~$\minorn{S}{\II}$ and~$\maxorn{S}{\II}$ are also ornamentations of~$T$.

\begin{lemma}
\label{lem:Sour2OrnIntreeval1}
For any acyclic sourcing~$S$ of~$\II$ and~$I, J \in \II$, we have~$\minorn{S}{I} \le \maxorn{S}{J}$.
Hence,~$\minorn{S}{\II} \le \maxorn{S}{\II}$.
\end{lemma}

\begin{proof}
Let~$v \in [n]$.
If~$v \ne S(I)$, then~$\minorn{S}{I}(v) = \{v\} \subseteq \maxorn{S}{J}(v)$.
If~$v \notin J$ or~$S(J) \ge v$, then ${\minorn{S}{I}(v) \subseteq I \cap [v] \subseteq \lessin{T}{v} = \maxorn{S}{J}(v)}$.
Finally, assume that~$S(I) = v$ and~$S(J) < v \in J$.
If~$S(J) \in I$, then~$S(J) \in I \ssm \{S(I)\}$ while~$S(I) \in J \ssm \{S(J)\}$ contradicting the acyclicity of~$S$.
Hence, we obtain~$(J \cap [S(J)]) \cap I = \varnothing$ and thus~$\minorn{S}{I}(v) = I \cap [v] \subseteq \lessin{T}{v} \ssm (J \cap [S(J)])= \maxorn{S}{J}(v)$.
We conclude that~$\minorn{S}{I}(v) \subseteq \maxorn{S}{J}(v)$ for any~$v \in [n]$, hence that~$\minorn{S}{I} \le \maxorn{S}{J}$.
\end{proof}

\begin{lemma}
\label{lem:Sour2OrnIntreeval2}
For any acyclic sourcing~$S$ of~$\II$, the ornamentations~$O$ of~$T$ such that~$\sour{O}^\II = S$ form an interval of the ornamentation lattice~$\Orn(T)$ with minimal element~$\minorn{S}{\II}$ and maximal element~$\maxorn{S}{\II}$.
\end{lemma}

\begin{proof}
Consider first an ornamentation~$O$ of~$T$ such that~$\sour{O}^\II = S$ and let~$v \in [n]$ and~$I \in \II$. Then
\begin{itemize}
\item If~$S(I) = v$, we have~$\min(I) \in O(v)$ by definition of~$\sour{O}^\II = S$, and thus~$I \cap [v] \subseteq O(v)$ since~$T$ is a tree.
We thus obtain that~$\minorn{S}{I}(v) \subseteq O(v)$.
\item If~$S(I) < v \in I$, we have~$\min(I) \in O(S(I)) \ssm O(v)$ by definition of~$\sour{O}^\II = S$, hence~$(I \cap [S(I)]) \cap O(v) = \varnothing$ since~$O$ is an ornamentation.
We thus obtain that~$O(v) \subseteq \maxorn{S}{I}(v)$.
\end{itemize}
In all cases, $\minorn{S}{I}(v) \subseteq O(v) \subseteq \maxorn{S}{I}(v)$ for any~$v \in [n]$ and~$I \in \II$.
We conclude that~$\minorn{S}{I} \le O \le \maxorn{S}{I}$ for any~$I \in \II$.

Conversely, consider an ornamentation~$O$ of~$T$ such that~$\minorn{S}{I} \le O \le \maxorn{S}{I}$ and let~$I \in \II$ and~$v \eqdef S(I)$. Then
\begin{itemize}
\item As~$S(I) = v$, we have~$I \cap [v] \subseteq \minorn{S}{I}(v) \subseteq \minorn{S}{I} \subseteq O(v)$, hence~$v \le \sour{O}^\II(I)$ by definition~of~$\sour{O}^\II$.
\item Conversely, for any~$v < w \in I$, we have~$S(I) = v < w \in I$ and so~$O(w) \subseteq \maxorn{S}{\II}(w) \subseteq \maxorn{S}{I}(w) \subseteq \lessin{T}{w} \ssm (I \cap [v])$. Hence, $\min(I) \notin O(w)$, and so~$\sour{O}^\II(I) \ne w$. We thus obtain that~$\sour{O}^\II(I) \le v$
\end{itemize}
We conclude that~$\sour{O}^\II(I) = S(I)$ for all~$I \in \II$, hence that~$\sour{O}^\II = S$.
\end{proof}

\begin{lemma}
\label{lem:Sour2OrnIntreeval3}
Let~$\II$ be a path intersection closed (\cref{def:pathIntersectionClose}) intreeval hypergraph of an increasing tree~$T$.
For any acyclic sourcings~$S_1$ and~$S_2$ of~$\II$, if~$S_1 \le S_2$ then~$\minorn{S_1}{\II} \le \maxorn{S_2}{\II}$.
\end{lemma}

\begin{proof}
Assume by means of contradiction that~$S_1 \le S_2$ and~$\minorn{S_1}{\II} \not\le \maxorn{S_2}{\II}$.
Since~$\minorn{S_1}{\II} \not\le \maxorn{S_2}{\II}$, there exist~$I, J \in \II$ such that~$\minorn{S_1}{I} \not\le \maxorn{S_2}{J}$.
This implies~$\min(I) \le S_2(J) < S_1(I) \le \max(J)$.
Hence, $S_1(I)$ and~$S_2(J)$ are distinct and both belong to~$I \cap J$.
Moreover, as~$S_1 \le S_2$, we have~$S_1(J) \le S_2(J) < S_1(I) \le S_2(I)$.
Since~$S_1$ and~$S_2$ are both acyclic, we deduce that~$\min(J) \le S_1(J) < \min(I \cap J)$ while~$\max(I) \ge S_2(I) > \max(I \cap J)$.
As~$\II$ is path intersection closed, there exists~$K \in \II$ such that~$I \cap J \subseteq K \subseteq [\min(I) , \max(J)]_T$.
As~$S_1(I)$ and~$S_2(J)$ both belong to~$I \cap J \subseteq K$, and~$S_1$ and~$S_2$ are both acyclic, we obtain that~$S_2(K) \le S_2(J) < S_1(I) \le S_1(K)$, which contradicts~$S_1 \le S_2$.
\end{proof}

\begin{proposition}
\label{prop:quasiLatticeMapIntreevalLattices}
Let~$\II$ be a path intersection closed (\cref{def:pathIntersectionClose}) and star sparse (\cref{def:pathStarSparse}) intreeval hypergraph of an increasing tree~$T$.
Then the map~$O \mapsto \sour{O}^\II$ of \cref{def:Orn2SourIntreeval} is a quasi-lattice map.
Hence, the acyclic sourcing poset~$\ASour(\II)$ is a lattice.
\end{proposition}

\begin{proof}
The map~$O \mapsto \sour{O}^\II$ satisfies the conditions of \cref{def:quasiLatticeMap} by \cref{lem:Orn2SourIntreeval1}, \cref{coro:Orn2SourIntreeval3} and~\cref{lem:Sour2OrnIntreeval1,,lem:Sour2OrnIntreeval2,,lem:Sour2OrnIntreeval3}.
The last statement thus follows from~\cref{prop:quasiLatticeMap}.
\end{proof}

%%%%%%%%%%%%%%%%%%%%%%%%%%%%%%%%%%%%%%

\subsection{Expression for the meet and join}
\label{subsec:formulaMeetJoin}

Finally, when~$\ASour(\II)$ is a lattice, we can obtain the following formula to compute its joins, mimicking the formula of~\cite[Prop.~4.15]{BergeronPilaud} when~$T$ is a path.
A similar formula holds for meets by symmetry.

\begin{proposition}
\label{prop:expressionJoin}
For a path intersection closed and star sparse intreeval hypergraph~$\II$ of an increasing tree~$T$, and for any acyclic sourcings~$S_1, \dots, S_q$ of~$\II$ and any~$I \in \II$, we have
\[
\bigg( \bigJoin_{p \in [q]} S_p \bigg)(I) = \min \bigg( I \ssm \Big( \bigcup_{p \in [q]} \bigcup_{\substack{J \in \II : \\ S_p(J) \in I}} \set{v \in J}{v < S_p(J)} \Big) \bigg).
\]
\end{proposition}

\begin{proof}
As~$\II$ is path intersection closed and star sparse, the acyclic sourcing poset~$\ASour(\II)$ is a lattice by~\cref{prop:quasiLatticeMapIntreevalLattices}, hence the join~$\bigJoin_{p \in [q]} S_p$ exists and we just need to prove that it coincides with the sourcing~$S$ defined by the right-hand-side of the formula of the statement.
We first prove that $S$ is acyclic and then that it is indeed the least upper bound of $\{S_p\}_{p \in [q]}$.

Assume by means of contradiction that~$S$ is cyclic.
By \cref{prop:pathStarSparse}, there is~$I,J \in \II$ such that~$S(I) \in J \ssm \{S(J)\}$ and~$S(J) \in I \ssm \{S(I)\}$.
Assume by symmetry that~$S(I) < S(J)$.
As~${S(I) \in J}$ and~$S(I) < S(J)$, applying the formula for~$S(J)$, we obtain that there exists~${p \in [q]}$ and~$K \in \II$ such that~$S_p(K) \in J$ and~$S(I) \in K$ and~$S(I) < S_p(K) \le S(J)$.
As~$I$ is a path containing~$S(I)$ and~$S(J)$, it contains also~$S_p(K)$.
We obtain that~$S_p(K) \in I$, $S(I) \in K$ and~${S(I) < S_p(K)}$ contradicting our definition of~$S$.
We conclude that~$S$ is acyclic.

For any~$p \in [q]$, we have~$S_p(I) \in I$ and so~$S(I) \ge S_p(I)$ for all~$I \in \II$, hence~$S \ge S_p$.
Consider an acyclic sourcing~$S'$ of~$\II$ such that~$S' \ge S_p$ for all~$p \in [q]$,
and assume by means of contradiction that $S'(I) < S(I)$ for some $I \in \II$.
Then, there is $p \in [q]$ and $J \in \II$ such that~$S_p(J) \in I$ and $S'(I) \in \set{v \in I \cap J}{v < S_p(J)}$.
In particular, $S_p(J)$ and $S'(I)$ are two different elements of~$I \cap J$, so $|I \cap J| > 1$.
The acyclicity of $S_p$ implies that $S_p(I) \notin J$, thus~${\min(I) < \min(I \cap J)}$~as~${S_p(I) \leq S'(I) \in J}$.
Similarly, the acyclicity of $S'$ implies that~${S'(J) \notin I}$, thus~${\max(I \cap J) < \max(J)}$ as $S'(J) \geq S_p(J) \in I$.
By path intersection closedness, there is~$K \in \II$ such that $I \cap J \subseteq K \subseteq [\min(J),\max(I)]_T$.
The acyclicity of $S_p$ now implies that $S_p(J) \leq S_p(K)$, thus $S'(I) < S_p(J) \leq S_p(K) \leq S'(K)$.
This contradicts the acyclicity of $S'$ as~${S'(I) \in I \cap J \subseteq K}$.
\end{proof}

%%%%%%%%%%%%%%%%%%%%%%%%%%%%%%%%%%%%%%
%%%%%%%%%%%%%%%%%%%%%%%%%%%%%%%%%%%%%%
%%%%%%%%%%%%%%%%%%%%%%%%%%%%%%%%%%%%%%

\section*{Acknowledgments}

This work was initiated at the Banff workshop ``Lattice Theory'' in January 2025.
We are grateful to the organizers (Emily Barnard, Cesar Ceballos, Colin Defant, Osamu Iyama, and Nathan Williams) and to all participants for the friendly and inspiring atmosphere.
We particularly thank Eleni Tzanaki for several conversations on related topics on hypergraphical polytopes.

%%%%%%%%%%%%%%%%%%%%%%%%%%%%%%%%%%%%%%
%%%%%%%%%%%%%%%%%%%%%%%%%%%%%%%%%%%%%%
%%%%%%%%%%%%%%%%%%%%%%%%%%%%%%%%%%%%%%

\bibliographystyle{alpha}
\bibliography{intreevalLattices}
\label{sec:biblio}

%%%%%%%%%%%%%%%%%%%%%%%%%%%%%%%%%%%%%%

\end{document}

%% file: figures/diagram.tex
\begin{tikzpicture}[scale=1.1]

	\node(R) at (-5,0) {$\Reori(\tc(D))$};
	\node(DR) at (-5,-0.35) {\tiny{Def.~\ref{def:Reori}}};
	\node(S) at (0,0) {$\Sour(\PP(D))$};
	\node(DS) at (0,-0.35) {\tiny{Def.~\ref{def:Sour}}};
	\node(O) at (5,0) {$\Orn(D)$};
	\node(DO) at (5,-0.35) {\tiny{Def.~\ref{def:Orn}}};

	\draw[thick,green] (R.north) edge [>={Stealth[scale=0.65,width=6pt]},-{>>},bend left=35] node [above]{\textcolor{black}{\small{$\orn{R}$} }} node [below]{\textcolor{black}{\tiny{\begin{tabular}{c} Def.~\ref{def:Reori2Orn} \\ \ref{lem:Reori2Orn1}, \ref{lem:Reori2Orn2}, \ref{prop:Reori2Orn}, \ref{prop:ReoriTC2Orn}, \ref{prop:Reori2OrnMeetSemilaticeTC} \end{tabular}}}}  (O.north);
	\draw[thick,green] (DO.south) edge [{Hooks[left]}-stealth,bend left=35] node [below]{\textcolor{black}{ \tiny{\begin{tabular}{c}Def.~\ref{def:Orn2Reori} \\ \ref{lem:Orn2Reori1}, \ref{prop:ReoriTC2Orn}, \ref{prop:Reori2OrnMinTC} \end{tabular}}}} node [above]{\textcolor{black}{\small{$\reori{O}$} }}  (DR.south);
	\draw[thick,green] (S.west) edge [-stealth] node [above]{\textcolor{black}{\small{$\reori{S}$} }} node [below]{\textcolor{black}{ \tiny{\begin{tabular}{c} Def.~\ref{def:Sour2Reori} \\ \ref{lem:Sour2Reori} \end{tabular}}}}  (R.east);
	\draw[thick,green] (S.north east) edge [>={Stealth[scale=0.65,width=6pt]},-{>>},bend left=33] node [above]{\textcolor{black}{\small{$\orn{S}$} }} node [below]{\textcolor{black}{ \tiny{\begin{tabular}{c} Def.~\ref{def:Sour2Orn} \\ \ref{lem:Sour2Orn2}, \ref{lem:Sour2Orn3}, \ref{prop:Sour2Orn} \end{tabular}}}}  (O.north west);	
	\draw[thick,green] (O.south west) edge [{Hooks[left]}-stealth,bend left=33] node [below]{\textcolor{black}{ \tiny{\begin{tabular}{c} Def.~\ref{def:Orn2Sour} \\ \ref{lem:Orn2Sour1}, \ref{lem:Orn2Sour2} \end{tabular}}}} node [above]{\textcolor{black}{\small{$\sour{O}$} }}  (S.south east);

	\node at (0,.8) {\textbf{\rotatebox{180}{\huge{$\circlearrowright$}}}}; \node at (-.5,.8) {\tiny{\ref{lem:Sour2Orn1}}};
	\node at (0,-1) {\textbf{\huge{$\circlearrowleft$}}}; \node at (-.5,-1) {\tiny{\ref{lem:Orn2Sour4}}};
	\node at (5.8,-.1) {\textbf{\rotatebox{90}{{\Large$\curvearrowright$}}}}; \node at (6.2,-.1) {\tiny{\ref{lem:Orn2Reori2}}};
	\node at (2.5,-.1) {\textbf{\rotatebox{90}{{\Large$\curvearrowright$}}}}; \node at (2.1,-.1) {\tiny{\ref{lem:Orn2Sour3}}};
	
%	\node(subposet1)[rotate=90] at (-5,-3.2) {\LARGE{$\subseteq$}};
%	\node(subposet2)[rotate=90] at (0,-3.2) {\LARGE{$\subseteq$}};
%	\node(subposet3)[rotate=90] at (5,-3.2) {\LARGE{$\subseteq$}};
%	\node(subposet3) at (5,-2.2) {\small{subposet}};
%	\node(subposet3) at (0,-2.2) {\small{subposet}};
%	\node(subposet3) at (-5,-2.2) {\small{subposet}};

	\node(Sym) at (-8,-6.5) {$\fS_n$};
	\node(AR) at (-5,-6.5) {$\AReori(\tc(D))$};
	\node(DAR) at (-5,-6.85) {\tiny{Def.~\ref{def:AReori}}};
	\node(AS) at (0,-6.5) {$\ASour(\PP(D))$};
	\node(DAS) at (0,-6.85) {\tiny{Def.~\ref{def:ASour}}};
	\node(AO) at (5,-6.5) {$\AOrn(D)$};
	\node(DAO) at (5,-6.85) {\tiny{Def.~\ref{def:AOrn}}};
	\node(Phantom) at (8,-6.5) {\phantom{$\fS_n$}};
		
	\draw[thick,green] (Sym.east) edge [>={Stealth[scale=0.65,width=6pt]},-{>>}] node [below]{\textcolor{black}{ \tiny{\begin{tabular}{c} Def.~\ref{def:Perm2AReori} \\ \ref{prop:Perm2AReori} \end{tabular}}}} node [above]{\small{\textcolor{black}{$\areori{\pi}$}}} (AR.west);	
	\draw[thick,green] (AR.north) edge [>={Stealth[scale=0.65,width=6pt]},-{>>},bend left=35] node [below]{\textcolor{black}{ \tiny{Def.~\ref{def:moreNotations}}}} node [above]{\small{\textcolor{black}{$\aorn{R}=\orn{R}$}}} (AO.north);	
	\draw[thick,red] (DAO.south) edge [{Hooks[left]}-stealth,bend left=35] node [below]{\textcolor{black}{ \tiny{Def.~\ref{def:moreNotations}}}} node [above]{\textcolor{black}{\small{$\areori{O}$}}} (DAR.south);
	\draw[thick,green] (AR.north east) edge [>={Stealth[scale=0.65,width=6pt]},-{>>},bend left=33] node [below]{\textcolor{black}{ \tiny{\begin{tabular}{c} Def.~\ref{def:AReori2ASour} \\ \ref{lem:AReori2ASour1}, \ref{lem:AReori2ASour2}, \ref{lem:AReori2ASour3}, \ref{prop:AReori2ASour} \end{tabular}}}} node [above]{\textcolor{black}{\small{$\asour{R}$} }}  (AS.north west);
	\draw[thick, red] (AS.south west) edge [{Hooks[left]}-stealth,bend left=33] node [above]{\textcolor{black}{\small{$\areori{S}$ }}} node [below]{\textcolor{black}{ \tiny{\begin{tabular}{c} Def.~\ref{def:ASour2AReori} \\ \ref{lem:ASour2AReori1}, \ref{rem:ASour2AReori} \end{tabular}}}} (AR.south east);
	\draw[thick,green] (AS.north east) edge [>={Stealth[scale=0.65,width=6pt]},{Hooks[left]}-{>>},bend left=33] node [below] {\textcolor{black}{\tiny{\begin{tabular}{c} Def.~\ref{def:moreNotations} \\ \ref{lem:ASour2AOrn1}, \ref{prop:ASour2AOrn} \end{tabular}}}} node [above]{\textcolor{black}{\small{$\aorn{S}=\orn{S}$}}} (AO.north west);
	\draw[thick,green] (AO.south west) edge [>={Stealth[scale=0.65,width=6pt]},{Hooks[left]}-{>>},bend left=33] node [below]{\textcolor{black}{ \tiny{Def.~\ref{def:moreNotations}}}} node [above]{\textcolor{black}{\small{$\asour{O}$}}} (AS.south east);

	\node at (0,-5.5) {\textbf{\rotatebox{180}{\huge{$\circlearrowright$}}}}; \node at (-.5,-5.5) {\tiny{\ref{lem:AReori2ASour2AOrn}}};
	\node at (0,-7.5) {\textbf{\huge{$\circlearrowleft$}}}; \node at (-.5,-7.5) {\tiny{\ref{def:moreNotations}}};
	\node at (-2.5,-6.6) {\textbf{\rotatebox{90}{\Large{$\curvearrowright$}}}}; \node at (-2.1,-6.6) {\tiny{\ref{lem:ASour2AReori2}}};
	\node at (2.5,-6.6) {\textbf{\rotatebox{270}{\huge{$\circlearrowright$}}}}; \node at (2,-6.6) {\tiny{\ref{def:moreNotations}}};

	\draw[gray,thick] (-5.1,-6.2) edge [{Hooks[right]}-stealth,dashed] node [left]{\rotatebox{90}{\LARGE$\subseteq$}} (-5.1,-.53);
	\draw[gray,thick] (.5,-6.2) edge [{Hooks[right]}-stealth,dashed,bend right=33] node [left]{\rotatebox{90}{\LARGE$\subseteq$}} (.5,-.53);
	\draw[gray,thick] (5.3,-6.2) edge [{Hooks[right]}-stealth,dashed] node [left]{\rotatebox{90}{\LARGE$\subseteq$}} (5.3,-.53);

\end{tikzpicture}

%% file: figures/pathIntersectionClosed.tex
\begin{tikzpicture}
	\node (A) at (-2.5,-3.5) {\includegraphics[scale=.8]{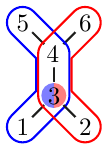}};
	\node (B) at (2.5,-3.5) {\includegraphics[scale=.8]{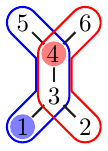}};
	\node (C) at (-2.5,3.5) {\includegraphics[scale=.8]{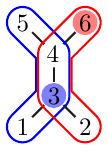}};
	\node (D) at (2.5,3.5) {\includegraphics[scale=.8]{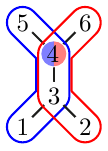}};
	\node (EF) at (1.15,0) {\includegraphics[scale=.8]{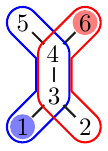}};
	\node (EF) at (-1.15,0) {\includegraphics[scale=.8]{pathIntersectionClosedEF}};
	\node (piA) at (-2.5,-2.3) {$314625$};
	\node (piB) at (2.5,-2.3) {$146325$};
	\node (piC) at (-2.5,2.3) {$631425$};
	\node (piD) at (2.5,2.3) {$463125$};
	\node (piE) at (1.2,1.2) {$164325$};
	\node (piF) at (-1.2,-1.2) {$163425$};
	\node[draw,circle,inner sep=1] at (-3.2,-3.5) {$A$};
	\node[draw,circle,inner sep=1] at (3.2,-3.5) {$B$};
	\node[draw,circle,inner sep=1] at (-3.2,3.5) {$C$};
	\node[draw,circle,inner sep=1] at (3.2,3.5) {$D$};
	\node[draw,circle,inner sep=1] at (.5,0) {$E$};
	\node[draw,circle,inner sep=1] at (-.5,0) {$F$};
	\node at (0,0) {$=$};
	\draw (piA.north) -- (piC.south);
	\draw (piA.east) edge [out=0, in=180, looseness=1.5] (piD.west);
	\draw (piB.north) -- (piD.south);
	\draw (piB.north) edge [out=90, in=0, looseness=.5] (piE.east);
	\draw (piF.west) edge [out=180, in=270, looseness=.5]  (piC.south);
	\draw (piF.east) edge [out=0, in=180, looseness=.7] (piE.west);
\end{tikzpicture}

%% file: figures/starSparse.tex
\begin{tikzpicture}
	\node (S1) at (-1.5,-3.5) {\includegraphics[scale=.8]{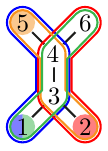}};
	\node (S2) at (1.5,-3.5) {\includegraphics[scale=.8]{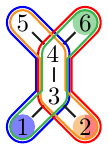}};
	\node (T1) at (-1.5,3.5) {\includegraphics[scale=.8]{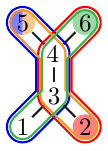}};
	\node (T2) at (1.5,3.5) {\includegraphics[scale=.8]{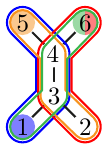}};
	\node (S) at (0,.5) {\includegraphics[scale=.8]{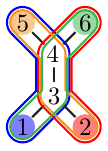}};
	\node (sigma1) at (-1.5,-2.3) {$152634$};
	\node (sigma2) at (1.5,-2.3) {$261534$};
	\node (tau1) at (-1.5,2.3) {$526134$};
	\node (tau2) at (1.5,2.3) {$615234$};
	\node[draw,circle,inner sep=1] at (-2.2,-3.5) {$S_1$};
	\node[draw,circle,inner sep=1] at (2.2,-3.5) {$S_2$};
	\node[draw,circle,inner sep=1] at (-2.2,3.5) {$T_1$};
	\node[draw,circle,inner sep=1] at (2.2,3.5) {$T_2$};
	\node[draw,circle,inner sep=1] at (.6,.5) {$S$};
%	\draw (sigma1.north) -- (S.south west);
%	\draw (sigma2.north) -- (S.south east);
%	\draw (S.north west) -- (tau1.south);
%	\draw (S.north east) -- (tau2.south);
	\draw (sigma1.north) -- (tau1.south);
	\draw (sigma1.north) edge [bend right] (tau2.south);
	\draw (sigma2.north) edge [bend left] (tau1.south);
	\draw (sigma2.north) -- (tau2.south);
\end{tikzpicture}